\documentclass[a4paper]{article}

\usepackage[latin1]{inputenc}
\usepackage[T1]{fontenc}
\usepackage{amssymb,amsmath,amsthm,amscd,mathrsfs}
\usepackage[all]{xy}
\usepackage{color}
\usepackage[margin=3cm]{geometry}
\usepackage{listings}
\usepackage{color}
\usepackage{tikz}
\usetikzlibrary{trees}

\definecolor{dkgreen}{rgb}{0,0.6,0}
\definecolor{gray}{rgb}{0.5,0.5,0.5}
\definecolor{mauve}{rgb}{0.58,0,0.82}

\newtheorem{thm}{Theorem}[section]
\newtheorem{defi}[thm]{Definition}
\newtheorem{prop}[thm]{Proposition}
\newtheorem{lemme}[thm]{Lemma}
\newtheorem{cor}[thm]{Corollary}

\newtheorem{defipro}[thm]{Definition-Proposition}
\newtheorem{nota}[thm]{Notation}
\newtheorem{hypo}[thm]{Hypothesis}
\newtheorem{erra}[thm]{Erratum}

\theoremstyle{remark}
\newtheorem{remark}[thm]{Remark}
\newtheorem{rmk}[thm]{Remark}
\newtheorem{ex}[thm]{Example}

\DeclareMathOperator{\Z}{\mathbb{Z}}

\DeclareMathOperator{\N}{\mathbb{N}}

\DeclareMathOperator{\codim}{codim}

\DeclareMathOperator{\Pic}{Pic}
\DeclareMathOperator{\rk}{rk}

\DeclareMathOperator{\Fix}{Fix}
\DeclareMathOperator{\SpFix}{SpFix}
\DeclareMathOperator{\Supp}{Supp}

\DeclareMathOperator{\Aut}{Aut}

\DeclareMathOperator{\id}{id}

\DeclareMathOperator{\Prym}{Prym}

\DeclareMathOperator{\Ker}{Ker}

\DeclareMathOperator{\diag}{diag}
\DeclareMathOperator{\Sing}{Sing}

\DeclareMathOperator{\GL}{GL}
\DeclareMathOperator{\pr}{pr}
\DeclareMathOperator{\Q}{\mathbb{Q}}

\DeclareMathOperator{\C}{\mathbb{C}}

\DeclareMathOperator{\Sy}{\mathfrak{S}}
\DeclareMathOperator{\Ay}{\mathfrak{A}}
\DeclareMathOperator{\Ab}{Ab}
\DeclareMathOperator{\by}{\mathfrak{b}}

\DeclareMathOperator{\SU}{SU}
\DeclareMathOperator{\SL}{SL}

\DeclareMathOperator{\Out}{Out}
\DeclareMathOperator{\AGL}{AGL}

\DeclareMathOperator{\tr}{tr}
\DeclareMathOperator{\inv}{inv}
\DeclareMathOperator{\Agam}{A\Gamma L}

\DeclareMathOperator{\SmallGroup}{\textbf{SmallGroup}}

\newcommand{\eq}[1][r]
{\ar@<-3pt>@{-}[#1]
\ar@<-1pt>@{}[#1]|<{}="gauche"
\ar@<+0pt>@{}[#1]|-{}="milieu"
\ar@<+1pt>@{}[#1]|>{}="droite"
\ar@/^2pt/@{-}"gauche";"milieu"
\ar@/_2pt/@{-}"milieu";"droite"}

\newcommand{\incl}[1][r]
  {\ar@<-0.2pc>@{^(-}[#1] \ar@<+0.2pc>@{-}[#1]}

\begin{document}
\title{\bf Thirty-three deformation classes of compact hyperkähler orbifolds}
\author{Grégoire \textsc{Menet}} 

\maketitle
\begin{center}
\emph{In honor of Professor Dimitri Markushevich for his 60 (+2) birthday}
\end{center}
\begin{abstract}
As their smooth analogue the irreducible symplectic varieties appear as elementary bricks in the generalizations of the Beauville--Bogomolov decomposition theorem (\cite{Bakker}, \cite{Campana}).
Generalizing the Fujiki construction \cite{Fujiki}, we investigate the irreducible symplectic varieties with simply connected smooth locus that can be obtained as terminalizations of quotients of the product $S^n$, where $S$ is a K3 surface.
In dimension 4, we compute the singularities for 29 orbifolds examples which appear to not be deformation equivalent. We also provide 4 additional orbifolds examples in dimension 6.
\end{abstract}
\section{Introduction}
\subsection{Motivation and main results}
Irreducible symplectic orbifolds (or compact hyperkähler orbifolds) can be seen as the simplest generalization of compact hyperkähler manifolds;
they are orbifolds with a simply connected smooth locus carrying a unique (up to scalar) non-degenerate holomorphic 2-form and they have a singular locus in codimension 4. 
In the last few years, the theory of irreducible symplectic (IHS) orbifolds have been developed especially with the generalization of the global Torelli theorem \cite{Lol} and of the Kähler cone knowledge \cite{Ulrike}. In this context, it appears natural to search for examples of such orbifolds. In a more general perspective, the irreducible symplectic orbifolds are a particular case of the irreducible symplectic (IHS) varieties which are getting attention in the last years especially because of recent important generalizations as the Beauville--Bogomolov decomposition theorem \cite{Bakker} and the global Torelli theorem \cite{Bakker2}.

In order to produce examples of irreducible symplectic orbifolds a huge amount of combinations seems possible to explore. Let $X$ be a projective symplectic manifold endowed with a finite automorphism group $G$ such that $X/G$ has a unique (up to scalar) non-degenerate holomorphic 2-form on its smooth locus. Let $Y\rightarrow X/G$ be a terminalization (see Section \ref{Terminalsection} for definition and existence). Let $Y_{reg}$ be the smooth locus of $Y$. If we assume that $\pi_1(Y_{reg})=0$, then $Y$ is an irreducible symplectic variety with singular locus of codimension 4 (see Definition \ref{Defi0} and Proposition \ref{basisi}). Therefore, it is natural to ask which conditions we need on $X$ and $G$ in order to have $\pi_1(Y_{reg})=0$. 
\begin{prop}\label{mainprop}
Assume that $X$ is simply connected then $\pi_1(Y_{reg})=0$ if and only if $G$ is generated by automorphisms with fixed locus of codimension 2.
\end{prop}
This proposition corresponds to Proposition \ref{fonda} in Section \ref{fundasection}.

Already in dimension 4, several cases of the previous simple construction can be studied. The variety $X$ can be:
\begin{itemize}
\item[(1)]
the product of two K3 surfaces;
\item[(2)]
a hyperkähler manifold of $K3^{[2]}$-type (deformation equivalent to a Hilbert scheme of 2 points on a K3 surface);
\item[(3)]
a hyperkähler fourfold of Kummer type (deformation equivalent to a generalized Kummer fourfold);
\item[(4)]
a complex torus of dimension 4 (however Proposition \ref{mainprop} do not apply in this case).
\end{itemize}
Forty years ago, Fujiki partially investigated the case (1) in \cite[Section 13]{Fujiki}. The main goal of this paper is to generalize and complete the Fujiki investigations. 

Fujiki has provided his construction in dimension 4. It can be generalized for every dimension as follows. 
Let $S$ be a projective K3 surface and $G$ a finite non-trivial symplectic automorphism group. Let $\theta:G\rightarrow G$ be an involution (which can be the identity). Let $n\geq 2$ be an integer.
We set $j_{\theta}:G\hookrightarrow \Aut(S^n)$ defined by
\begin{equation} 
j_{\theta}(g)(x_1,x_2,x_3,...,x_n)=(g(x_1),\theta(g)(x_2),x_3,...,x_n).
\label{jtheta}
\end{equation}
This means that $j_{\theta}(g)$ is given by the diagonal action with $g$ on the first factor, $\theta(g)$ on the second and the identity on the other factors. 
The permutation group $\mathfrak{S}_n$ also acts naturally on $S^n$. 
\begin{defi}\label{Fujikiconstruc}
Let 
$$Y\rightarrow S^n/\left\langle j_{\theta}(G),\mathfrak{S}_n\right\rangle$$
be a Kähler terminalization of the quotient $S^n/\left\langle j_{\theta}(G),\mathfrak{S}_n\right\rangle$.
Such a variety $Y$ is called a \emph{Fujiki variety}, moreover we denote $S(G)_{\theta}^{[n]}:=Y$. This construction of varieties is called the \emph{Fujiki construction}.
\end{defi}
\begin{rmk}
Note that a projective (and thus Kähler) terminalization of $S^n/\left\langle j_{\theta}(G),\mathfrak{S}_n\right\rangle$ always exists from \cite{Birkar}.
\end{rmk}
\begin{rmk}
Note that all the Fujiki varieties are primitively symplectic varieties with singularities in codimension 4 (see Definition \ref{Primitivy} and Proposition \ref{primimi}).
\end{rmk}
\begin{rmk}
Since $S(G)_{\theta}^{[n]}$ necessarily contains at least one exceptional divisor (see Proposition \ref{terminal}),
the Fujiki varieties consist in a strict sub-space of their space of deformations; therefore, these deformations are generically new objects.
\end{rmk}
\begin{rmk}
In the previous definition of Fujiki varieties, we did not emphasize the choice of the terminalization for $S^n/\left\langle j_{\theta}(G),\mathfrak{S}_n\right\rangle$; this is because this choice does not change the deformation class of the variety as explained by Proposition \ref{bimero}. 

The deformation class of $S(G)_{\theta}^{[n]}$ only depends of $\theta$, $n$ and the deformation class of $(S,G)$. To be more precise, if $(S,G)$ and $(S',G')$ are deformation equivalent, given an involution $\theta$ on $G$, we can find an involution $\theta'$ on $G'$ such that $S(G)^{[n]}_{\theta}$ and $S(G')^{[n]}_{\theta'}$ are deformation equivalent. The deformation classes of the couples $(S,G)$ are classified in \cite{xiao}.
\end{rmk}
It is natural to ask when $S(G)_{\theta}^{[n]}$ is an irreducible symplectic variety (see Definition \ref{Defi0}). In particular, note that a primitively symplectic variety with rational singularities and simply connected smooth locus is irreducible symplectic (see Proposition \ref{basisi}).
In order to study this property, we introduce the following definition.
\begin{defi}\label{validintro}
An involution $\theta:G\rightarrow G$ is said \emph{valid} if there exist a family of generators $(g_1,...,g_k)$ of $G$ such that $\theta(g_i)=g_i^{-1}$ for all $i\in\left\{1,...,k\right\}$.
\end{defi}
As a consequence of Proposition \ref{mainprop}, we have the following Corollaries. We distinguish the case $n=2$ and the case $n\geq3$.
\begin{cor}\label{corirr}
A Fujiki variety $S(G)_{\theta}^{[2]}$ is an irreducible symplectic variety with simply connected smooth locus if and only if $\theta$ is valid.
\end{cor}
\begin{cor}\label{corirr2}
Let $n\geq3$ be an integer. All the Fujiki varieties $S(G)_{\theta}^{[n]}$ are irreducible symplectic varieties with simply connected smooth locus.
\end{cor}
These corollaries are proved in Section \ref{Fujikivar}.

In some sense the Fujiki construction is fully general.
\begin{thm}\label{geneth}
Let $\mathcal{G}$ be a finite automorphism group of $S^n$. Let $Y\rightarrow S^n/\mathcal{G}$ be a terminalization. If $Y$ is an irreducible symplectic variety with simply connected smooth locus, then there exists:
\begin{itemize}
\item
 a K3 surface $\Sigma$, 
\item 
a finite symplectic automorphism group $G$ on $\Sigma$ which is abelian when $n\geq3$,
\item
a valid involution $\theta$,
\end{itemize}
such that all Kähler terminalizations of $\Sigma^n/\left\langle j_{\theta}(G),\mathfrak{S}_n\right\rangle$ are deformation equivalent to $Y$.
\end{thm}
This theorem is a consequence of Theorem \ref{quotientsvalide} and is proven in Section \ref{mainthsection}. 
Let $\prod S_i$ be the product of K3 surfaces not all isomorphic. Let $\mathcal{G}$ be a finite automorphism group of $\prod S_i$. According to \cite[Page 10, Remark b]{Beauville}, a terminalization $Y\rightarrow \prod S_i/\mathcal{G}$ is not primitively symplectic. Therefore, we see with Theorem \ref{geneth} that, up to deformation, the Fujiki construction covers all possible examples obtained from quotients of products of K3 surfaces.

In dimension 4, the Fujiki construction leads to several examples of irreducible symplectic orbifolds. 
\begin{defi}\label{admissiblegroup}\label{admidefi}
A finite symplectic automorphism group $G$ on a K3 surface $S$ is said 
\emph{admissible} if $S/G$ has only singularities of type $A_1$, $A_2$, $A_3$ or $A_5$.
\end{defi}
This definition is related to the notion of admissible singularities of Fujiki in \cite[Section 7]{Fujiki}. Indeed, if $G$ is admissible, then $S^2/\left\langle j_{\theta}(G),\mathfrak{S}_2\right\rangle$ has admissible singularities in the sense of Fujiki. It corresponds to singularities with a well known orbifold terminalization (see Section \ref{recallFujiki} for more details).
\begin{thm}\label{main4}
All the possible (up to deformation) irreducible symplectic orbifolds in dimension 4 obtained by Fujiki construction with an admissible group $G$ are listed below. We provide 
their second Betti number $b_2$, the group $G$ used for the construction and their singularities:
\[
\begin{tabular}{|c|c|c|}
\hline
 $b_2$&$G$  & singularities\\
\hline
 4&$\Ay_4^2$  & $a_2=4$, $a_3=6$, $a_4=4$, $\by_6=2$\\
\hline
 5&$C_3\times  \Ay_4$ & $a_2=3$, $a_3=9$, $a_4=3$, $\by_6=1$\\
\hline
 5&$C_4^2\rtimes C_3$ & $a_2=10$, $a_3=15$, $a_4=1$, $a_8=2$, $\by_4=1$\\
\hline
 5&$C_2^4\rtimes C_6$ & $a_2=16$, $a_3=6$, $a_4=4$, $a_6=1$, $\by_4=1$\\
\hline
 6&$C_3\times \Sy_3$  & $a_2=9$, $a_3=10$, $a_6=1$\\
\hline
 6&$C_2\times \Ay_4$  & $a_2=13$, $a_3=6$, $a_4=4$, $a_6=1$\\
\hline
 6&$C_3^2\rtimes C_4$ & $a_2=10$, $a_3=14$, $a_4=6$\\
\hline
 6&$C_2^2\rtimes \Ay_4$ & $a_2=14$, $a_3=15$, $a_4=6$\\
\hline
 6&$C_2^4\rtimes \Sy_3$ & $a_2=19$, $a_3=12$, $a_4=6$, $\by_4=1$\\
\hline
 7&$C_3^2$  &  $a_3=12$\\
\hline
 7&$\Ay_4$  & $a_2=12$, $a_3=15$, $a_4=4$\\
\hline
 7& $C_2^2\times \Ay_4$& $a_2=12$, $a_3=3$, $a_4=4$\\
\hline
 7&$\mathfrak{S}_3\wr C_2$  & $a_2=20$, $a_3=12$, $a_4=3$\\
\hline
 8&$C_6$  & $a_2=9$, $a_3=6$, $a_6=1$\\
\hline
 8&$C_2\times C_6$  & $a_2=12$, $a_3=3$\\
\hline
 8&$C_4^2$  & $a_2=6$\\
\hline
 8&$\Ay_{3,3}$  & $a_2=28$, $a_3=12$\\
\hline
 8&$\Ay_{3,3}$  & $a_2=28$, $a_3=20$\\
\hline
 8&$\Sy_4$  & $a_2=24$, $a_3=12$, $a_4=3$\\
\hline
 8&$C_2^3\rtimes C_4$ & $a_2=17$, $a_4=6$, $\by_4=1$\\
\hline
 8&$\Sy_{3,3}$ & $a_2=20$, $a_3=15$\\
\hline
 10&$C_4$  & $a_2=10$, $a_4=6$\\
\hline
10 & $C_2^2\rtimes C_4$ & $a_2=10$, $a_4=6$\\
\hline
 10&$\Sy_3$  & $a_2=28$, $a_3=12$\\
\hline
 10&$C_2\times C_4$ & $a_2=12$, $a_4=4$ \\
\hline
 10&$\mathcal{D}_6$  & $a_2=28$, $a_3=10$\\
\hline
10& $C_2\times \Sy_4$  & $a_2=28$, $a_3=10$\\
\hline
 11&$C_3$  & $a_3=15$\\
\hline
 11&$\mathcal{D}_4$  & $a_2=36$, $a_4=3$\\
\hline
11 & $C_2^2\wr C_2$ & $a_2=36$, $a_4=3$\\
\hline
 14&$C_2^2$  & $a_2=36$\\
\hline
 16&$C_2$  & $a_2=28$ \\
\hline
\end{tabular} 
 \]
\end{thm}
The previous theorem is a consequence of Theorem \ref{main3} and Proposition \ref{defequiprop}.
The numbers $a_k$ correspond to the number of symplectic cyclic singularities of order k, for $k\leq 6$. 
In dimension 4, there are several possibilities for symplectic cyclic singularities of order 8. In our case there are of analytic type:
$\C^4/\left\langle g\right\rangle$ with:
$g=\diag(\xi_8,\xi_8^{-1},\xi_8^{3},\xi_8^{5})$ and their number is denoted by $a_8$.
The number $\mathfrak{b}_k$ corresponds to the number of singular points of the analytic form $\C^4/\left\langle g,s\right\rangle$ with:
$$g=\diag(\xi_k,\xi_k^{-1},\xi_k^{-1},\xi_k)\ \text{and}\   s=\begin{pmatrix}
																													0 & 0 & 1 & 0\\
																													0 & 0 & 0 & 1\\
																													-1 & 0 & 0 & 0\\
																													0 & -1 & 0 & 0
																													\end{pmatrix}.$$
More details on the groups names are given in Section \ref{notanota}.		

There are two main difficulties to prove the previous theorem which are the reasons for the length of this paper. One of the difficulty is to determine the singularities of the orbifolds. The second difficulty is to prove that we have found all the possible valid involutions. For both difficulties, the help of a computer has been required.
\begin{rmk}
Note that there are three couples of orbifolds that share the same second Betti number and singularities. They are orbifolds constructed with the groups:
\begin{itemize}
\item
$C_2^2\wr C_2$ and $\mathcal{D}_4$;
\item
$C_2^2\rtimes C_4$ and $C_4$;
\item
$C_2\times \mathfrak{S}_4$ and $\mathcal{D}_6$.
\end{itemize}
These three couples could lead to deformation equivalent orbifolds. However all the other orbifolds are necessarily independent under deformation
according to Proposition \ref{locallytriv}. Therefore, we have found at least 29 deformation classes of irreducible symplectic orbifolds.
\end{rmk}
When $G$ is abelian, there is only one possible valid involution given by $\inv(g)=g^{-1}$ for all $g\in G$. To simplify the notation, when $G$ is abelian, if we do not specify the involution $\theta$, it refers to this unique valid involution $\inv$. 
In dimension higher or equal to 6, we will see in Lemma \ref{n>2} that without loss of generality, we can assume that $G$ is abelian and $\theta$ valid.
We obtain:
\begin{prop}\label{Higherprointro}
In dimension higher or equal to 6, all the Fujiki varieties (up to deformation) are given by the 14 following series:
$S(C_2)^{[n]}$; $S(C_3)^{[n]}$; $S(C_2^2)^{[n]}$; $S(C_4)^{[n]}$; $S(C_5)^{[n]}$; $S(C_6)^{[n]}$; $S(C_7)^{[n]}$; $S(C_2^3)^{[n]}$; $S(C_2\times C_4)^{[n]}$; $S(C_8)^{[n]}$; $S(C_3^2)^{[n]}$; $S(C_2\times C_6)^{[n]}$; $S(C_2^4)^{[n]}$; $S(C_4^2)^{[n]};$ with $n\geq3$.
Moreover all these varieties are independent under deformation.
\end{prop}
This proposition is proved in Section \ref{Fujikirelationsection}.
The second Betti numbers of the previous series are provided in Corollary \ref{examples}.

It is not an easy problem to determine if the irreducible symplectic varieties given by Proposition \ref{Higherprointro} are orbifolds since the terminalization are mostly unknown in practice. However in few cases, the terminalizations are easy to get.
\begin{prop}\label{dim6}
The varieties
$S(C_2^k)^{[3]}$ are irreducible symplectic orbifolds for all $1\leq k \leq 4$.
\end{prop} 
This proposition corresponds to Proposition \ref{dim6bis}. 

The 29 examples of irreducible symplectic orbifolds in dimension 4 (see Theorem \ref{main4}) and the 4 examples of irreducible symplectic orbifolds in dimension 6 (see Proposition \ref{dim6}) allow to get 33 new deformation classes of compact hyperkähler orbifolds.

\subsection{Previous related works} 
We give an overview on the previous works providing examples of irreducible symplectic varieties. More information can be found in \cite{Perego}.
\subsubsection*{Fujiki work}
Fujiki in \cite[Theorem 13.1]{Fujiki} provides the orbifolds $S(G)^{[2]}$ when $G$ is admissible and abelian. In this paper, we have shown via Corollary \ref{corirr} that these examples are irreducible symplectic orbifolds. Moreover, Fujiki has made mistakes on his computation of the singularities for some of these orbifolds. These mistakes have been corrected in the present paper (see Remark \ref{mistakeFujiki}).
\subsubsection*{Orbifolds of Nikulin type}
The most well known example of IHS orbifolds are the \emph{orbifolds of Nikulin type} (see \cite[Definition 3.1]{Camere} for the terminology). 
These orbifolds are deformation equivalent to the Fujiki orbifold $S(C_2)^{[2]}$ as explained in \cite[Proposition 3.10]{Ulrike}. These orbifolds were already known to be irreducible symplectic as explained in \cite[Proposition 3.8]{Lol}.
As explained in \cite[Proposition 3.12]{Ulrike}, the Nikulin orbifolds are also deformation equivalent to the Markushevich--Tikhomirov varieties $\mathcal{P}^{0}$ introduced in \cite[Definition 3.3]{Markou}. Their Beauville--Bogomolov lattice has been computed in \cite{Lol3} and their wall divisors are given in \cite{Ulrike2}. Moreover a compete family of these orbifolds is given in \cite{Camere}.
\subsubsection*{Quotients of generalized Kummer}
Let $K_2(T)$ be a generalized Kummer fourfold constructed from a 2-dimensional torus $T$. The involution $-\id$ on $T$ induces a symplectic involution $\iota$ on $K_2(T)$ with a fixed locus given by a K3 surface and 36 isolated points. Let $K'\rightarrow K_2(T)/\iota$ be the blow-up in the K3 surface of singularities. The orbifold $K'$ is an irreducible symplectic orbifold as explained in \cite[Proposition 3.8]{Lol}. Moreover, its Beauville--Bogomolov form has been computed in \cite{Kapfer}. Since the orbifold $K'$ has second Betti number 8 and 36 singularities of type $a_2$ (see Notation \ref{defsing}), we know by Proposition \ref{locallytriv} that it is not deformation equivalent to one of the examples provided by Theorem \ref{main4}.

Similarly in \cite[Section 5.5]{Fu}, we provided the orbifold $K'_3$ obtained as a blow-up of $K_2(T)$ quotiented by an automorphism of order 3. By Proposition \ref{mainprop}, we know that $K'_3$ is an irreducible symplectic orbifold. However in \cite[Section 5.5]{Fu}, it is shown that $K'_3$ has second Betti number 7 and 12 singularities of type $a_3$ (see Notation \ref{defsing}). Therefore according to the data of Theorem \ref{main4}, the orbifold $K_3'$ could be deformation equivalent to $S(C_3^2)^{[2]}$.
\subsubsection*{Relative Prymian}
Another method to construct irreducible symplectic orbifolds existing in the literature is via \emph{relative Prymian} (see \cite{Markou}, \cite{Arbarello}, \cite{Matteini} and \cite{Shen}). The general idea is the following. Let $(S,H)$ be a polarized K3 surface endowed with an anti-symplectic involution $i$. Let $C\in \left|H\right|^i$ be a smooth curve, we can consider the Prym variety $\Prym(C,i_{|C})$. The relative Prymian associated to $(S,H)$ and $i$ is a variety $Y$ with a Lagrangian fibration $Y\rightarrow \left|H\right|^i$ such that the generic fibers are the $\Prym(C,i_{|C})$.

The relative Prymian in \cite{Markou} is deformation equivalent to $S(C_2)^{[2]}$ as explained before. However several questions are open when considering the other examples in \cite{Arbarello}, \cite{Matteini} and \cite{Shen}. Are they all irreducible symplectic? Are they all orbifold? What are their second Betti numbers? Because of this lack of knowledge, it is not possible to tell if these examples are deformation equivalent or not to one of the examples in the current paper.
\subsubsection*{Moduli spaces of semi-stable sheaves}
In a larger perspective, I also would like to mention the work of Perego and Rapagnetta in \cite{Perego2} who provide series of irreducible symplectic varieties which are not orbifold. These examples are obtained via moduli spaces of semi-stable sheaves on K3 surfaces or abelian surfaces.
\subsection{What remains to be explored in dimension 4}\footnote{Several of the cases listed below have been studied during the referring process of this paper. Up to my knowledge all the new works are mentioned in this section.}
Already in dimension 4, many directions remain to be explored in order to find more examples of irreducible symplectic orbifolds. As mentioned above an irreducible symplectic orbifold $Y$ can be obtained as a terminalization of a quotient $X/G$ with at least 4 possibilities for $X$. 
\subsubsection*{When $X$ is the product of two K3 surfaces}
This is the case considered in the present paper and by Theorem \ref{geneth}, it corresponds to the orbifolds obtained via Fujiki construction from a K3 surface $S$, a symplectic group $G$ on $S$ and a valid involution $\theta: G\rightarrow G$ (see Definition \ref{Fujikiconstruc}). However, we have only explored the examples with admissible groups $G$ (see Definition \ref{admissiblegroup}). Therefore according to \cite{xiao}, it remains 47 groups to be studied. The difficulty to study these other cases when $G$ is not admissible is the lack of known terminalizations.
The quotient $S^2/\left\langle j_{\theta}(G),\mathfrak{S}_2\right\rangle$ has singularities in codimension 2 for which a terminalization is not known. However, our method to determine the valid involutions (see Section \ref{examplescalculs}) and to compute the Betti numbers (See Proposition \ref{b2} and Section \ref{topo}) still apply when $G$ is not admissible.
\subsubsection*{When $X$ is a manifold of Kummer type}
Let $T$ be a 2-dimensional torus. Let $K_n(T)$ be the associated generalized Kummer hyperkähler manifold of dimension $2n$. An automorphism group of $K_n(T)$ is said \emph{natural} if it is induced by an automorphism of $T$.
 
Bertini, Capasso, Grossi, Mauri and Mazzon (\cite{Bertini}) have classified the cases when $X$ is a generalized Kummer hyperkähler manifold and $G$ a natural symplectic automorphism group. It remains to study the cases when $G$ is non-natural. 
\subsubsection*{When $X$ is a manifold of $K3^{[2]}$-type}
In \cite{Gerald}, the symplectic automorphism groups of such a manifold are classified. Therefore, it could be possible to study these cases. 

Let $S$ be a K3 surface. Let $S^{[2]}$ be the Hilbert scheme of 2 points on $S$. As before a \emph{natural} automorphism of $S^{[2]}$ is an automorphism induced by an automorphism of $S$.
Note that 
the varieties obtained with $X=S^{[2]}$ and $G$ a natural symplectic automorphism groups correspond to the Fujiki varieties when $\theta=\id$. 
Bertini, Capasso, Grossi, Mauri and Mazzon (\cite{Bertini}) have classified this case.

Therefore, it remains to study the case when $G$ is non-natural. When $G$ is natural but non admissible (see Definition \ref{admidefi}), it also remains to compute the singularities of the obtained varieties.
\subsubsection*{When $X$ is a 4-dimensional torus}
Fujiki in \cite[Section 13 and 14]{Fujiki}, has already partially considered this case. Indeed, we can perform the Fujiki construction with a 2-dimensional torus instead of a K3 surface. Let $T$ be a 2-dimensional torus, $G$ a symplectic automorphism group of $T$ and $\theta:G\rightarrow G$ an involution. We can consider a terminalization $T(G)^{[2]}_{\theta}\rightarrow T^2/\left\langle j_{\theta}(G),\mathfrak{S}_2\right\rangle$. In \cite[Theorem 13.1]{Fujiki}, Fujiki studies the examples $T(G)^{[2]}_{\theta}$ when $G$ is abelian and $\theta$ the valid involution on an abelian group. 
However, Proposition \ref{mainprop} is not applicable in this context; consequently, these orbifolds are not necessarily irreducible symplectic. It can be verified that while some are irreducible symplectic, others are not. For instance, when $G$ is abelian, Yamaguchi in \cite{Yamaguchi} has classified which orbifolds $T(G)^{[2]}_{\theta}$ are irreducible symplectic (with $\theta$ the only valid involution on an abelian group). He proved that $T(C_3^3)^{[2]}$, $T(C_2^2\times C_4)^{[2]}$ and $T(C_6)^{[2]}$ are irreducible symplectic while $T(C_3)^{[2]}$, $T(C_3^2)^{[2]}$, $T(C_4)^{[2]}$ and $T(C_2\times C_4)^{[2]}$ are not. Note that $T(C_3^3)^{[2]}$ is deformation equivalent to the generalized Kummer fourfold. 
\subsubsection*{Relative Prymians}
As mentioned above, some irreducible symplectic orbifolds can be obtained via the relative Prymian varieties. There are many possibilities for such constructions varying the K3 surface $S$, the involution $i$ and the polarization $H$. These varieties have been studied in full generality in \cite{Brakkee}. The authors provide a criteria for when these varieties are irreducible symplectic. The remaining task is to determine which of these varieties are orbifolds.
\subsubsection*{To go further}
We could even imagine to go further and choose for $X$ one of our orbifold examples. For instance the symplectic automorphisms of a Nikulin type orbifold have been classified in \cite{Brandhorst}. In particular two new automorphisms are identified (new in the sense that they are not induced from automorphisms of a $K3^{[2]}$-type hyperkähler manifold and so can potentially lead to orbifolds that could not be obtained from a quotient of a $K3^{[2]}$-type hyperkähler manifold). 
\subsection{Organization of the Paper}
The paper is organized as follows. The section \ref{remind} is dedicated to some reminders and to the proof of Proposition \ref{mainprop}. 
In Sections \ref{mainthsection} we prove Theorem \ref{geneth}. In Section \ref{Fujikivar}, Corollaries \ref{corirr} and \ref{corirr2} are proved. In Section \ref{BettiSection}, we explain how to compute the second Betti number of a Fujiki variety. In Section \ref{Fujikirelationsection}, we study the Fujiki constant of a Fujiki variety and we prove
Proposition \ref{Higherprointro}.
In section \ref{deformation}, we propose a criterion to determine when two involutions on a same group provide bimeromorphic Fujiki varieties. In Section \ref{sing}, we explain our method to determine the singularities of a Fujiki variety of dimension 4 when $G$ is admissible. In particular in Section \ref{verif}, we provide a technique to verify the computation of the singularities based on a result of Beckmann and Song \cite{Song}. 
In Section \ref{examplescalculs}, we provide our examples of orbifolds; in particular, we prove Theorem \ref{main4} and Proposition \ref{dim6}. Finally
 in the annexes, we provide explanations on our computer programs.
\subsection{Notation}\label{notanota}
\begin{nota}
Let $Y$ be a complex space. We denote $Y_{reg}:=Y\smallsetminus \Sing Y$.
\end{nota}
\begin{nota}
\begin{itemize}
\item
We denote by $C_n$ the cyclic group of order $n$.
\item
We denote by $\mathcal{D}_n$ the dihedral group of order $2n$.
\item
We denote by $\mathfrak{S}_n$ and $\mathfrak{A}_n$ the symmetric and the alternating groups of $n$ elements.
\item
We set $\mathfrak{A}_{3,3}:=\mathfrak{S}_{3}^2\cap \mathfrak{A}_{6}$.
\end{itemize}
\end{nota}																													
For some groups, the groups names that we use are not the one used in \cite{xiao}. We provide the relation between our groups names and the Xiao's groups names.
\begin{center}
 \underline{\textbf{Groups names}}
\end{center}
\[
\begin{tabular}{|c|c|c|}
\hline
Group name of the paper & Xiao's group name & \textbf{SmallGroup} \\
\hline
$C_2^2\rtimes C_4$ & $\Gamma_2c_1$ & (16,3) \\
\hline
$C_2^3\rtimes C_4$ & $\Gamma_7a_1$ & (32,6) \\
\hline
$C_2^2\wr C_2$ & $2^4C_2$ & (32,27) \\
\hline
$C_3^2\rtimes C_4$ & $3^2C_4$ & (36,9)\\
\hline
 $\mathfrak{S}_3^2$ &$\mathfrak{S}_{3,3}$& (36,10)\\
\hline
$C_4^2\rtimes C_3$ & $4^2C_3$&  (48,3) \\
\hline
$C_2^2\times\mathfrak{A}_4$ & $2^2(C_2\times C_6)$& (48,49)\\
\hline
$C_2^2\rtimes \mathfrak{A}_4$ & $2^4C_3$ & (48,50)\\
\hline
$\mathfrak{S}_3\wr C_2$ & $N_{72}$ & (72,40) \\
\hline
$C_2^4\rtimes C_6$ & $2^4C_6$ & (96,70)\\
\hline
$C_2^4\rtimes \mathfrak{S}_3$ & $2^4D_6$ & (96,227)\\
\hline
\end{tabular}
\]
The group $\mathfrak{S}_3\wr C_2$ alias $N_{72}$ is one of the maximal symplectic finite group of a K3 surface described by Mukai in \cite{Mukai}.

~\\
\noindent\textbf{Acknowledgements.} 
I am very grateful to Romain Demelle, Arvid Perego, Martin Schwald and Jieao Song for very useful discussions. I also want to thank Annalisa Grossi, Mirko Mauri and Enrica Mazzon for discussing their on coming paper and for very interesting comments. I would like to warmly thank the referee for the valuable suggestions that help to improve the
readability of the paper. 
This work has been financed by the PRCI SMAGP (ANR-20-CE40-0026-01).

\section{Irreducible symplectic varieties}\label{remind}
\subsection{Definitions}\label{defdef}
We follow the usual definitions from \cite{Bakker} and \cite{Bakker2}.
\begin{defi}
Let $X$ be a normal complex analytic space. 
\begin{itemize}
\item
We denote by $\Omega_X^{[p]}$ the sheaf of \emph{reflexive holomorphic $p$-forms} on $X$ given by $i_*\Omega_{X_{reg}}^p$, where $i:X_{reg}\hookrightarrow X$ is the inclusion.
\item
We call a \emph{holomorphic symplectic form} on $X$ a closed reflexive 2-form on $X$ which is non-degenerate at each point of $X_{reg}$.
\item
Let $f:\widetilde{X}\rightarrow X$ be a resolution of singularities. We say that $X$ is a \emph{symplectic variety} if $X$ admits a symplectic form $\sigma$ such that $f^*(\sigma)$ extends to a holomorphic 2-form on $\widetilde{X}$.
\end{itemize}
\end{defi}
\begin{rmk}
By \cite[Corollary 1.8]{Kebekus}, a complex analytic space with only rational singularities is a symplectic variety if and only if it admits a symplectic form.
\end{rmk}
\begin{defi}\label{Primitivy}
A \emph{primitive symplectic variety} is a normal compact Kähler symplectic variety with $h^1(X,\mathcal{O}_X)=0$ and $h^0(X,\Omega_X^{[2]})=1$.
\end{defi}
\begin{defi}\label{Defi0}
Let $X$ be a symplectic compact Kähler variety $(X,\sigma_X)$ with rational singularities. We say that $X$ is an \emph{irreducible holomorphic symplectic (IHS) variety} if for all quasi-étale covers $q:\widetilde{X}\rightarrow X$, the algebra $H^0\left(\widetilde{X},\Omega^{[\bullet]}_{\widetilde{X}}\right)$ is generated by the reflexive pullback $q^*(\sigma_X)$.
\end{defi}
\begin{rmk}
An irreducible symplectic variety is a primitive symplectic variety; however the contrary is not always true.
\end{rmk}
\subsection{$\Q$-factorial and terminal singularities}\label{Terminalsection}
\begin{defi}
Let $X$ be a variety with canonical singularities. Let $Y$ be a $\Q$-factorial normal variety with only terminal singularities. We say that $Y$ is a \emph{terminalization} of $X$ if there exists a crepant proper bimeromorphic morphism $Y\rightarrow X$. 
\end{defi}
\begin{thm}[\cite{Birkar}, Corollary 1.4.3]
If $X$ is an algebraic variety then $X$ admits a terminalization.
\end{thm}
Therefore in a program of classification, the IHS varieties with $\Q$-factorial and terminal singularities are more relevant. Moreover in this case, their moduli spaces are more coherent as shown by the following result. 
\begin{prop}[\cite{Bakker2}, Proposition 5.13]\label{locallytriv}
Let $X$ be an IHS varieties with $\Q$-factorial and terminal singularities. Then any deformations of $X$ is locally trivial.
\end{prop}
For these reasons, we are going to investigate IHS varieties with $\Q$-factorial and terminal singularities in this paper.
We recall how to characterize terminal singularities in our case.
\begin{prop}[\cite{Namikawa}, Corollary 1]\label{terminal}
A symplectic variety has terminal singularities if and only if $\codim \Sing X\geq 4$.
\end{prop}
\begin{prop}\label{bimero}
Let $Y$ and $Y'$ be two primitive symplectic varieties which are two terminalizations of a same projective variety $Z$. 
Then $Y$ and $Y'$ are deformation equivalent.
\end{prop}
\begin{proof}
Note first that $Y$ and $Y'$ are projective. Indeed they are Kähler from the definition of primitive symplectic varieties and they are Moishezon because $Z$ is Moishezon. Therefore by \cite[Theorem 6]{Namikawa2} they are projective.

Since $Y$ and $Y'$ have terminal singularities, 
the proof of \cite[Lemma 3.2]{Ulrike} can be copied word for word to prove that there exists a birational map
$f:Y\dashrightarrow Y'$ which is an isomorphism in codimension 1. Since $Y$ and $Y'$ have $\Q$-factorial singularities, this implies a well defined isomorphism map $f_*:(\Pic Y)_{\Q} \rightarrow (\Pic Y')_{\Q}$. Therefore we conclude the proof with \cite[Corollary 6.17]{Bakker2}.
\end{proof}
\begin{thm}[\cite{Yamagishi}, Theorem 3.1]\label{YamaTh}
Let $f:Y \to \C^n/G$ be a terminalization for a quotient $\C^n/G$ with $G \subset \mathrm{SL}_n(\C)$ finite.
Then there exists a bijection
\[
\{\text{irreducible exceptional divisors of }f\} 
\;\longleftrightarrow\; 
\{\text{conjugacy classes of junior elements of }G\},
\]
such that, if $E$ corresponds to an element $g \in G$ of order $r$, one has the equality of valuations
\[
\nu_E \;=\; \tfrac{1}{r}\,\nu_g|_{\mathbb{C}(Y)}.
\]
\end{thm}

\begin{remark}\label{YamaRk}
The equality $\nu_E = \tfrac{1}{r}\,\nu_g$ should be understood as a precise identification between 
two valuations defined on the invariant field $\mathbb{C}(X_1,\dots,X_n)^G$:
\begin{itemize}
  \item the \emph{geometric valuation} $\nu_E$, associated with the exceptional divisor $E$ over $\mathbb{C}^n/G$, and
  \item the \emph{combinatorial valuation} $\nu_g$, associated with the group element $g\in G$.
\end{itemize}

\smallskip
Recall that if 
\[
g=\mathrm{diag}(e^{2\pi i a_1/r},\dots,e^{2\pi i a_n/r}),
\]
then for a monomial $X_1^{m_1}\cdots X_n^{m_n}$ in $\mathbb{C}[X_1,\dots,X_n]$, one defines:
\[
\nu_g(X_1^{m_1}\cdots X_n^{m_n}) = \sum_{j=1}^n m_j a_j.
\]
This construction induces by restriction a valuation on the invariant ring 
$\mathbb{C}[X_1,\dots,X_n]^G$.

\smallskip
If $g$ is a \emph{junior element}, then in a suitable basis we may write
\[
g=\mathrm{diag}(e^{2\pi i a_1/r},\,e^{2\pi i a_2/r},\,1,\dots,1),
\]
with $a_1,a_2>0$ and $a_1+a_2=r$.  
In this case the fixed locus of $g$ is:
\[
\mathrm{Fix}(g)=\{(x_1,x_2,\dots,x_n)\in\mathbb{C}^n \mid x_1=x_2=0\}.
\]
Consequently,
\[
\nu_g(x_1^{m_1}\cdots x_n^{m_n}) = m_1a_1+m_2a_2.
\]
Hence $\nu_g$ measures only the vanishing of a $G$-invariant function 
along the directions of the coordinates $(x_1,x_2)$, which are normal to $\mathrm{Fix}(g)$.

\smallskip
The identity $\nu_E = \tfrac{1}{r}\nu_g$ means that $E$ is exactly the exceptional divisor 
whose geometric valuation coincides with the combinatorial one determined by $g$.  
Therefore, this implies that $E$ dominates the image of the fixed locus in the quotient:
\[
E \text{ lies over } \mathrm{Fix}(g)/G \subset \mathbb{C}^n/G.
\]

\smallskip
In particular, since only junior elements of $G$ give rise to exceptional divisors in the resolution $f$, all exceptional divisors of $f$ lie over loci of codimension~2 in $\mathbb{C}^n/G$.
\end{remark}
\subsection{Fundamental group of the smooth locus}\label{fundasection}
In this paper, we look for IHS variety with simply connected smooth locus. 
In this section, we explain why this is relevant and we provide a criterion in the case of terminalizations of quotients.
\begin{prop}\label{basisi}
Let X be a compact Kähler analytic complex space with rational singularities such that $H^0\left(X, \Omega^{[2]}_{X}\right)=\C\sigma$ with $\sigma$ a holomorphic symplectic form. If $\pi_1(X_{reg})=0$ then $X$ is an IHS variety.
\end{prop}
\begin{proof}
 Indeed, let
$X$ be as in the statement of the proposition. Let $\gamma:Y\rightarrow X$ be a quasi-étale cover. By purity of branch loci, $\gamma$ is étale over the smooth locus $X_{reg}$ of $X$. Since $\pi_1(X_{reg})=0$, the cover $\gamma$ is necessarily trivial. Hence, $X$ is an IHS variety.
\end{proof}
However, there exists IHS varieties with $\pi_1(X_{reg})\neq0$.
\begin{prop}\label{codim4}
Let $X$ be an IHS manifold and $G$ a non-trivial finite symplectic automorphism group of $X$. If for all $g\in G$, we have $\codim \Fix g\geq 4$ then $X/G$ is an IHS variety. Moreover $\pi_1((X/G)_{reg})\neq0$.
\end{prop}
\begin{proof}
By Proposition \ref{terminal}, we know that $X/G$ has terminal singularities.
Then, the proof of the first statement is identical to the one of \cite[Proposition 2.15]{Perego}.
Let $p:X\rightarrow X/G$ be the quotient map.
To prove the second statement, we only need to remark that the restriction
$$p:X\smallsetminus \left(\bigcup_{g\in G\smallsetminus\left\{\id\right\}}\Fix g\right)\rightarrow (X/G)_{reg}$$
is an étale cover.
\end{proof}

In order to obtain irreducible symplectic varieties with simply connected smooth locus, we propose the following criterion. A similar result have been proved independently by Bertini, Capasso, Grossi, Mauri and Mazzon (\cite[Proposition 7.1]{Bertini}).
\begin{prop}\label{fonda}
Let $X$ be a simply connected manifold which admits a nowhere vanishing holomorphic 2-form $\varphi$. Let $G$ be a finite automorphism group of $X$ which respects the holomorphic 2-form $\varphi$ (that is $g^*(\varphi)=\varphi$ for all $g\in G$). Let $Y\rightarrow X/G$ be a terminalization of $X/G$. Let $Y_{reg}:=Y\smallsetminus \Sing Y$. Then $\pi_1(Y_{reg})=0$ if and only if $G$ is generated by automorphisms with fixed locus of codimension 2.
\end{prop}
\begin{proof}
Let $f:Y\rightarrow X/G$ be a terminalization. Let $Y^{o}:=Y\smallsetminus f^{-1}(f(\Sing Y))$. 
We consider the restriction $f_{reg}:Y^o\rightarrow \overline{X}'$, with $\overline{X}':=X/G\smallsetminus f(\Sing Y)$. 
Let $P:X\rightarrow X/G$ be the quotient map. We also denote $X'=P^{-1}(\overline{X}')$. We summarize our notation in the following figure.
$$\xymatrix@R10pt{Y\ar[r]^f&X/G&\ar[l]^{P}X\\
Y^{o}\ar[r]^{f_{reg}}\incl[u]&\overline{X}'\incl[u]&\ar[l]^{P_{|X'}}\incl[u] X'\\
Y\smallsetminus f^{-1}(f(\Sing Y))\ar@{=}[u]\ar[r]^{f_{reg}}& X/G\smallsetminus f(\Sing Y)\ar@{=}[u]&\ar[l]^{\ \ \ \ \ \ P_{|X'}}\ar@{=}[u]P^{-1}(\overline{X}').}$$

The morphism $f_{reg}$ is a resolution of singularities. Therefore by \cite[Theorem 7.8]{Kolar}, $\pi_1(Y^o)=\pi_1(\overline{X}')$. 
Moreover by Proposition \ref{terminal}, $\codim \Sing Y\geq 4$, therefore $\codim f(\Sing Y)\geq 4$. Then 
by Theorem \ref{YamaTh} and Remark \ref{YamaRk}, we have $\codim f^{-1}(f(\Sing Y))\geq2$. Hence the space $Y_{reg}\smallsetminus Y^o$ has codimension at least 2 in $Y_{reg}$, therefore we obtain $\pi_1(Y_{reg})=\pi_1(Y^o)$.
It remains to prove that $\pi_1(\overline{X}')=0$ if and only if $G$ is generated by automorphisms with fixed locus of codimension 2.

By Proposition \ref{terminal}, we have $\codim \Sing Y\geq 4$, in particular $\codim f(\Sing Y)\geq 4$. Moreover $\codim P^{-1}(f(\Sing Y))\geq 4$. It follows from \cite[Lemma 1.2]{Fujiki} that $\pi_1(\overline{X}')=0$ if $G$ is generated by automorphisms with fixed locus of codimension 2.

Now we assume that $G$ is not generated by automorphisms with fixed locus of codimension 2 and we show that 
$\pi_1(\overline{X}')\neq0$. Let $H\subset G$ be the sub-group of $G$ generated by elements with fixed locus of codimension 2. By assumption $H\neq G$. Moreover $H$ is a normal sub-group of $G$. Let $\overline{G}:=G/H$. We have $X/G=(X/H)/\overline{G}$. To prove our claim, it is enough to show that
$(X'/H)\rightarrow(X'/H)/\overline{G}$ is a non-ramified cover. Let $\overline{g}\in \overline{G}\smallsetminus\left\{\overline{\id}\right\}$. We show that $\Fix \overline{g} \cap X'/H=\emptyset$. Let $Q:X\rightarrow X/H$. Let $\overline{x}\in \Fix \overline{g}$ and $x\in Q^{-1}(\overline{x})$. We prove that $\overline{x}\notin X'/H$. Before, we summarize the new notation in the following figure.
$$\xymatrix@C15pt{P(x)\ar@{}[r]|-{\in}&X/G&\incl[l]\left(X'/H\right)/\overline{G}&\ar@{=}[l]\overline{X}'\\
\overline{x}\ar@{}[r]|-{\in}\ar[u]&X/H\ar[u]&\ar[u]\incl[l]X'/H&\\
x\ar@{}[r]|-{\in}\ar[u]&X\ar[u]^{Q}&\incl[l]X'\ar[u]^{Q_{|X'}}\ar[uur]_{P_{|X'}}&}$$
The sub-group $G_x:=\left\{\left.g\in G\right|\ g(x)=x\right\}$ induces a local action around the point $x$. The point $P(x)$ is a singular point in $X/G$ of analytic type $\C^n/G_x$, with $n=\dim X$. Since $\overline{x}\in \Fix \overline{g}$, we have $G_x\nsubseteq H$ and $G_x$ is not generated by automorphisms with fixed locus of codimension 2. It follows by \cite[Theorem 1.2]{Verbitsky} that the space $\C^n/G_x$ does not admit a crepant resolution. Hence necessarily $P(x)\notin \overline{X}'$ otherwise $Y^{o}$ would induce a crepant resolution for the singularities $\C^n/G_x$. Finally since $P(x)\notin \overline{X}'$, we obtain that $\overline{x}\notin X'/H$.

\end{proof}
\begin{rmk}\label{mainremark}
We consider the same objects as in the statement of the previous proposition.
If $G$ is not generated by automorphisms with fixed locus in codimension 2,
it is shown by Bertini, Capasso, Grossi, Mauri and Mazzon in \cite[Proposition 7.1]{Bertini}, that we can always find a terminalization $Y\rightarrow X/G$ such that $Y=Z/H$ with $Z$ 
a symplectic variety with simply connected smooth locus
and $H$ a non trivial group such that $\codim \Fix h\geq4$ for all $h\in H\smallsetminus \left\{\id\right\}$.
\end{rmk}
\begin{rmk}
Searching for examples in a program of classification, Propositions \ref{codim4}, \ref{fonda} and Remark \ref{mainremark} show that it is relevant to restrict our attention on the groups $G$ generated by automorphisms with fixed locus in codimension 2 and on irreducible symplectic varieties $Y$ with $\pi_1(Y_{reg})=1$. Moreover, in the orbifold case, the condition $\pi_1(Y_{reg})=1$ is required to be irreducible symplectic (see Definition \ref{defiorbi} below); this condition allows the uniqueness in the Beauville--Bogomolov decomposition theorem \cite{Campana}.
\end{rmk}
\subsection{Irreducible symplectic orbifolds}
As explained in the previous sections \ref{Terminalsection} and \ref{fundasection}, we want to look for IHS varieties with $\Q$-factorial terminal singularities and a simply connected regular locus. A particular case are the IHS orbifolds which appear very similar to the IHS manifolds in their behavior 
(see \cite{Lol} and \cite{Ulrike}). 
\begin{defi}\label{defiorbi}
Let $X$ be a compact Kähler complex analytic complex space with a unique, up to scalar, holomorphic symplectic form.
We say that $X$ is a \emph{primitively symplectic orbifold} if:
\begin{itemize}
\item
$X$ has only quotient singularities;
\item
$\codim \Sing X\geq 4$;
\end{itemize}
If in addition $\pi_1(X_{reg})=0$,
we say that $X$ is an \emph{irreducible symplectic (IHS) orbifold}.
\end{defi}
\begin{rmk}
In the orbifold case, note that the condition $\pi_1(X_{reg})=0$ allows the uniqueness of the Beauville--Bogomolov decomposition (see \cite{Campana}).
\end{rmk}

\subsection{Topological invariants of orbifolds in dimension 4}\label{topo}
In this section, we recall some results from \cite{Fu} with some additional computations in order to provide the Betti numbers and the Chern numbers of the examples in Theorem \ref{main3}. This is necessary to apply the verification method described in Section \ref{verif}.
\begin{defipro}[{\cite[Definition-Proposition 2.3]{Fu}}]\label{defilocal uniformizing system}
Let $X$ be an $n$-dimensional complex orbifold. Let $x\in \Sing X$. 
Then there exist a finite subgroup $G_x$ of $\GL_{n}(\C)$ and an open neighborhood $V_x\subset \C^n$ of the origin $0\in \C^n$, stable under the action of $G_x$, with $V_x/G_x$ isomorphic to an open neighborhood $U_x$ of $x$ (the map $\pi_x:V_x\rightarrow V_x/G_x\simeq U_x$ sending 0 to $x$) and such that:
$$\codim\Fix(g)\geq 2\ \text{for all}\ g\in G_x\backslash \{\id\}.$$
Such a group $G_x$ is unique up to conjugation. 
We call $(U_x,V_x,G_x,\pi_x)$ a \emph{local uniformizing system} of $x$, and $G_x$ the \emph{regional fundamental group} of $X$ at $x$.
\end{defipro}
\begin{defi}[locally V-free sheaf]\label{def:VB}
Let $X$ be an orbifold. 
Let $\mathscr{F}$ be a coherent sheaf on $X$. 
The sheaf $\mathscr{F}$ is said to be \emph{locally} V-\emph{free} if for any $x\in X$, there exist a local uniformizing system $(U_x,V_x,G_x,\pi_x)$, a free coherent sheaf $\hat{\mathscr{F}}_{V_x}$ on $V_x$, and a $G_x$-action on $\hat{\mathscr{F}}_{V_x}$ such that $\mathscr{F}_{|U_x}\simeq \pi_{x*}\left(\hat{\mathscr{F}}_{V_x}^{G_x}\right)$. 
\end{defi}
\begin{nota}\label{repre}
Let $X$ be an orbifold, $x\in X$ and $\mathscr{F}$ a locally V-free sheaf. Let $(U_x,V_x,G_x,\pi_x)$ be a local uniformizing system of $x$ and let $\hat{\mathscr{F}}_{V_x}$ be a locally free sheaf on $V$ endowed with an action of $G_x$ as in Definition \ref{def:VB}. Hence, the fiber of $\hat{\mathscr{F}}_{V_x}$ at $0$ is endowed with an action of $G_x$, which provides a representation of $G_x$. 
We denote by $\rho_{x,\mathscr{F}}$ \emph{the representation of $G_x$ associated with $x$ and $\mathscr{F}$}.
\end{nota}
\begin{nota}\label{sS0}
Let $X$ be an orbifold. Let $x\in \Sing X$ be an isolated singularity. We set:
\begin{equation}
s_x:=\frac{1}{|G_x|}\left[6\left(\sum_{\substack{g\in G_x\\g\neq \id}}\frac{\tr(\rho_{x,\Omega_X^{[1]}}(g))-4}{\det(\rho_{x, T_{X}}(g)-\id)}\right)+|G_x|-1\right],
\label{sdef}
\end{equation}
with $(U_x,V_x,G_x,\pi_x)$ a local uniformizing system around $x$. We set:
\begin{equation}\label{eqn:Si}
S_{0,x}:=\frac{1}{|G_x|}\sum_{g\neq\id}\frac{1}{\det(\id-\rho_{x, T_{X}}(g))}.
\end{equation}
When $X$ has only isolated singularities, we set:
$$s(X):=\sum_{x\in\Sing X} s_x\ \ \ \ \text{and}\ \ \ \ S_0(X):=\sum_{x\in\Sing X} S_{0,x}.$$
\end{nota}
\begin{prop}[{\cite[Proposition 3.6]{Fu}}]\label{RR}
Let $X$ be a primitively symplectic orbifold of dimension 4. We have:
\begin{equation*}
b_{4}(X)+b_{3}(X)-10b_{2}(X)=46+s(X),
\end{equation*}
or equivalently,
$$\chi(X)=48+12b_2(X)-3b_3(X)+s(X).$$
\end{prop}
\begin{rmk}
Proposition \ref{RR} shows that the knowledge of $b_2$, $b_3$ and the singularities is enough to compute the topological Euler characteristic and all the Betti numbers of a 4-dimensional primitively symplectic orbifold ($b_1=0$ by \cite[Proposition 6.7]{Fujiki}).
\end{rmk}
The invariant $s_x$ can be computed for all kind of isolated singularities $x$. We provide some examples of computations that are necessary for this paper.
\begin{lemme}[{\cite[Example 3.8]{Fu}}]\label{s}
Let $x$ be an isolated singularity of analytic type $\C^4/\left\langle g\right\rangle$, with
$g=\diag(\xi_n,\xi_n^{-1},\xi_n^{k},\xi_n^{-k})$ and $\gcd(n,k)=1$.
Then $$s_x=-(n-1).$$
\end{lemme}
\begin{lemme}\label{s2}
Let $x$ be an isolated singularity of analytic type $\C^4/\left\langle g,\sigma\right\rangle$, with
$$g=\diag(\xi_{2n},\xi_{2n}^{-1},\xi_{2n}^{-1},\xi_{2n})\ \text{and}\   \sigma=\begin{pmatrix}
																													0 & 0 & 1 & 0\\
																													0 & 0 & 0 & 1\\
																													-1 & 0 & 0 & 0\\
																													0 & -1 & 0 & 0
																													\end{pmatrix}.$$
Then 	$$s_x=-(n+2).$$																												
\end{lemme}	
\begin{proof}
Let $G=\left\langle g,\sigma\right\rangle$. The elements of $G\smallsetminus \left\{\id\right\}$ can be divided in two sets $A:=\left\{g,g^2,...,g^{2n-1}\right\}$ and $B:=\left\{\sigma,\sigma\circ g,...,\sigma\circ g^{2n-1}\right\}$. Therefore:
$$s_x:=\frac{1}{4n}\left[6\left(\sum_{\substack{h\in A}}\frac{\tr(\rho_{x,\Omega_X^{[1]}}(h))-4}{\det(\rho_{x, T_{X}}(h)-\id)}+\sum_{\substack{h\in B}}\frac{\tr(\rho_{x,\Omega_X^{[1]}}(h))-4}{\det(\rho_{x, T_{X}}(h)-\id)}\right)+4n-1\right].$$
By \cite[(15)]{Fu}:
$$\sum_{\substack{h\in A}}\frac{\tr(\rho_{x,\Omega_X^{[1]}}(h))-4}{\det(\rho_{x, T_{X}}(h)-\id)}=-\frac{4n^2-1}{6}.$$
Moreover:
$$\sigma\circ g^k=\begin{pmatrix}
																													0 & 0 & \xi_{2n}^{-k} & 0\\
																													0 & 0 & 0 & \xi_{2n}^k\\
																													-\xi_{2n}^k & 0 & 0 & 0\\
																													0 & -\xi_{2n}^{-k} & 0 & 0
																													\end{pmatrix}.$$
	Hence: $$\tr (\sigma\circ g^k)=0\ \text{and}\ \det(\sigma\circ g^k-\id)=4.$$
	We obtain:
		$$\sum_{\substack{h\in B}}\frac{\tr(\rho_{x,\Omega_X^{[1]}}(h))-4}{\det(\rho_{x, T_{X}}(h)-\id)}=-2n.$$	
Therefore:
\begin{align*}
s_x&=\frac{1}{4n}\left[-4n^2+1-12n+4n-1\right]\\
&=-(n+2)
\end{align*}
\end{proof}			
We can also compute the Chern numbers $c_4$ and $c_2^2$ via the Blache--Rieamann--Roch theorem (\cite{Blache}).	
\begin{prop}[{\cite[Theorem 2.14]{Blache}}]\label{c4}
Let $X$ be a 4-dimensional compact complex orbifold with only
isolated singularities. Then:
$$c_4(X)=\chi(X)-\sum_{x\in\Sing X}\left(1-\frac{1}{|\"aG_x|}\right),$$
where $G_x$ is the regional fundamental group (see Definition-Proposition \ref{defilocal uniformizing system}) of $X$ at $x$.
\end{prop}
\begin{prop}[{\cite[(2)]{Fu}}]\label{c2}
Let $X$ be a primitively symplectic orbifold of dimension 4. Then:
$$c_2(X)^2=720-240S_0(X)+\frac{c_4(X)}{3}.$$
\end{prop}
We provide $S_{0,x}$ for the singularities that appear in this paper.
\begin{lemme}\label{S0}
We consider the isolated singularity $x$ of analytic type $\C^4/G$. We set $$\sigma:=\begin{pmatrix}
																													0 & 0 & 1 & 0\\
																													0 & 0 & 0 & 1\\
																													-1 & 0 & 0 & 0\\
																													0 & -1 & 0 & 0
																													\end{pmatrix}.$$
\begin{itemize}
\item[(i)]
When $G=\left\langle -\id\right\rangle$, then $S_{0,x}=\frac{1}{2^5}$.
\item[(ii)]
When $G=\left\langle \diag(\xi_{3},\xi_{3}^{-1},\xi_{3}^{-1},\xi_{3})\right\rangle$, then $S_{0,x}=\frac{2}{27}$.
\item[(iii)]
When $G=\left\langle \diag(i,-i,-i,i)\right\rangle$, then $S_{0,x}=\frac{9}{2^6}$.
\item[(iv)]
When $G=\left\langle \diag(\xi_{6},\xi_{6}^{-1},\xi_{6}^{-1},\xi_{6})\right\rangle$, then $S_{0,x}=\frac{329}{864}$.
\item[(v)]
When $G=\left\langle \diag(\xi_{8},\xi_{8}^{-1},\xi_{8}^{3},\xi_{8}^5)\right\rangle$, then $S_{0,x}=\frac{41}{2^7}$.
\item[(vi)]
When $G=\left\langle \sigma,\diag(i,-i,-i,i)\right\rangle$, then $S_{0,x}=\frac{25}{2^7}$.
\item[(vii)]
When $G=\left\langle \sigma,\diag(\xi_{6},\xi_{6}^{-1},\xi_{6}^{-1},\xi_{6})\right\rangle$, then $S_{0,x}=\frac{545}{1728}$.
\end{itemize}
 
\end{lemme}		
\begin{proof}
\begin{itemize}
\item[(i)]
We have: $$S_{0,x}=\frac{1}{2}\times \frac{1}{2^4}=\frac{1}{2^5}.$$
\item[(ii)]
We have: $$S_{0,x}=\frac{2}{3}\times \frac{1}{\left(1-\xi_3\right)^2\left(1-\xi_3^{-1}\right)^2}=\frac{2}{3}\frac{1}{\left|1-\xi_3\right|^4}=\frac{2}{3}\times \frac{1}{9}=\frac{2}{27}.$$
\item[(iii)]
$$S_{0,x}=\frac{1}{4}\left(\frac{1}{2^4}+\frac{2}{(1-i)^2(1+i)^2}\right)=\frac{1}{4}\left(\frac{1}{2^4}+\frac{2}{2^2}\right)=\frac{9}{2^6}.$$
\item[(iv)]
We have:
$$S_{0,x}=\frac{1}{6}\left(\frac{1}{2^4}+\frac{2}{(1-\xi_6)^2(1-\xi_6^{-1})^2}+\frac{2}{(1-\xi_3)^2(1-\xi_3^{-1})^2}\right)=\frac{1}{6}\left(\frac{1}{2^4}+2+\frac{2}{9}\right)=\frac{329}{864}.$$
\item[(v)]
$$S_{0,x}=\frac{1}{8}\left(\frac{1}{2^4}+\frac{2}{(1-i)^2(1+i)^2}+\frac{4}{\left|1-\xi_8\right|^2\left|1-\xi_8^3\right|^2}\right).$$
However $\left|1-\xi_8\right|^2=2-\sqrt{2}$ and $\left|1-\xi_8^3\right|^2=2+\sqrt{2}$.
Hence:
$$S_{0,x}=\frac{1}{8}\left(\frac{1}{16}+\frac{1}{2}+2\right)=\frac{41}{2^7}.$$
\item[(vi)]
Let $g=\diag(i,-i,-i,i)$ and $G=\left\langle g,\sigma\right\rangle$.
The elements of $G\smallsetminus\left\{\id\right\}$ can be shared in two sets $A:=\left\{g,g^2,g^{-1}\right\}$ and $B:=\left\{\sigma,\sigma\circ g,\sigma\circ g^2,\sigma\circ g^{-1}\right\}$.
For all $h\in B$, we have $\det(h-\id)=4$. 
Hence:
$$S_{0,x}=\frac{1}{8}\left(\frac{4}{4}+\frac{1}{2^4}+\frac{2}{(1-i)^2(1+i)^2}\right)=\frac{1}{8}\left(1+\frac{1}{16}+\frac{1}{2}\right)=\frac{25}{2^7}.$$
\item[(vi)]
Let $g=\diag(\xi_{6},\xi_{6}^{-1},\xi_{6}^{-1},\xi_{6})$ and $G=\left\langle g,\sigma\right\rangle$.
As before, the elements of $G\smallsetminus\left\{\id\right\}$ can be shared in two sets $A:=\left\{g,g^2,...,g^{-1}\right\}$ and $B:=\left\{\sigma,\sigma\circ g,\sigma\circ g^2,...,\sigma\circ g^{-1}\right\}$.
For all $h\in B$, we have $\det(h-\id)=4$. 
Hence:
\begin{align*}
S_{0,x}&=\frac{1}{12}\left(\frac{6}{4}+\frac{1}{2^4}+\frac{2}{(1-\xi_6)^2(1-\xi_6^{-1})^2}+\frac{2}{(1-\xi_3)^2(1-\xi_3^{-1})^2}\right)\\
&=\frac{1}{12}\left(\frac{6}{4}+\frac{1}{16}+2+\frac{2}{9}\right)=\frac{545}{1728}.
\end{align*}
\end{itemize}
\end{proof}	
\section{Quotients of K3 surfaces to the power $n$}
\subsection{Criterion to be an irreducible symplectic variety}\label{mainthsection}
Let $S$ be a K3 surface and $n\geq2$ and integer.
In this section we want to exhibit the finite automorphism groups $\mathcal{G}$ on $S^n$ such that the quotient $S^n/\mathcal{G}$ leads to an irreducible symplectic variety. 
To simplify the readability, we recall Definition \ref{validintro} and notation (\ref{jtheta}) from the introduction.
\begin{nota}\label{jthetabis}
Let $S$ be a K3 surface. Let $G$ be a finite symplectic automorphism group of $S$. Let $\theta:G\rightarrow G$ be an involution.
We denote by $j_{\theta}:G\hookrightarrow \Aut(S^n)$ the embedding given by:
$$j_{\theta}(g)(x_1,x_2,x_3,x_4,...,x_n)=(g(x_1),\theta(g)(x_2),x_3,x_4,...,x_n),$$
that is $j_{\theta}(g)$ has the diagonal action given by $g$ on the first factor, by $\theta(g)$ on the second factor and by $\id$ on the other factors.
\end{nota}
\begin{defi}
Let $S$ be a K3 surface and $G$ a finite symplectic automorphism group of $S$.
Let $\theta:G\rightarrow G$ be an involution. We say that $\theta$ is a \emph{valid} involution if there exists $(g_1,...,g_k)$ a family of generators of $G$ such that $\theta(g_i)=g_i^{-1}$ for all $i\in\{1,...,k\}$.
\end{defi}
\begin{defi}
Let $n\geq2$ be an integer and $S$ a K3 surface. Let $\mathcal{G}$ be a finite subgroup of $\Aut(S^n)$. We say that $\mathcal{G}$ is \emph{primitive} if 
we cannot find $\mathcal{H}\subset \mathcal{G}$ a non-trivial normal subgroup and a K3 surface $\Sigma$ such that $\Sigma^n$ is a resolution of $S^n/\mathcal{H}$.
\end{defi}
\begin{defi}
Let $S$ be a K3 surface and $n\geq2$ an integer.
We call a \emph{trivial complex reflexion} an element of $\Aut(S)^n$ of the form $(\id,...,\id,g,\id,...,\id)$; that is an element which acts trivial on all factors of $S^n$ apart one.
\end{defi}
\begin{lemme}\label{n>2bis}
Let $S$ be a K3 surface endowed with a symplectic automorphism group $G$. Let $n\geq2$ be an integer. Let $\theta:G\rightarrow G$ be an involution. 
The group $\left\langle j_{\theta}(G),\mathfrak{S}_n\right\rangle\subset \Aut(S^n)$ is primitive if and only if the unique trivial complex reflection in $\left\langle j_{\theta}(G),\mathfrak{S}_n\right\rangle$ is the identity.
\end{lemme}
\begin{proof}
We set $\mathcal{G}=\left\langle j_{\theta}(G),\mathfrak{S}_n\right\rangle$. We denote by $\mathcal{H}$ the subgroup of $\mathcal{G}$ generated by trivial complex reflexions.

Since $\mathfrak{S}_n\subset \mathcal{G}$ necessarily we have $\mathcal{H}=H^n$ with $H$ a symplectic automorphism group of $S$. 
Let $\Sigma\rightarrow S/H$ be the K3 surface obtained by crepant resolution. Therefore $\Sigma^n\rightarrow S^n/H^n$ is a resolution. 
So if $\mathcal{H}$ is not trivial $\mathcal{G}$ is not primitive.

Now assume that $\mathcal{G}$ is not primitive. Then there exists a non-trivial subgroup 
$\mathcal{H}' \subset \mathcal{G}$ and a K3 surface $\Sigma$ such that 
$\Sigma^n \to S^n/\mathcal{H}'$ is a crepant resolution. 
By definition of a crepant resolution, holomorphic $2$-forms extend, hence:
$$
H^0(\Omega^2_{\Sigma^n}) \;\simeq\; H^0(\Omega^2_{S^n/\mathcal{H}'}) 
\simeq H^0(\Omega^2_{S^n})^{\mathcal{H}'}.
$$
On the other hand, we also know that:
$$
h^0(\Omega^2_{\Sigma^n}) \;=\; h^0(\Omega^2_{S^n}).
$$
Therefore the inclusion 
$H^0(\Omega^2_{S^n})^{\mathcal{H}'} \subset H^0(\Omega^2_{S^n})$ 
is an equality, which implies that $\mathcal{H}'$ acts trivially on 
$H^0(\Omega^2_{S^n})$, i.e.\ $\mathcal{H}'$ fixes all holomorphic $2$-forms on $S^n$.
It implies that $\mathcal{H}'\subset \mathcal{G}\cap \Aut(S)^n$.
Indeed all the elements in $\mathcal{G}$ can be written $\sigma \circ g$ with $\sigma \in \mathfrak{S}_n$ and $g\in \mathcal{G}\cap \Aut(S)^n$; moreover an element $\sigma \circ g$ with $\sigma \neq\id$ acts non-trivially on $H^0(\Omega_{S^n}^2)$, so such an element cannot be in $\mathcal{H}'$.

Moreover, by Proposition \ref{fonda}, we know that $\mathcal{H}'$ is generated by elements with fixed locus in codimension 2. In $\Aut(S)^n$ only the trivial complex reflexions can have a fixed locus in codimension 2. That is $\mathcal{H}'$ is generated by trivial complex reflexions and $\mathcal{H}'\subset \mathcal{H}$.
\end{proof}
\begin{lemme}\label{n>2}
Let $S$ be a K3 surface endowed with a symplectic automorphism group $G$. Let $n\geq3$ be an integer. Let $\theta:G\rightarrow G$ be an involution.
The group $\left\langle j_{\theta}(G),\mathfrak{S}_n\right\rangle$ is primitive if and only if $\theta$ is valid and $G$ is abelian. 

Moreover if $G$ is not abelian or $\theta$ is not valid, there exists a K3 surface $\Sigma$, an abelian automorphism group $G'$ on $\Sigma$, a valid involution $\theta'$ on $G'$ and a crepant bimeromorphic morphism:
\begin{equation}
\Sigma^n/\left\langle j_{\theta'}(G'),\mathfrak{S}_n\right\rangle\rightarrow S^n/\left\langle j_{\theta}(G),\mathfrak{S}_n\right\rangle.
\label{mainlemmaprimi}
\end{equation}
\end{lemme}
\begin{proof}
We set $\mathcal{G}=\left\langle j_{\theta}(G),\mathfrak{S}_n\right\rangle$. Using Lemma \ref{n>2bis}, it is enough to show that $\mathcal{G}$ does not contain a trivial complex reflexion different from the identity if and only if $\theta$ is valid and $G$ is abelian.

First, we assume that $G$ is not abelian, then there exists $g\in G$ and $h\in G$ such that $g\circ h \circ g^{-1} \circ h^{-1}\neq\id$.
Since $\mathfrak{S}_n\subset \mathcal{G}$, all the elements of the form $(\id,...,\id,b,\id,...,\id,\theta(b),\id,...,\id)$ are in $\mathcal{G}$ with $b$ in position $i$ and $\theta(b)$ in position $j$, for all $i$ and $j\neq i$ in $\left\{1,...,n\right\}$ and all $b\in G$.
Therefore, we have:
\begin{itemize}
\item[$\bullet$]
 $(g,\theta(g),\id,...,\id)\in \mathcal{G}$;
\item[$\bullet$]
$(h,\id,\theta(h),...,\id)\in \mathcal{G}$;
\item[$\bullet$]
 $(g^{-1},\theta(g^{-1}),\id,...,\id)\in \mathcal{G}$;
\item[$\bullet$]
$(h^{-1},\id,\theta(h^{-1}),...,\id)\in \mathcal{G}$.
\end{itemize}
The product of these elements provides:
\begin{equation}
(g\circ h\circ g^{-1}\circ h^{-1},\id,...,\id)\in \mathcal{G}.
\label{commutator}
\end{equation}

Now, we assume that $\theta$ is not valid. Therefore, there exists $g\in G$ such that $g\circ\theta(g)\neq\id$.
We have:
\begin{itemize}
\item[$\bullet$]
 $(g,\theta(g),\id,...,\id)\in \mathcal{G}$;
\item[$\bullet$]
$(\theta(g),\id,g,...,\id)\in \mathcal{G}$;
\item[$\bullet$]
 $(\id,\theta(g^{-1}),g^{-1},...,\id)\in \mathcal{G}$.
\end{itemize}
The product of these elements provides:
$(g\circ \theta(g),\id,...,\id)\in \mathcal{G}$.

Now, we prove the other direction; we assume that $G$ is abelian and $\theta$ is valid. Let $(g,\id,\id,...,\id)\in \mathcal{G}$ be a trivial complex reflexion, we are going to show that $g=\id$.
By definition of $\mathcal{G}$ with $\theta$ valid and $G$ abelian, we can write:
\begin{equation}
(g,\id,\id,...,\id)=\prod_{1\leq i<j\leq n}\overline{g_{i,j}},
\label{productss}
\end{equation}
where $\overline{g_{i,j}}=(\id,...,\id,g_{i,j},\id,...,\id,g_{i,j}^{-1},\id,...,\id)$ with $g_{i,j}\in G$ in position $i$ and $g_{i,j}^{-1}$ in position $j$. We set $g_{j,i}:=g_{i,j}^{-1}$.
Equation (\ref{productss}) implies that: 
\begin{equation}
g=\prod_{j=2}^n g_{1,j};
\label{prod1}
\end{equation}
and
\begin{equation}
\id=\prod_{i=2}^n\ \left(\prod_{1\leq j \leq n,\ j\neq i}^n g_{i,j}\right).
\label{prod2}
\end{equation}
Since $g_{j,i}:=g_{i,j}^{-1}$, equation (\ref{prod2}) implies that:
$$\prod_{i=2}^n g_{i,1}=\id.$$
Inverting the previous equation, we obtain that $g=\id$.

Now, we do not make any assumption on $G$ and $\theta$ and we show that there exist a K3 surface $\Sigma$ endowed with an abelian symplectic automorphism group $G'$ and a valid involution $\theta':G'\rightarrow G'$ such that (\ref{mainlemmaprimi}) is verified.
Let $[G,G]$ be the commutator subgroup of $G$. Since $\mathfrak{S}_n\subset \mathcal{G}$ and by (\ref{commutator}), we have $[G,G]^n\subset \mathcal{G}$.
Let $Y\rightarrow S/[G,G]$ be the K3 surface obtained after crepant resolution. Let $\overline{\theta}$ be the involution on $\Ab(G)$ obtained from $\theta$. Since $[G,G]^n$ is a normal subgroup of $\mathcal{G}$, the group $\overline{\mathcal{G}}:=\mathcal{G}/[G,G]^n$ induces an automorphism group of $Y^n$.
Then, we have $\mathcal{G}/[G,G]^n=\left\langle j_{\overline{\theta}}(\Ab(G),\mathfrak{S}_{n}\right\rangle$, with $\Ab(G)$ seen as an automorphism group of $Y$.
Similarly, we consider the group $H=\left\langle\left. g\circ\overline{\theta}(g)\right|\ g\in \Ab(G)\right\rangle$. We consider $\Sigma\rightarrow Y/H$ the K3 surface obtained from the crepant resolution. The group $G':=\Ab(G)/H$ induces an automorphism group of $\Sigma$ and $\overline{\theta}$ induces a valid involution $\theta'$ on $G'$. 
By construction the relation (\ref{mainlemmaprimi}) is verified.
\end{proof}
\begin{thm}\label{quotientsvalide}
Let $n\geq2$ be an integer and $S$ a K3 surface. Let $\mathcal{G}$ be a finite primitive subgroup of $\Aut(S^n)$.
Let $Y$ be a terminalization of $S^n/\mathcal{G}$. 

The complex space $Y$ is an irreducible symplectic variety with $\pi_1(Y_{reg})=0$ if and only if
there exists:
\begin{itemize}
\item
 $G$ a symplectic finite automorphism group of $S$ which is abelian if $n\geq3$ and 
\item
$\theta:G\rightarrow G$ a valid involution such that:
\end{itemize}
\begin{equation}
S^n/\mathcal{G}\simeq S^n/\left\langle j_{\theta}(G),\mathfrak{S}_n\right\rangle,
\label{main}
\end{equation}
with $\mathfrak{S}_n$ acting by permutation of the factors.
\end{thm}
\begin{proof}
If (\ref{main}) is realized, then $Y$ is an irreducible symplectic variety with $\pi_1(Y_{reg})=0$ by Proposition \ref{fonda}.
Indeed, let $(g_1,...,g_m)$ be a family of generators of $G$ such that $\theta(g_i)=g_i^{-1}$ for all $i\in\left\{1,...,m\right\}$.
Then, the group $\left\langle j_{\theta}(G),\mathfrak{S}_n\right\rangle$ is generated by the transpositions and by $(1,2)\circ j_{\theta}(g_i)$ for $i\in \left\{1,...,m\right\}$ which fix $\left\{\left.(x,g_i(x),y_3,....,y_n)\right|\ (x,y_3,...,y_n)\in S^{n-1}\right\}$.

We assume that $Y$ is an irreducible symplectic variety with $\pi_1(Y_{reg})=0$ and we are going to prove (\ref{main}).
We denote by $\pr_i:S^n\rightarrow S$ the $i$th projection.
As explained in \cite[Section 3]{Beauville} via the uniqueness of the Beauville--Bogomolov decomposition theorem, we know that $\Aut(S^n)$ is given by the natural semi-direct product between $\Aut(S)^n$ and $\mathfrak{S}_n$. We denote by $\Aut_0(S)$ the group of symplectic automorphisms of $S$. 
\begin{itemize}
\item\emph{First claim: $\mathcal{G}$ is a subgroup of $\Aut_0(S)^n\rtimes\mathfrak{S}_n$}

We denote $H^0(\Omega_{S}^2)=\C \sigma$. To have $Y$ primitively symplectic, we need that
$H^0(S^n,\Omega_{S^n}^2)^\mathcal{G}$ is generated by one non-degenerate holomorphic form. It can be written as follows: 
\begin{equation}
H^0(S^n,\Omega_{S^n}^2)^\mathcal{G}=\C\left(\sum_{i=1}^n \lambda_i\pr_i^*(\sigma)\right).
\label{2form}
\end{equation}
with $\lambda_i\neq0$ for all $i$.

Let $g\in \Aut(S)^n \cap \mathcal{G}$. We can write $g=(g_1,...,g_n)$ with $g_i\in \Aut(S)$ for all $i$. We have:
$$g^*\left(\sum_{i=1}^n \lambda_i\pr_i^*(\sigma)\right)=\sum_{i=1}^n \mu_i\lambda_i\pr_i^*(\sigma),$$
with $g_i^*(\sigma)=\mu_i\sigma$. However from (\ref{2form}), we have that $g$ fixes the form $\sum_{i=1}^n \lambda_i\pr_i^*(\sigma)$;
it implies that $\mu_i=1$ for all $i$. That is $\mathcal{G}$ is a sub-group of $\Aut_0(S)^n\rtimes\mathfrak{S}_n$.
\item\emph{Second claim: $\mathcal{G}/\Aut(S)^n \cap \mathcal{G}$ acts transitively on the $n$ factors of the product $S^n$}

We have seen that necessarily, $\Aut(S)^n \cap \mathcal{G}$ acts trivially on $H^0(S^n,\Omega_{S^n}^2)$. Hence to respect (\ref{2form}), the subgroup $ \mathcal{G}/\Aut(S)^n \cap \mathcal{G}$ has to act transitively on the factors of $S^n$. Indeed, if it was not the case, we would have at least two different orbits and therefore we would have $\rk H^0(S^n,\Omega_{S^n}^2)^\mathcal{G}\geq 2$.
\item\emph{Third claim: $\mathcal{G}$ does not contain any trivial complex reflexion different from the identity} 

Let $\mathcal{H}\subset \mathcal{G}$ be the normal subgroup generated by the trivial complex reflexions. By definition of trivial complex reflexion, we have $\mathcal{H}=\prod_{i=1}^{n} H_i$, with $H_i$ finite subgroups of $\Aut_0(S)$. Then $S^n/\mathcal{H}=\prod_{i=1}^{n} S/H_i$. We are going to show that $S/H_i\simeq S/H_1$ for all $i\in\left\{2,...,n\right\}$.

By claim 2, there exists $s_{1,i}\in \mathcal{G}$ that exchanges the first factor of $S^n$ with the $i$th factor. We can write $s_{1,i}=P\circ \left(g_1,...,g_n\right)$ with $P\in \mathfrak{S}_n$ such that $P(i)=1$ and $\left(g_1,...,g_n\right)\in \Aut_0(S)^n$. 
Let $h\in H_1$, we have:
$$s_{1,i}^{-1}\circ (h,\id,....,\id)\circ s_{1,i}=(\id,...,\id,g_i^{-1}\circ h\circ g_i,\id,...,\id),$$
with $g_i^{-1}\circ h\circ g_i$ in position $i$.
Therefore, we have $g_i^{-1} H_1 g_i\subset H_i$. Since the situation is symmetric, we also have $H_i \subset g_i^{-1}H_1 g_i $; that is $H_i=g_i^{-1} H_1 g_i$. Hence, the morphism $S\rightarrow S:\ x\mapsto g_i^{-1}(x)$ induces an isomorphism $S/H_1\simeq S/H_i$. Let $\Sigma$ be the crepant resolution of $S/H_1$. Then $\Sigma^n$ is a crepant resolution of $S^n/\mathcal{H}$. Since $\mathcal{G}$ is primitive, $\mathcal{H}$ has to be trivial.

\item\emph{Fourth claim: $\mathcal{G}/(\mathcal{G}\cap\Aut(S)^n)=\mathfrak{S}_n$}

If $\mathcal{G}\cap\Aut(S)^n$ is trivial, then by Proposition \ref{fonda}, we know that $\mathcal{G}$ is a permutation group generated by transpositions.
Moreover, by the second claim, $\mathcal{G}$ acts transitively on the factors of $S^n$. Necessarily, we have $\mathcal{G}=\mathfrak{S}_n$. This ends the proof; so for the sequel, we assume that $\mathcal{G}\cap\Aut(S)^n$ is not trivial.

Let $g\in \Aut(S)^n \cap \mathcal{G}$. Since $g$ is not a trivial complex reflexion, we can write $g=(g_1,...,g_n)$ with at least two non-trivial factors. For simplicity in the notation, we can assume without loss of generality that these factors are $g_1$ and $g_2$. Therefore $\Fix g\subset \Fix g_1\times \Fix g_2\times S^{n-2}$. Therefore, $\codim \Fix g>2$. However we have seen by Proposition \ref{fonda} that $\mathcal{G}$ needs to be generated by elements with fixed locus in codimension 2. Hence the generators of $\mathcal{G}$ are of the form $s \circ g$ with $s\in \mathfrak{S}_n$ non trivial. 
Let  $s\circ g$ be such an element with $\codim \Fix s \circ g=2$. We have $s\circ g(x_1,...,x_n)=(g_{s^{-1}(1)}(x_{s^{-1}(1)}),...,g_{s^{-1}(n)}(x_{s^{-1}(n)}))$.
If $(x_1,...,x_n)\in \Fix s \circ g$, then $g_{s^{-1}(i)}(x_{s^{-1}(i)})=x_i$ for all $i$. The element $s$ is not trivial, there is $j\in\left\{1,...,n\right\}$ such that $s(j)\neq j$.
We have:  
\begin{equation}
g_{j}(x_{j})=x_{s(j)}\ \text{and}\ g_{s(j)}(x_{s(j)})=x_{s^2(j)}.
\label{equal}
\end{equation}
 To verify $\codim \Fix s \circ g=2$, necessarily, the two previous equations have to be the same. There is only one possibility $s^2(j)=j$ and $g_{s(j)}=g_j^{-1}$. Moreover since the other equations $g_{i}(x_{i})=x_{s(i)}$ cannot be equal to (\ref{equal}), they are  necessarily trivial. That is $s(i)=i$ and $g_i=\id$ for all $i\notin \left\{j, s(j)\right\}$. Therefore $s=(j,s(j))$ and $g=(\id,...,g_j,\id,....,\id,g_j^{-1},\id,...,\id)$, with $g_j$ in position $j$ and $g_{j}^{-1}$ in position $s(j)$. We have seen that $\mathcal{G}$ is generated by elements $s \circ g$ with $s$ a transposition.
In particular $\mathcal{G}/(\mathcal{G}\cap\Aut(S)^n)$ is generated by transpositions.
Moreover by the second claim $\mathcal{G}/(\mathcal{G}\cap\Aut(S)^n)$ acts transitively on the factors. It follows that $\mathcal{G}/(\mathcal{G}\cap\Aut(S)^n)=\mathfrak{S}_n$.
 \item\emph{Fifth claim: there exist an automorphism group $\mathcal{G}'$ such that $S^n/\mathcal{G}\simeq S^n/\mathcal{G}'$ and $\mathcal{G}'\supset\mathfrak{S}_n$}

We have seen that $\mathcal{G}$ is an extension $(\mathcal{G}\cap\Aut(S)^n):\mathfrak{S}_n$. We are going to prove that we can find $\mathcal{G}\simeq\mathcal{G}'\supset \mathfrak{S}_n$.
 
Let $(s_i\circ g_{i,i})$ be a family of elements of $\mathcal{G}$ with $s_i=(j_i,k_i)$ a transposition and $g_{i,i}\in\mathcal{G}\cap\Aut(S)^n$ with an automorphism $g_i\in \Aut(S)$ in position $j_i$, the automorphism $g_i^{-1}$ in position $k_i$ and $\id$ in all other factors. We have seen in the previous claim that there exists a family $(s_i\circ g_{i,i})$ as described before such that $(s_i)$ generates $\mathfrak{S}_n$. We choose such a family $(s_i\circ g_{i,i})$ minimal; i.e. we cannot find an elements of $(s_i)$ which is generated by the other elements of the family. We are going to reorder the family $(s_i\circ g_{i,i})$ in a convenient way. We keep $s_1\circ g_{1,1}$. Then, we define the order recursively as follows. Assume that we have $(s_i\circ g_{i,i})$ for $i\leq m-1$, we explain how to choose $s_m\circ g_{m,m}$. Let $R_{m-1}:=\cup_{i=1}^{m-1} \Supp s_i$, with $\Supp s_i$ the support of the transposition. We choose $s_{m}\circ g_{m,m}$ such that $j_m\in R_{m-1}$. This is possible because if it was not the case the family $(s_i)$ would not generate $\mathfrak{S}_n$. Moreover this choice implies that $k_m\notin R_{m-1}$. Indeed, if $k_m\in R_{m-1}$, then $s_m$ would be generated by the $s_i$ for $1\leq i \leq m-1$; this would contradict our hypothesis of minimality on $(s_i\circ g_{i,i})$.
So we have our family $(s_m\circ g_{m,m})_{1\leq m \leq t}$ such that:
$$\left\langle (s_m)_{1\leq m \leq t}\right\rangle=\mathfrak{S}_n\ \text{and}\ j_m\in R_{m-1},\  k_m\notin R_{m-1},\ \forall\ 2\leq m \leq t.$$
Note that since $(s_m)_{1\leq m \leq t}$ generates $\mathfrak{S}_n$ necessarily we have $t=n-1$ and $\left\{j_1,k_1,k_2,k_3,...,k_{n-1}\right\}=\left\{1,...,n\right\}$.
Now, we are going to construct a morphism $\psi:S^n\rightarrow S^n$ such that:
\begin{equation}
\psi \circ s_m\circ g_{m,m}\circ\psi^{-1}=s_m,\ \forall\ 1\leq m \leq n-1.
\label{psi}
\end{equation}
We write $\psi$ diagonally:
$$\psi=(\psi_1,...,\psi_n).$$
We are going to define $\psi_{k_m}$ recursively. 
We set $\psi_{k_1}=g_{k_1}^{-1}$ and $\psi_{j_1}=\id$. Then, we assume that $\psi_{k_{i}}$ has been defined for $0\leq i \leq m-1$ and we provide $\psi_{k_{m}}$.
We set $\psi_{k_m}=\psi_{j_m}\circ g_{m}^{-1}$.
It remains to prove (\ref{psi}).
We compute $\psi \circ s_m\circ g_{m,m}\circ\psi^{-1}$:
We have:
\begin{align*}
&\psi \circ s_m\circ g_{m,m}\circ\psi^{-1}(x_1,...,x_n)\\
&=(x_1,....,x_{j_m-1},\psi_{j_m}\circ g_m^{-1}\circ\psi_{k_m}^{-1}(x_{k_m}),x_{j_m+1},...,x_{k_m-1},\psi_{k_m}\circ g_m\circ \psi_{j_m}^{-1}(x_{j_m}),x_{k_m+1},...,x_n).
\end{align*}
However, by construction $\psi_{j_m}\circ g_m^{-1}\circ\psi_{k_m}^{-1}=\id$. We set $\mathcal{G}'=\psi\circ\mathcal{G}\circ\psi^{-1}$ and we obtain our claim.
\item\emph{Sixth claim: $\mathcal{G}'\cap \Aut(S)^n$ is generated by elementary elements}

For simplicity in the notation, an element of the form $(\id,...,\id,g,\id,...,\id,g^{-1},\id,...,\id)$ is called an \emph{elementary element} for the sequel of the proof. 

Let $h\in \mathcal{G}'\cap \Aut(S)^n$.
As we have seen in the proof of the fourth claim, the group $\mathcal{G}'$ is generated by elements of the form $s\circ(\id,...,\id,g,\id,...,\id,g^{-1},\id,...,\id)$, with $s$ the transposition $(i,j)$; $g$ and $g^{-1}$ are in position $i$ and $j$ respectively. 
In particular, we can write:
$$h=\prod s_i \circ (\id,...,\id,g_i,\id,...,\id,g_i^{-1},\id,...,\id),$$
with $s_i$ transpositions.
That is $h=\sigma \circ \alpha$
with $\sigma$ a permutation and $\alpha$ the product of elementary elements conjugated by permutations. 
However, the conjugate by a permutation of an elementary element is an elementary element. 
Therefore $h=\sigma \circ \alpha$ with $\sigma$ a permutation and $\alpha$ a product of elementary elements. Finally since $h \in \mathcal{G}'\cap \Aut(S)^n$, we have $\sigma=\id$ and this prove our claim.

 \item\emph{Seventh claim: proof of (\ref{main})}

It remains to show that $\mathcal{G}'$ has the form prescribed by the proposition. 
Let: $$F=\left\{\left.g\in \Aut(S)\right|\ (g,g^{-1},\id,...,\id)\in \mathcal{G}'\cap \Aut(S)^n\right\}.$$ We consider $G=\left\langle F\right\rangle$. Moreover we consider $\theta:G\rightarrow G$ the involution such that $\theta(g)=g^{-1}$ for all $g\in F$. We only need to prove that $\theta$ is a well defined involution to conclude the proof; indeed, with claim 5 and 6, we will have $\mathcal{G}'=\left\langle j_{\theta}(G),\mathfrak{S}_n\right\rangle$. 
To prove that $\theta$ is well defined, it is enough to show that 
if there exists a dependency relation of elements of $F$, then it does not lead to any contradiction in the definition of $\theta$.
Let
$$h_1\circ...\circ h_k=\id,$$ with $h_i\in F$ for all $i$.
This imply:
\begin{equation}
h_k^{-1}\circ...\circ h_1^{-1}=\id.
\label{theta3}
\end{equation}
We need to show that: 
\begin{equation}
h_1^{-1}\circ...\circ h_k^{-1}=\id.
\label{theta2}
\end{equation}
We have $(h_i,h_i^{-1},\id,...,\id)\in \mathcal{G}'$ for all $i$. Hence:
$$(h_1\circ...\circ h_k,h_1^{-1}\circ...\circ h_k^{-1},\id,...,\id)\in \mathcal{G}'.$$
Therefore multiplying the first factor by (\ref{theta3}), we obtain:
$$(\id,h_1^{-1}\circ...\circ h_k^{-1},\id,...,\id)\in \mathcal{G}'.$$
Hence, by the third claim:
$h_1^{-1}\circ...\circ h_k^{-1}=\id$.
This corresponds to (\ref{theta2}).
\item\emph{Last claim: $G$ is abelian when $n\geq3$}

Since the group $\mathcal{G}$ is primitive, the group $\mathcal{G}'=\left\langle j_{\theta}(G),\mathfrak{S}_n\right\rangle$ is also primitive. Therefore by Lemma \ref{n>2}, the group $G$ is abelian.
\end{itemize}
\end{proof}
\begin{rmk}\label{fixed}
Let $S$ be a K3 surface, $G$ a finite symplectic automorphism group of $S$ and $\theta:G\rightarrow G$ an involution. Let $\mathcal{G}=\left\langle j_{\theta}(G),\mathfrak{S}_n\right\rangle$.
When $\mathcal{G}$ is primitive, we have seen in the previous proof that an element in $\left\langle j_{\theta}(G),\mathfrak{S}_n\right\rangle$ which has a fixed locus in codimension 2 is of the form: 
$$g_{i,j}:=s\circ(\id,...,\id,g,\id,...,\id,g^{-1},\id,...,\id),$$ with $s=(i,j)$ the transposition, where $g$ and $g^{-1}$ are in position $i$ and $j$ respectively. Moreover the fixed component in codimension 2 of this automorphism is:
$$\left\{\left.(y_1,...,y_{i-1},x,y_{i+1},...,y_{j-1},g(x),y_{j+1},...,y_n)\right|\ (x,y_1,...,y_n)\in S^{n-1}\right\}.$$
\end{rmk}
Theorem \ref{geneth} is a direct consequence of Theorem \ref{quotientsvalide} and Claim 3 of its proof.
\begin{proof}[Proof of Theorem \ref{geneth}]
Let $S$ be a K3 surface. Let $\mathcal{G}$ be a finite automorphism group of $S^n$. Let $Y\rightarrow S^n/\mathcal{G}$ be a terminalization which is an irreducible symplectic variety with simply connected smooth locus. Let $\mathcal{H}\subset \mathcal{G}$ be the normal subgroup generated by the trivial complex reflexions. We have seen in Claim 3 of the proof of Theorem \ref{quotientsvalide} that there exists a K3 surface $\Sigma$ such that 
$\Sigma^n\rightarrow S^n/\mathcal{H}$ is a crepant resolution. Since $\mathcal{H}$ is normal $\mathcal{G}/\mathcal{H}$ induces an automorphism group of $\Sigma^n$. By definition the group $\mathcal{G}/\mathcal{H}$ does not contain any trivial complex reflexion (different from $\id$), therefore the proof of Theorem \ref{quotientsvalide} applies to the couple $(\Sigma^n, \mathcal{G}/\mathcal{H})$. We can find $G$ a symplectic automorphism group of $\Sigma$ which is abelian if $n\geq3$ and $\theta$ a valid involution on $G$ such that:
$$\Sigma^n/\left(\mathcal{G}/\mathcal{H}\right)\simeq \Sigma^n/\left\langle j_{\theta}(G),\mathfrak{S}_n\right\rangle.$$
All Kähler terminalizations of $\Sigma^n/\left\langle j_{\theta}(G),\mathfrak{S}_n\right\rangle$ are primitive symplectic by Proposition \ref{primimi} and they are terminalization of $S^n/\mathcal{G}$; therefore Proposition \ref{bimero} provides the claim of Theorem \ref{geneth}.
\end{proof}
\subsection{Fujiki irreducible symplectic varieties}\label{Fujikivar}
After Theorem \ref{quotientsvalide}, the definition of Fujiki varieties is very natural; we recall Definition \ref{Fujikiconstruc}.
\begin{defi}
Let $S$ be a projective K3 surface and $G$ a finite symplectic automorphism group of $S$. Let $\theta:G\rightarrow G$ be an involution. Let $n\in\N\smallsetminus\left\{0,1\right\}$.
We denote by
$$S(G)^{[n]}_{\theta}\rightarrow S^n/\left\langle j_{\theta}(G),\mathfrak{S}_n\right\rangle,$$
a terminalization of $S^n/\left\langle j_{\theta}(G),\mathfrak{S}_n\right\rangle$ and we call it the \emph{Fujiki variety} of dimension $2n$ associated to $(S,G,\theta)$. 
\end{defi}
\begin{rmk}
According to Lemma \ref{n>2}, when $n\geq3$, we can assume without loss of generalities that $G$ is abelian and $\theta$ is valid.
\end{rmk}
\begin{nota}
When $G$ is abelian, there is only one valid involution which is $\inv(g)=g^{-1}$. To simplify the notation, when $G$ is abelian, we denote $S(G)^{[n]}$ instead of $S(G)^{[n]}_{\inv}$.
\end{nota}
\begin{prop}\label{primimi}
A Fujiki variety is a primitive symplectic variety.
\end{prop}
\begin{proof}
Let $Y=S(G)^{[n]}_{\theta}$ be a Fujiki variety. By definition $Y$ is Kähler.
We only have to verify that $h^1(Y,\mathcal{O}_Y)=0$. Since $S^n$ is simply connected, according to \cite[Lemma 1.2]{Fujiki}, this is also the case for $S^n/\left\langle j_{\theta}(G),\mathfrak{S}_n\right\rangle$. Since $Y$ and $S^n/\left\langle j_{\theta}(G),\mathfrak{S}_n\right\rangle$ have rational singularities, the result follows.
\end{proof}

We are ready to prove Corollary \ref{corirr} and \ref{corirr2}.
\begin{proof}[Proof of Corollary \ref{corirr}]
If $\theta$ is valid, $S(G)^{[n]}_{\theta}$ is an irreducible symplectic variety with simply connected smooth locus by Theorem \ref{quotientsvalide}.

Now, we assume that $S(G)^{[2]}_{\theta}$ is an irreducible symplectic variety with simply connected smooth locus. 
We set $\mathcal{G}=\left\langle j_{\theta}(G),\mathfrak{S}_2\right\rangle$. We apply Claim 7 of the proof of Theorem \ref{quotientsvalide} to $\mathcal{G}$.
Therefore, we can find $G'$ an automorphism group of $S$ and $\theta'$ a valid involution on $G'$ such that $\mathcal{G}=\left\langle j_{\theta'}(G'),\mathfrak{S}_2\right\rangle$. Since $\mathcal{G}\cap\Aut(S)^n=j_{\theta}(G)$, necessarily, we obtain $G=G'$ and $\theta=\theta'$. 
\end{proof}
\begin{proof}[Proof of Corollary \ref{corirr2}]
By Lemma \ref{n>2}, we can assume that $G$ is abelian and $\theta$ valid. Then the result follows directly from Theorem \ref{quotientsvalide}. 
\end{proof}
\subsection{Second Betti numbers}\label{BettiSection}
First, we provide the second Betti number of a Fujiki variety in dimension higher or equal to 6; we recall from Lemma \ref{n>2} that in this case we can assume without loss of generality that $G$ is abelian and $\theta$ valid. 
\begin{prop}\label{Betti}
Let $S$ be a K3 surface, $G$ a finite abelian symplectic automorphism group of $S$ and $\inv:G\rightarrow G$ the valid involution. Let $n\geq 3$ be an integer. 
Then:
$$b_2\left(S(G)^{[n]}_{\inv}\right)=\dim_{\C} H^2(S,\C)^G+1.$$
\end{prop}
\begin{proof}
Let $\mathcal{G}=\left\langle j_{\inv}(G),\mathfrak{S}_n\right\rangle$.
We set $\pi: S^n\rightarrow S^n/\mathcal{G}$.
According to Remark \ref{fixed}, the sub-varieties of $S^n$ of codimension 2 fixed by an element of $\mathcal{G}$ are of the form:
$$D_{i,j,g}:=\left\{\left.(y_1,...,y_{i-1},x,y_{i+1},...,y_{j-1},g(x),y_{j+1},...,y_n)\right|\ (x,y_1,...,\hat{y}_i,...,\hat{y}_j,...,y_n)\in S^{n-1}\right\},$$
with $\hat{y}_i$ and $\hat{y}_j$ meaning that $y_i$ and $y_j$ are omitted. 
We are going to prove that $\mathcal{G}$ acts transitively on this set of varieties; that is all these varieties have the same image by $\pi$.
For this purpose, we prove that there exists $f\in \mathcal{G}$ such that $f(D_{i,j,g})=D_{1,2,\id}$ for all $i$, $j$ and $g\in G$.
Indeed, we can choose $f=h\circ s$, with $s$ a permutation which exchanges $i$ with $1$ and $j$ with $2$ and:
$$h=(\id,g^{-1},g,\id,...,\id).$$
We verify that $(h\circ s)(D_{i,j,g})=D_{1,2,\mathrm{id}}$. 
The permutation $s$ acts on a point of $D_{i,j,g}$ by moving the coordinates in positions $i$ and $j$ to the first two slots. Hence
\[
s(D_{i,j,g})=D_{1,2,g}.
\]
Now, take a point $(x,g(x),z_3,\dots,z_n)\in D_{1,2,g}$. Applying $h$ gives
\[
h(x,g(x),z_3,\dots,z_n)=(x,\,g^{-1}(g(x)),\,g(z_3),\,z_4,\dots,z_n)=(x,\,x,\,g(z_3),\,z_4,\dots,z_n).
\]
Since $g$ is an automorphism of $S$, the coordinate $g(z_3)$ still ranges over all of $S$ when $z_3$ varies, while the other coordinates remain free. Therefore
\[
h(D_{1,2,g})=D_{1,2,\mathrm{id}}.
\]
Hence, we obtain that $\Sing \left(S^n/\mathcal{G}\right)$ contains only one irreducible component of codimension 2 that we denote by $D$. Moreover each variety $D_{i,j,g}$ is globally fixed by only one non trivial element in $\mathcal{G}$ which is $g_{i,j}=s\circ(\id,...,\id,g,\id,...,\id,g^{-1},\id,...,\id),$ with $s=(i,j)$ the transposition (see Remark \ref{fixed}). 
The automorphism $g_{i,j}$ is an involution; therefore a generic point in $D$ corresponds to a singularity of analytic type $\C^{2n-2}\times(\C^2/\pm\id)$. Since the singularities in codimension 4 are terminal (see Proposition \ref{terminal}), necessarily a terminalization of $S^n/\mathcal{G}$ has only one exceptional divisor. 

To conclude the proof, we only have to compute $\dim_{\C} H^2(S^n/\mathcal{G},\C)=\dim_{\C} H^2(S^n,\C)^{\mathcal{G}}$. Let $\pr_i:S^n\rightarrow S$ be the $i$-th projection. According to the Künneth formula we have that: 
$$H^2(S^n,\C)=\bigoplus_{i=1}^n\pr_i^*\left(H^2(S,\C)\right).$$ 
We consider the sub-group of $H^2(S^n,\C)$ given by $\left\{\left.\sum_{i=1}^n\pr_i^*(\alpha)\right|\ \alpha\in H^2(S,\C)^G\right\}$. This sub-group is isomorphic to $H^2(S,\C)^G$, so for simplicity in the notation we simply denote it by $H^2(S,\C)^G$. By construction, we have that $H^2(S,\C)^G\subset H^2(S^n,\C)^{\mathcal{G}}$.  
On the other hand, if an element $\sum_{i=1}^n \pr_i(\alpha_i)$ of $H^2(S^n,\C)$ is invariant under the action of $\mathfrak{S}_n$, it takes the form $\sum_{i=1}^n \pr_i(\alpha)$. Moreover, invariance under the action of $j_{\inv}(G)$ requires that $\alpha$ belongs to $H^2(S,\C)^G$.
This shows that $H^2(S^n,\C)^{\mathcal{G}}=H^2(S,\C)^G$ and therefore $\dim_{\C} H^2(S^n,\C)^{\mathcal{G}}=\dim_{\C} H^2(S,\C)^G$.
\end{proof}
When $n\geq3$, according to Lemma \ref{n>2}, there are relatively few possibilities for a Fujiki variety.  According to \cite{xiao} and excluding $K3^{[n]}$, there are 14 series. We provide their second Betti number. The group names are given in Section \ref{notanota}.
\begin{cor}\label{examples}
We have the following second Betti numbers:

$$b_2\left(S(C_2)^{[n]}\right)=15;\ \ \ b_2\left(S(C_3)^{[n]}\right)=11;\ \ \ b_2\left(S(C_2^2)^{[n]}\right)=11;\ \ \ b_2\left(S(C_4)^{[n]}\right)=9;$$
$$b_2\left(S(C_5)^{[n]}\right)=7;\ \ \ \ b_2\left(S(C_6)^{[n]}\right)=7;\ \ \ \ b_2\left(S(C_7)^{[n]}\right)=5;\ \ \ \ b_2\left(S(C_2^3)^{[n]}\right)=9;$$
$$b_2\left(S(C_2\times C_4)^{[n]}\right)=7;\ \ \ b_2\left(S(C_8)^{[n]}\right)=5;\ \ \ \ b_2\left(S(C_3^2)^{[n]}\right)=7;\ \ \ \ b_2\left(S(C_2\times C_6)^{[n]}\right)=5;$$
$$b_2\left(S(C_2^4)^{[n]}\right)=8;\ \ \ b_2\left(S(C_4^2)^{[n]}\right)=5.$$
\end{cor}
When $n=2$, the situation is much richer.
\begin{prop}\label{b2}
Let $S$ be a K3 surface and $G$ a finite symplectic automorphism group of $S$. Let $\theta:G\rightarrow G$ be an involution (not necessarily valid). Let $F:=\left\{\left.g\in G\right|\ \theta(g)=g^{-1}\right\}$. The group $G$ acts on $F$ via the action $g\cdot h=\theta(g)\circ h\circ g^{-1}$. We denote by $F/G$ the orbits of this action. Then:
$$b_2\left(S(G)^{[2]}_{\theta}\right)=\dim_{\C} H^2(S,\C)^G+\#\left(F/G\right).$$
\end{prop}
\begin{proof}
The proof is similar to the one of Proposition \ref{Betti}. As before, we have:
$$b_2\left(S(G)^{[2]}_{\theta}\right)=\dim_{\C} H^2(S,\C)^G+R,$$
with $R$ the number of irreducible components of $\Sing S^2/\mathcal{G}$ in codimension 2 with $\mathcal{G}=\left\langle j_{\theta}(G),\mathfrak{S}_n\right\rangle$.
We are going to verify that $R=\#\left(F/G\right)$.
As we have seen in Remark \ref{fixed}, a fixed surface by an element of $\mathcal{G}$ is given by: 
$$D_g:=\left\{\left.(x,g(x))\right|\ x\in S\right\},$$
for some $g\in F$. Hence the set of fixed surfaces can be identified with $F$.
Let $h\in G$, then $h(x,g(x))=(h(x),\theta(h)\circ g(x))$.
If we set $y=h(x)$, we see that:
$$h(D_g)=D_{\theta(h)\circ g \circ h^{-1}}.$$
Let $s_0$ be the involution which permutes the two factors of $S^2$.
We have $s_0(D_g)=D_{g^{-1}}=g(D_{g})$. Hence the orbits under the action of $\mathcal{G}$ on $F$ correspond to the orbits under the action of $G$ as stated in the statement of the proposition. Let $\pi:S^2\rightarrow S^2/\mathcal{G}$ be the quotient map; since two fixed surfaces have the same image by $\pi$ if and only if there are in a same orbit under the action of $\mathcal{G}$, we obtain that $R=\#(F/G)$. 
\end{proof}
\subsection{The third Betti number}
From \cite{Fujiki}, we can deduced the third Betti number in dimension 4 when $G$ is admissible (see Definition \ref{admidefi}).
\begin{prop}\label{b3}
Let $S$ be a K3 surface and $G$ a finite admissible symplectic automorphism group of $S$.
Then: $$b_3\left(S(G)^{[2]}_{\theta}\right)=0.$$
\end{prop}
\begin{proof}
As before we set $\mathcal{G}=\left\langle j_{\theta}(G),\mathfrak{S}_n\right\rangle$. According to Remark \ref{fixed},
the fixed surfaces by the action of $\mathcal{G}$ on $S^2$ are given by 
$$\left\{\left.(x,g(x))\right|\ x\in S\right\},$$
with $\theta(g)=g^{-1}$. 
Therefore all the fixed surfaces are isomorphic to $S$; hence with a trivial third Betti number. Then, we conclude the proof by applying \cite[Lemma 7.11]{Fujiki}. Note that "admissible singularities" in \cite{Fujiki} are the singularities described in Section \ref{recallFujiki} (if $G$ is admissible $S^2/\left\langle j_{\theta}(G),\mathfrak{S}_2\right\rangle$ has admissible singularities in the sense of Fujiki). 
\end{proof}
\begin{rmk}
In particular, all the irreducible symplectic orbifolds provided by Theorem \ref{main4} have trivial third Betti number.
\end{rmk}
\subsection{Fujiki relation}\label{Fujikirelationsection}
Let $S$ be a K3 surface and $G$ a finite symplectic automorphism group of $S$. Let $\theta:G\rightarrow G$ be an involution. Let $n\geq 2$ be an integer. In this section, we set $\mathcal{G}=\left\langle j_{\theta}(G),\mathfrak{S}_n\right\rangle$ and according to Lemma \ref{n>2}, we assume that $G$ is abelian and $\theta$ is valid when $n\geq3$. Let $\pi: S^n\rightarrow S^n/\mathcal{G}$ be the quotient map and $r:S(G)^{[n]}_{\theta}\rightarrow S^n/\mathcal{G}$  be a terminalization. We consider the following injection:
$$\epsilon: H^2(S,\Z)^G\rightarrow H^2\left(S^n/\mathcal{G},\Z\right), \alpha\mapsto \pi_*(\alpha\otimes 1\otimes...\otimes 1).$$ 

We recall the following well known result.
\begin{prop}[{\cite[Section 3.4]{Tim}}]\label{Fujikiformula}
Let $X$ be a primitive symplectic variety of dimension $2n$. There exists a non-degenerate quadratic form $q_X$ on $H^2(X,\Z)$ and a positive rational number $C_X$ such that:
$$\alpha^{2n}=C_Xq_X(\alpha)^n,$$
for all $n\in H^2(X,\Z)$.
\end{prop}
\begin{rmk}
The form $q_X$ is called the \emph{Beauville--Bogomolov form} and $C_X$ is called the \emph{Fujiki constant}.
\end{rmk}
In the case of Fujiki variety, we obtain the following result.
\begin{prop}\label{Fujiki}
Let $q_{S(G)^{[n]}_{\theta}}$ and $C_{S(G)^{[n]}_{\theta}}$ be respectively the Beauville--Bogomolov quadratic form and the Fujiki constant of $S(G)^{[n]}_{\theta}$. Then:
$$C_{S(G)^{[n]}_{\theta}}q_{S(G)^{[n]}_{\theta}}(r^*(\epsilon(\alpha)))^n=\frac{(2n)!\left|G\right|}{n!2^n}\left(\left|G\right|^{2n-3}(n-1)!^2(\alpha^2)\right)^{n},$$
for all $\alpha\in H^2(S,\Z)$.
\end{prop}
In order to prove this proposition, we need to compute the cardinality of $\mathcal{G}$.
\begin{lemme}\label{card}
Under the assumptions considered in this section, it holds that:
$$\left|\mathcal{G}\right|=\left|G\right|^{n-1}n!.$$
\end{lemme}
\begin{proof}
First, note that $\mathcal{G}$ does not contain any trivial complex reflexion different from $\id$. It is immediate when $n=2$ and a consequence of Lemmas \ref{n>2bis} and \ref{n>2} when 
$n\geq3$. 
We have $\mathcal{G}=(G^n\cap\mathcal{G})\rtimes \mathfrak{S}_n$. Let $g=(g_1,...,g_n)\in G^n\cap\mathcal{G}$, its first $n-1$ factors $g_i$ can be chosen freely in $G$. 
However, the last factor $g_n$ is determined by the $n-1$ first factors because $\mathcal{G}$ does not contain any trivial complex reflection different from $\id$. Therefore $\#(G^n\cap\mathcal{G})=\#(G^{n-1})$. The result follows.
\end{proof}
\begin{proof}[Proof of Proposition \ref{Fujiki}]
Let $\alpha\in H^2(S,\Z)^G$.
We denote $\widetilde{\alpha}=\sum_{i=1}^n\pr^*_i(\alpha)$.
By \cite[Lemma 3.6]{Lol4} and Lemma \ref{card}:
$$\pi_*(\widetilde{\alpha})^{2n}=\left(\left|G\right|^{n-1}n!\right)^{2n-1}\pi_*\left(\widetilde{\alpha}^{2n}\right).$$
We have $\pi_*(\sum_{i=1}^n\pr^*_i(\alpha))=\sum_{i=1}^n\pi_*(\pr^*_i(\alpha))=n\epsilon(\alpha)$; therefore:
$$n^{2n}\epsilon(\alpha)^{2n}=\left(\left|G\right|^{n-1}n!\right)^{2n-1}\frac{(2n)!}{2^n}(\alpha^2)^n.$$
We obtain:
$$\epsilon(\alpha)^{2n}=\frac{(2n)!\left(\left|G\right|^{n-1}(n-1)!\right)^{2n-1}}{n2^n}(\alpha^2)^n.$$
So:
$$\epsilon(\alpha)^{2n}=\frac{(2n)!\left|G\right|}{n!2^n}\left(\left|G\right|^{2n-3}(n-1)!^2(\alpha^2)\right)^{n}.$$
Finally, the Fujiki relation (Proposition \ref{Fujikiformula}) gives the result.
\end{proof}
\begin{rmk}
If $G$ is trivial, the morphism $i:H^2(S,\Z)\rightarrow H^2(S^{[n]},\Z)$ introduced in \cite{Beau} is given by $\frac{r^*\circ \epsilon}{(n-1)!}$. 
Hence, we obtain the well known Fujiki constant of $S^{[n]}$. If $G$ is not trivial, it is not clear if $\frac{r^*\circ \epsilon}{(n-1)!}$ is integral.
\end{rmk}
More generally, when the variety comes from a partial resolution of any quotient, we can still provide a relation between the Fujiki constants in a similar way as Proposition \ref{Fujiki} (it will be used in Section \ref{verif}).
\begin{prop}\label{Fujiki2}
Let $X$ be a primitive symplectic variety of dimension $2n$. Let $H$ be a finite symplectic automorphism group of $X$. 
Let $r:Y\rightarrow X/H$ be a proper bimeromorphic morphism such that $Y$ is a primitive symplectic variety. Let $\pi:X\rightarrow X/H$ be the quotient map.
Then:
$$C_Yq_Y(r^*\pi_*(\alpha))^n=\left|H\right|^{2n-1}C_Xq_X(\alpha)^n.$$
\end{prop}
\begin{proof}
Let $\alpha\in H^2(X,\Z)^{H}$. By \cite[Lemma 3.6]{Lol4}, we have:
\begin{equation}
\pi_*(\alpha)^{2n}=\left|H\right|^{2n-1}\pi_*\left(\alpha^{2n}\right).
\label{Fujikiequa}
\end{equation}
By Proposition \ref{Fujikiformula},
we have:
$$r^*\pi_*(\alpha)^{2n}=C_Yq_Y(r^*\pi_*(\alpha))^n\ \text{and}\ \alpha^{2n}=C_Xq_X(\alpha)^n.$$
Therefore (\ref{Fujikiequa}) becomes:
$$C_Yq_Y(r^*\pi_*(\alpha))^n=\left|H\right|^{2n-1}C_Xq_X(\alpha)^n.$$
\end{proof}
Now, we are ready to prove Proposition \ref{Higherprointro}.
\begin{proof}[Proof of Proposition \ref{Higherprointro}]
As we have seen in Lemma \ref{n>2}, when $n\geq3$, the Fujiki varieties are given by 14 series corresponding to the 14 possible abelian groups of a K3 surface. It remains to show that these 14 series are independent under deformation.

We recall from Corollary \ref{corirr2} that a Fujiki variety in dimension higher or equal to 6 is irreducible symplectic; moreover by construction it has 
$\Q$-factorial and terminal singularities. 
Therefore by Proposition \ref{locallytriv}, if two such varieties are deformation equivalent they have the same Fujiki constant and the same second Betti number.
Using Corollary \ref{examples}, we are going to compare the Fujiki constants of all the varieties with the same second Betti number.
If two varieties $S(H_1)^{[n]}$ and $S(H_2)^{[n]}$ would have the same Fujiki constant, Proposition \ref{Fujiki} would provide that: 
$$\frac{\left|H_1\right|\left(\left|H_1\right|^{2n-3}\right)^{n}}{q_{S(H_1)^{[n]}_{\theta}}(r^*(\epsilon(\alpha)))^n}=\frac{\left|H_2\right|\left(\left|H_2\right|^{2n-3}\right)^{n}}{q_{S(H_2)^{[n]}_{\theta}}(r^*(\epsilon(\alpha)))^n},$$
for some $\alpha\in H^2(S,\Z)$ such that $\alpha^2\neq0$.
Hence $\left(\frac{|H_1|}{|H_2|}\right)^{1/n}$ is a rational number. 
According to Corollary \ref{examples}, all the fractions $\frac{|H_1|}{|H_2|}$ that we have to consider are:
$$\frac{3}{4};\ \ \frac{1}{2};\ \ \frac{5}{6};\ \ \frac{5}{8};\ \ \frac{5}{9};\ \ \frac{2}{3};\ \ \frac{8}{9};\ \ \frac{7}{8};\ \ \frac{7}{12};\ \ \frac{7}{16};\ \ \frac{2}{3}.$$
However, for any $n\geq3$, the $n$th roots of all the previous numbers are irrational.

\end{proof}

\subsection{Some bimeromorphisms in dimension 4}\label{deformation}
Two different involutions $\theta$ on a given group can lead to two bimeromorphic Fujiki varieties. 
\begin{defi}\label{equiinv}
Let $G$ be a finite symplectic automorphism group of a K3 surface $S$. Let $\theta_1$ and $\theta_2$ be two involutions on $G$. We say that $\theta_1$ and $\theta_2$ are \emph{equivalent} if there is an isomorphism 
$$S^2/\left\langle j_{\theta_1}(G),\mathfrak{S}_2\right\rangle \rightarrow S^2/\left\langle j_{\theta_2}(G),\mathfrak{S}_2\right\rangle.$$
\end{defi}
\begin{rmk}
Let $G$ be a finite symplectic automorphism group of a K3 surface $S$.
According to Proposition \ref{bimero}, if $\theta_1$ and $\theta_2$ are equivalent the varieties $S(G)^{[2]}_{\theta_1}$ and $S(G)^{[2]}_{\theta_2}$ are equivalent by deformation.
\end{rmk}
The following results will be essential to prove Theorem \ref{main4} in Section \ref{examplescalculs}.
\begin{prop}\label{invodeform}
Let $G$ be a finite symplectic automorphisms group of a K3 surface $S$. Let $\widetilde{G}$ be another automorphism group of $S$ (not necessarily symplectic nor finite) such that $G\subset \widetilde{G}$. Let $\theta_1$ and $\theta_2$ be two involutions on $G$.
Let $\mathcal{G}_i=\left\langle j_{\theta_i}(G),\mathfrak{S}_2\right\rangle$, with $i\in\left\{1, 2\right\}$. 
We assume that there exists $h_1\in \widetilde{G}$ and $h_2\in \widetilde{G}$ such that:
\begin{itemize}
\item[(i)]
$h_1$ and $h_2$ acts on $G$ by conjugation;
\item[(ii)]
$h_1\circ h_2^{-1}\in G$;
\item[(iii)]
$\theta_2(h_1\circ h_2^{-1})=h_2\circ h_1^{-1}$;
\item[(iv)]
$\theta_2(h_1\circ g \circ h_1^{-1})=h_2\circ \theta_1(g)\circ h_2^{-1}$ for all $g\in G$.
\end{itemize}
Then $\theta_1$ and $\theta_2$ are equivalent.
\end{prop}
\begin{proof}
We consider the following isomorphism:
$$\Phi:\xymatrix@R0pt{S^2\ar[r]&S^2\\
(x,y)\ar[r]&(h_1(x),h_2(y))}$$
and we show that it leads to an isomorphism between the quotients.
Let $(x,y)\in S^2$ and $g\in G$, we need to show that $\Phi(x,y)$, $\Phi(y,x)$ and $\Phi(g(x),\theta_1(g)(y))$ are in the same orbit under the action of $\mathcal{G}_2$. We have:
\begin{itemize}
\item
$\Phi(x,y)=(h_1(x),h_2(y))$;
\item
$\Phi(y,x)=(h_1(y),h_2(x))$;
\item
$\Phi(g(x),\theta_1(g)(y))=(h_1\circ g(x),h_2\circ\theta_1(g)(y))$.
\end{itemize}
Let $\mathfrak{S}_2=\left\langle s_0\right\rangle$. We have:
$$h_1\circ h_2^{-1}\circ s_0\Phi(y,x)=h_1\circ h_2^{-1}(h_2(x),h_1(y))=(h_1(x),\theta_2(h_1\circ h_2^{-1})\circ h_1(y))=(h_1(x),h_2(y)).$$
So $\Phi(x,y)$ and $\Phi(y,x)$ are in the same orbit.
Similarly, we need to show that $(h_1(x),h_2(y))$ and $(h_1\circ g(x), h_2\circ \theta_1(g)(y))$ are in the same orbit under the action of $\mathcal{G}_2$. If we consider the action of $h_1\circ g\circ h_1^{-1}$ on $(h_1(x),h_2(y))$, we only need to verify that:
$$\theta_2(h_1\circ g\circ h_1^{-1})\circ h_2=h_2\circ \theta_1(g).$$
That is:
$$\theta_2(h_1\circ g\circ h_1^{-1})=h_2\circ \theta_1(g)\circ h_2^{-1},$$
which is exactly our assumption.
\end{proof}

With the previous proposition, we can find many equivalent involutions. We provide an example with Corollary \ref{classvalid} below. Before, we recall that the set of \emph{inner automorphisms} of a group $G$ 
is given by the image of the following map:
$$\xymatrix@R0pt{G\ar[r]&\Aut(G)\\
g\ar[r]&\left(h\mapsto ghg^{-1}\right).}$$ 
\begin{cor}\label{classvalid}
Let $G$ be a finite symplectic automorphism group of a K3 surface $S$. We assume that $G$ has a trivial center. 
Then all inner involutions of $G$ are equivalent.
\end{cor}
\begin{proof}
Let $\theta_1$ and $\theta_2$ be two inner involutions on $G$. There exists $f_1$ and $f_2$ of order 2 such that $\theta_1(g)=f_1\circ g \circ f_1$ and 
$\theta_2(g)=f_2\circ g \circ f_2$ for all $g\in G$.
We choose $h_2=f_2\circ f_1$ and $h_1=\id$ and we conclude with Proposition \ref{invodeform}. 
\end{proof}
\begin{rmk}
If $G$ has a non trivial center, there is a complication because $f_1$ and $f_2$ are not necessarily of order 2.
\end{rmk}
																				
\section{Computing the singularities in dimension 4}\label{sing}
\subsection{Overview of the section}
Let $S$ be a K3 surface, $G$ a symplectic automorphism group of $S$ and $\theta$ an involution of $G$ (not necessarily valid). 
In all this section, we set $\mathcal{G}=\left\langle j_{\theta}(G),\mathfrak{S}_2\right\rangle$.

In this section, we propose a method to compute the singularities of $S(G)_{\theta}^{[2]}$ when the quotient $S/G$ only has singularities of type  $A1$, $A2$, $A3$ and $A5$; when $S/G$ has other kind of singularities, a terminalization of $S^2/\mathcal{G}$ is not known. The goal is to characterize the singularities of $S(G)_{\theta}^{[2]}$ uniquely in term of $G$ and $\theta$. The number of singularities will be obtained in term of the cardinality of some subsets of $G$. The cardinality of these subsets can then be computed via a computer program (see \cite[Singularities]{github} and Section \ref{singProgram}).

The section is organized as follows. In Section \ref{recallFujiki}, we recall the terminalizations of some singularities of $S^2/\mathcal{G}$ that have been provided by Fujiki in \cite[Section 7]{Fujiki}. In Section \ref{specific}, we provide the notion of \emph{specific fixed points} of an automorphism; this notion is fundamental to determine the singularities of $S(G)_{\theta}^{[2]}$. In Sections \ref{onS} and \ref{onSxS}, we study respectively the specific fixed points for the action of $G$ on $S$ and for the action of $j_{\theta}(G)$ on $S\times S$. We deduce in Section \ref{singFinal} the singularities of $S(G)_{\theta}^{[2]}$. Since the computation of singularities is quite technical, we propose a method to verify the results in Section \ref{verif}; in particular, this method reveals mistakes of computation in \cite{Fu} and \cite{Fujiki} (a correction will be provided in Section \ref{examplescalculs}).

\subsection{Some known local terminalizations of symplectic singularities in dimension 4}\label{recallFujiki}
In \cite[Section 7]{Fujiki}, Fujiki considers several kinds of quotient singularities that he calls $N_k$ and $M_k$ for $k\in \left\{2,3,4,6\right\}$.
Fujiki call these singularities \emph{admissible singularities}.
They are defined as follows:
$$N_k=\C^4/\left\langle g_k,s_0\right\rangle\ \ \text{and}\ \ M_k=\C^4/\left\langle h_k,s_0\right\rangle,$$
with $$g_k=\diag(\xi_k,\xi_k^{-1},\xi_k^{-1},\xi_k),\ \ h_k=\diag(\xi_k,\xi_k^{-1},\xi_k,\xi_k^{-1})\ \ \text{and}\  \ s_0=\begin{pmatrix}
																													0 & 0 & 1 & 0\\
																													0 & 0 & 0 & 1\\
																													1 & 0 & 0 & 0\\
																													0 & 1 & 0 & 0
																													\end{pmatrix}.$$
\begin{rmk}
For $k>2$, note that $\left\langle h_k,s_0\right\rangle$ is commutative, however $\left\langle g_k,s_0\right\rangle$ is not.
\end{rmk}
\begin{prop}[\cite{Fujiki}, Section 7]\label{mini}
There exists terminalizations $\widetilde{N}_k$ and $\widetilde{M}_k$ of $N_k$ and $M_k$ respectively such that:
\begin{itemize}
\item
$\widetilde{N}_2=\widetilde{M}_2$, $\widetilde{N}_3$ and $\widetilde{N}_4$ are smooth;
\item
$\widetilde{N}_6$ has one isolated singularities of analytic type $\C^4/-\id$;
\item
$\widetilde{M}_3$ has 2 isolated singularities of analytic type $\C^4/g_3$;
\item
$\widetilde{M}_4$ has 4 isolated singularities of analytic type $\C^4/-\id$;
\item
$\widetilde{M}_6$ has 4 isolated singularities of analytic type $\C^4/g_3$.
\end{itemize}
\end{prop}	
We refer to \cite[Section 7]{Fujiki} for the concrete construction of these terminalizations.			

The next important remark has been communicated to the author by Mirko Mauri.
\begin{rmk}
Let $Y\rightarrow S^2/\left\langle j_{\theta}(G),\mathfrak{S}_2\right\rangle$ be a terminalization obtained from gluing the different terminalizations of Proposition \ref{mini}.
As explained in \cite[Proposition 7.9]{Fujiki} $Y$ is not necessarily Kähler. 
However the teminalizations of Proposition \ref{mini} can still be used to compute the singularities of $S(G)_{\theta}^{[2]}$. Indeed according to \cite[Corollary 25]{Namikawa3}, the varieties $Y$ and $S(G)_{\theta}^{[2]}$ have the same singularities.
\end{rmk}																								
\subsection{Definitions, hypothesis and notations}\label{specific}
\begin{defi}
Fix a set $X$ and a group $G$ acting on $X$. 
Let $x\in X$ and $G_x$ be its stabilizer under the $G$-action. We say that $x$ is a \emph{specific fixed point} of $g$ (related to the action of $G$ on $X$)
if:
$$G_x=\left\langle g\right\rangle.$$
\end{defi}
Let $G$ be an automorphism group of a K3 surface $S$.
We consider the specific fixed points only in the two following configurations: 
\begin{itemize}
\item
$G$ acts on $S$; 
\item
$j_{\theta}(G)$ acts on $S\times S$ 
(see (\ref{jtheta}) for the definition of $j_{\theta}$). 
\end{itemize}
\begin{nota}
To avoid too complicated notation, when we consider the action of the group $\mathcal{G}$ on $S\times S$, the subgroup $j_{\theta}(G)$ is denoted by $G$ and an element $j_{\theta}(g)$ is denoted by $g$.
\end{nota}
We recall Definition \ref{admissiblegroup} from the introduction.
\begin{defi}\label{admissiblegroup2}
A finite symplectic automorphism group $G$ on a K3 surface $S$ is said 
\emph{admissible} if $S/G$ has only singularities of type $A_1$, $A_2$, $A_3$ or $A_5$.
\end{defi}
When $G$ is admissible, we will see that $S^2/\left\langle j_{\theta}(G),\mathfrak{S}_2\right\rangle$ will have only codimension 2 singularities of analytic type $N_k$ or $M_k$ with $k\in \left\{2,3,4,6\right\}$.
For this reason, in order to determine the terminalization of $S^2/\mathcal{G}$ using Proposition \ref{mini}, we make the following assumption.
\begin{hypo}\label{hypo}
In all this section, we assume that $G$ is an admissible group. In particular, $G$ only contains automorphisms of order $2,3,4$ or $6$.
\end{hypo}
\begin{defi}\label{type}
A surface in $S\times S$ is called a \emph{fixed surface} if there exists an element in $\mathcal{G}$ which fixes all its points.
\end{defi}

\begin{nota}\label{fixednota}
Let $g\in G\smallsetminus \left\{\id\right\}$. We set:
\begin{itemize}
\item[(i)]
$\SpFix_S g$ the set of specific fixed points of $g$ for the action of $G$ on $S$;
\item[(ii)]
$n(g):=\#\SpFix_S g$;
\item[(iii)]
$\SpFix_{S\times S} g$ the set of specific fixed points of $g$ for the action of $G$
on $S\times S$;
\item[(iv)]
$F:=\left\{\left.s\in G\right|\ \theta(s)=s^{-1}\right\}$.
\item[(v)]
$s_0$ the automorphism of $S\times S$ which exchange the two copies of $S$:
$s_0(x,y)=(y,x)$;
\item[(vi)]
$Z:=S\times S/\mathcal{G}$;
\item[(vii)]
$\Sigma:=\left\{\left.(x,s(x))\right|\ x\in S,\ s\in G\right\}$.
\end{itemize}
\end{nota}
\begin{defi}\label{typefixed}
Let $h_1$ and $h_2$ in $G\smallsetminus\left\{\id\right\}$.
A point $(x,y)\in S\times S$ is said of \emph{type} $(h_1,h_2)$ if $x\in \SpFix_S h_1$ and $y\in \SpFix_S h_2$.
\end{defi}
This definition is important for the following reason explained in the next statements. 
\begin{remark}\label{Rapporteurbete}
It is important to emphasize that if $(x,y)\in S\times S$ is fixed by some $g\in G$, 
then $x$ is fixed by $g$ and $y$ is fixed by $\theta(g)$ ; however $x$ is not necessarily a specific fixed point of $g$ nor $y$ is 
a specific fixed point of $\theta(g)$. Indeed, it can be that $\langle g\rangle=\langle g_i\rangle \cap \langle g_j\rangle$ with $\langle g_i \rangle\ne \langle g_j \rangle$ (for instance with $g$ of order 3 and $g_i$, $g_j$ two different automorphisms of order 6) ; then we can have $(x,y)\in\SpFix_{S\times S} g$ with $(x,y)$ of type $(g_i,\theta(g_j))$.
\end{remark}
To help the reader to understand, we illustrate this configuration in the following example.
\begin{ex}\label{Rapporteurbete2}
Let $G=\langle a\rangle\times\langle b\rangle\simeq C_6\times C_2$, with $\mathrm{ord}(a)=6$, $\mathrm{ord}(b)=2$,
and take $\theta=\mathrm{id}$ for simplicity.
Set
\[
g_i := a,\qquad g_j := ab,
\]
so both $g_i$ and $g_j$ have order $6$ but generate \emph{different} cyclic subgroups:
\[
\langle g_i\rangle=\{a^k\mid k\in\mathbb Z/6\},\qquad
\langle g_j\rangle=\{a^kb^k\mid k\in\mathbb Z/6\}.
\]
A direct check shows
\[
\langle g_i\rangle\cap \langle g_j\rangle
= \langle a^2\rangle.
\]
Let $g:=a^2$, which has order $3$. Choose $x\in S$ with stabilizer $G_x=\langle g_i\rangle$
(i.e.\ $x$ is a \emph{specific fixed point} of $g_i$), and choose $y\in S$ with stabilizer
$G_y=\langle \theta(g_j)\rangle=\langle g_j\rangle$ (i.e.\ $y$ is a \emph{specific fixed point} of $\theta(g_j)$).
Then $g\in \langle g_i\rangle$ and $g\in \langle g_j\rangle$, hence $g$ fixes both $x$ and $y$, so
\[
g\cdot(x,y)=(x,y).
\]
In particular $(x,y)$ is \emph{fixed by $g$} and is of type $(g_i,\theta(g_j))$ with $\langle g_j\rangle\neq\langle g_i\rangle$,
while
\[
\langle g\rangle=\langle g_i\rangle\cap\langle g_j\rangle.
\]
Note more precisely that $(x,y)\in \SpFix_{S\times S} g$ because according to Hypothesis \ref{hypo} neither $x$ nor $y$ can be fixed by both $g_i$ and $g_j$. 
\end{ex}
\begin{lemme}\label{differenttypeofpoint}
Let $g\in G\smallsetminus \left\{\id\right\}$. Then $(x,y)\in \SpFix_{S\times S} g$ if and only if there exists $g_i$ and $g_j$ in $G$ such that $(x,y)$ is of type $(g_i,\theta(g_j))$ and $\left\langle g_i\right\rangle\cap \left\langle g_j\right\rangle=\left\langle g\right\rangle$.
\end{lemme}
\begin{proof}
For any $(x,y)\in S\times S$, the stabilizer 
for the action $G\curvearrowright S\times S$, $g\cdot(x,y)=(g(x),\theta(g)(y))$, is
\[
G_{(x,y)}=\{\,h\in G\mid h(x)=x,\ \theta(h)(y)=y\,\}
= G_x \cap \theta^{-1}(G_y).
\]
If $(x,y)$ is of type $(g_i,\theta(g_j))$ in the sense of Definition~4.10, i.e. 
$G_x=\langle g_i\rangle$ and $G_y=\langle \theta(g_j)\rangle$, then
\[
G_{(x,y)} \;=\; \langle g_i\rangle \cap \theta^{-1}\big(\langle \theta(g_j)\rangle\big)
\;=\; \langle g_i\rangle \cap \langle g_j\rangle.
\]

\smallskip
\emph{($\Rightarrow$)} If $(x,y)\in \SpFix_{S\times S}(g)$, then by definition 
$G_{(x,y)}=\langle g\rangle$. 
Writing $G_x=\langle g_i\rangle$ and $G_y=\langle \theta(g_j)\rangle$
(cyclic by Hypothesis \ref{hypo}), the identity above gives 
$\langle g\rangle=G_{(x,y)}=\langle g_i\rangle\cap \langle g_j\rangle$, and $(x,y)$ is of type 
$(g_i,\theta(g_j))$. 

\smallskip
\emph{($\Leftarrow$)} Conversely, if $(x,y)$ is of type $(g_i,\theta(g_j))$ and
$\langle g_i\rangle\cap \langle g_j\rangle=\langle g\rangle$, then the same identity yields
$G_{(x,y)}=\langle g\rangle$, i.e. $(x,y)\in \SpFix_{S\times S}(g)$.

This proves the equivalence.
\end{proof}
\subsection{Specific fixed points for the action on $S$}\label{onS}
In this section we study the specific fixed points of an automorphism $g\in G$ acting on $S$. We recall that $G$ is assumed to be admissible in Hypothesis \ref{hypo}.
\begin{lemme}\label{speci46}
Let $g\in G$. 
If $g$ has order $4$ (resp. has order $6$), then $g$ has $4$ (resp. $2$) specific fixed points. 
\end{lemme}
\begin{proof}
For instance by \cite[Chapter 15 section 1]{HuybrechtsK3}, we know that a symplectic automorphism of order 4 on a K3 surface has 4 fixed points and a symplectic automorphism of order 6 has 2 fixed points. Moreover if such point was not specific, it would contradict our Hypothesis \ref{hypo} on $G$.
\end{proof}
\begin{lemme}\label{speci3}
Let $g\in G$ be an automorphism of order 3. Let $k_6(g)$ be the number of different cyclic sub-groups of $G$ of order 6 which contain $g$.
Then $6-2k_6(g)$ is a non-negative integer which corresponds to the number of specific fixed points of $g$.
\end{lemme}
\begin{proof}
Let $(\left\langle g_i\right\rangle)_{i\in \left\{1,...,k_6(g)\right\}}$ be the cyclic groups of order 6 containing $g$.
By Lemma \ref{speci46}, we know that each automorphism $g_i$ has 2 specific fixed points. 
Since $g_i^2=g$, we have:
$$\Fix g\supset \bigsqcup_{i=1}^{k_6(g)}\Fix g_i,$$
the union being disjoint because of Hypothesis \ref{hypo}. Note that $\#\Fix g=6$ by \cite[Chapter 15 section 1]{HuybrechtsK3}. 
Let $x\in \Fix g\smallsetminus \bigsqcup_{i=1}^{k_6(g)}\Fix g_i$; it remains to shows that $x$ is a specific fixed point of $g$. If it is not the case then $\left|G_x\right|>3$. By hypothesis \ref{hypo}, it follows that $\left|G_x\right|=6$ and $G_x$ is cyclic. Therefore, there exists $h\in G$ such that $G_x=\left\langle h\right\rangle$. Moreover since $g\in G_x$, we have $h^2=g$ or $h^{-2}=g$. Hence $h$ or $h^{-1}$ is one of the $g_i$, $i\in\left\{1,...,k_6(g)\right\}$. This is a contradiction since $x\in \Fix g\smallsetminus \bigsqcup_{i=1}^{k_6(g)}\Fix g_i$. 
\end{proof}
\begin{lemme}\label{speci2}
Let $g\in G$ be an automorphism of order 2. Let $k_6(g)$ (resp. $k_4(g)$) be the number of different cyclic sub-groups of $G$ of order 6 (resp. 4) which contain $g$.
Then $8-2k_6(g)-4k_4(g)$ is a non-negative integer which corresponds to the number of specific fixed point of $g$.
\end{lemme}
\begin{proof}
The proof is identical to the one of Lemma \ref{speci3}.
\end{proof}
\subsection{Specific fixed points for the action on $S\times S$}\label{onSxS}
\subsubsection{Preliminaries}
The goal of this section is to understand the fixed points of the elements of $\mathcal{G}$.
First, we highlight that there are two different kinds of fixed points on $S\times S$.
We recall Remark \ref{fixed} in dimension 4.
\begin{lemme}\label{surfaces}
The set of fixed surfaces in $S\times S$ is in bijection with $F=\left\{\left.s\in G\right|\ \theta(s)=s^{-1}\right\}$ (see Notation \ref{fixednota} (iv)).
A fixed surface in $S\times S$ is given by:
$$\left\{\left.(x,s(x))\right|\ x\in S\right\},$$
with $s\in F$ and it is fixed by $s_0\circ s$.
\end{lemme}
We can be slightly more general and determine all the points $(x,y)\in S^2$ which are fixed by an element of the form $s_0\circ s$ for any $s\in G$.
\begin{lemme}\label{surfaces0}
Let $s\in G$.
A point $(x,y)\in S\times S$ is fixed by $s_0\circ s$ if and only if $y=s(x)$ and $x$ is a fixed point of $\theta(s)\circ s$.
\end{lemme}
\begin{proof}
$(\Rightarrow)$ If $s_0\circ s(x,y)=(x,y)$ then $(\theta(s)(y),s(x))=(x,y)$, hence $y=s(x)$ and $\theta(s)(y)=x$, i.e. $x$ is fixed by $\theta(s)\circ s$.\\
$(\Leftarrow)$ Conversely, if $y=s(x)$ and $x$ is fixed by $\theta(s)\circ s$, then
$s_0\circ s(x,y)=(\theta(s)(s(x)),s(x))=(x,y)$.
\end{proof}
\begin{ex}
We consider again the group in Example \ref{Rapporteurbete2} to illustrate the previous lemmas.
We choose $s=a$.
Then $\theta(s)\circ s=a^2$ is an automorphism of order $3$.
Choose $x\in \mathrm{Fix}(ab)$ and set $y:=s(x)=a(x)$. We check the two conditions of Lemma \ref{surfaces0}:
\[
y=s(x)\quad\text{and}\quad \theta(s)\circ s(x)=a^2(x)=x.
\]
Hence $s_0\circ s(x,y)=(x,y)$, i.e.\ $(x,y)$ is a fixed point of $s_0\circ s$.
\end{ex}
The previous lemma shows that a point in $S^2$ which is fixed by an element of the form $s_0\circ s$ is written $(x,s(x))$ with $x\in S$. It shows that there are two kinds of fixed points in $S^2$:
\begin{itemize}
\item
The points not in $\Sigma$ (see Notation \ref{fixednota} (iv)).
According to Lemma \ref{surfaces0}, they can only be fixed by elements in $G$. According to Hypothesis \ref{hypo}, their stabilizers are cyclic groups. 
\item
The points in $\Sigma$. They can potentially be fixed by elements of the form $s_0\circ s$. So their stabilizers are more complicated. Moreover these points can be in fixed surfaces according to Lemma \ref{surfaces} and therefore lead to the singularities $N_k$ and $M_k$ (see Section \ref{recallFujiki}).
\end{itemize}
We need to study these two kinds of fixed points separately. To determine the singularities, our analysis requires knowing both the number of fixed points and the structure of their stabilizers.
In order to achieve this goal, we organize this section as follows. 
\begin{itemize}
\item
In Section \ref{inSigma}, we count the points in $\Sigma$ fixed by an element in $G\smallsetminus \left\{\id\right\}$. 
\item
In Section \ref{notinsigma}, we count the fixed points which are not in $\Sigma$.
\item
In Section \ref{pointsinsigma}, we determine the stabilizers of the fixed points in $\Sigma$.
\end{itemize}
Then our algorithm works as follows. For each element in $g\in G\smallsetminus\left\{\id\right\}$ we consider its specific fixed points and determine their contribution to the singularities (see Section \ref{singFinal}).  

We summarize the organization of the section in the following tree. 
\begin{figure}[h!]
\centering
\begin{tikzpicture}[
  level 1/.style={sibling distance=50mm},
  level 2/.style={sibling distance=30mm},
  every node/.style={align=center, font=\small},
  edge from parent/.style={->,draw},
  >=latex
]

\node[font=\bfseries] {Specific fixed points of $g$ in $S\times S$}
  child { node {Points in $\Sigma$\\
	see Section \ref{inSigma}}
    child { node {type $(g,\theta(g))$} child { node {different stabilizers\\
		see Section \ref{pointsinsigma}} }}
    child { node {type $(g_i,\theta(g_j))$}  child { node {different stabilizers\\
		see Section \ref{pointsinsigma}} }
  }}
  child { node {Points not in $\Sigma$}
    child { node {different types\\
		see Lemma \ref{fixedG}} child { node {cyclic stabilizers: $\left\langle g\right\rangle$.} }
  }};

\end{tikzpicture}
\end{figure}
\subsubsection{Fixed points in $\Sigma$}\label{inSigma}
First, we determine the different possibilities for a fixed point in $\Sigma$.  
\begin{lemme}\label{order=}
Let $(x,y)\in\Sigma$. We assume that there exists $g_i$ and $g_j$ in $G$ such that $(x,y)$ is of type $(g_i,\theta(g_j))$. Then $g_i$ and $g_j$ have the same order.
\end{lemme}
\begin{proof}
Since $(x,y)\in\Sigma$, there exists $s\in G$ such that $y=s(x)$.
Since $\theta(g_j)$ fixes $s(x)$ then $s^{-1}\circ \theta(g_j)\circ s$ fixes $x$. Hence $s^{-1}\circ \theta(g_j)\circ s\in\left\langle g_i\right\rangle$. That is $\mathcal{O}(g_j)\leq \mathcal{O}(g_i)$. Similarly since $g_i$ fixes $x$, the element $s\circ g_i \circ s^{-1}$ fixes $s(x)$. Therefore, $s\circ g_i \circ s^{-1}\in \left\langle \theta(g_j)\right\rangle$. That is $\mathcal{O}(g_i)\leq \mathcal{O}(g_j)$.
\end{proof}
Let $g\in G\smallsetminus \left\{\id\right\}$. 
Let $(x,y)\in \left(\SpFix_{S\times S} g\right) \cap \Sigma$. According to the two previous lemmas, there exists $g_i$ and $g_j$ two automorphisms of the same order such that $(x,y)$ is of type $(g_i,\theta(g_j))$ and
$\left\langle g\right\rangle=\left\langle g_i\right\rangle\cap\left\langle g_j\right\rangle$.
There are two possibilities.
\begin{itemize}
\item[(a)]
We have $\left\langle g\right\rangle=\left\langle g_i\right\rangle=\left\langle g_j\right\rangle$ and $(x,y)$ is of type $(g,\theta(g))$. 
\item[(b)]
We have $\left\langle g_i\right\rangle\neq \left\langle g_j\right\rangle$.
In this case $g$ has order 2 or 3 and $g_i$, $g_j$ have order 4 or 6.
\end{itemize}
Our objective would be to count the number of points in $\left(\SpFix_{S\times S} g\right) \cap \Sigma$. However, as we have seen, the points in $\left(\SpFix_{S\times S} g\right) \cap \Sigma$ can be of different types, so could have different stabilizers. For this reason we need to be more subtle. 
\begin{lemme}\label{SgigjLemme}
Let $g_i\in G\smallsetminus\left\{\id\right\}$ and $g_j\in G\smallsetminus\left\{\id\right\}$. Let $x\in \SpFix _S g_i$. We set: 
$$\Sigma_{x,g_i,g_j}:=\left\{\left.(x,y)\in \Sigma\ \right|y\in\SpFix_S \theta(g_j)\right\}.$$
We set :
\begin{equation}
\mathcal{S}_{g_i,g_j}:=\left\{\left.s\in G\right|\ s^{-1}\circ \theta(g_j) \circ s=g_i\ \text{or}\ s^{-1}\circ \theta(g_j) \circ s=g_i^{-1}\right\}.
\label{Sgigj}
\end{equation}
There exists a bijection:
$$\xymatrix@R0pt{\mathcal{S}_{g_i,g_j}/\left\langle g_i\right\rangle\ar[r]&\Sigma_{x,g_i,g_j}\\
\overline{s}\ar[r]&(x,s(x)),}$$
where $\mathcal{S}_{g_i,g_j}/\left\langle g_i\right\rangle$ are the orbits of $\mathcal{S}_{g_i,g_j}$ under the right multiplication action of $\left\langle g_i\right\rangle$ and $\overline{s}$ the class of $s$ in $\mathcal{S}_{g_i,g_j}/\left\langle g_i\right\rangle$. 
\end{lemme}
\begin{proof}
Let $y\in \SpFix_S \theta(g_j)$ such that $(x,y)\in \Sigma$. Therefore by Lemma \ref{surfaces0}, there exists $s\in G$ such that $y=s(y)$. We show that $s\in\mathcal{S}_{g_i,g_j}$.
We have:
$$\theta(g_j)(s(x))=s(x).$$
That is:
$$s^{-1}\circ\theta(g_j)\circ s(x)=x.$$
Since $x\in \SpFix g_i$, this is possible if and only if $s^{-1}\circ\theta(g_j)\circ s\in \left\langle g_i\right\rangle$.
In fact, since $s^{-1}\circ\theta(g_j)\circ s$ and $g_i$ have the same order by Lemma \ref{order=}, this is possible if and only if $s^{-1}\circ\theta(g_j)\circ s= g_i$ or $=g_i^{-1}$. Hence we have a surjective map:
$$\mathcal{S}_{g_i,g_j}\rightarrow \Sigma_{x,g_i,g_j}: s\mapsto(x,s(x)).$$

If $s$ and $s'$ are in $\mathcal{S}_{g_i,g_j}$, it can be that $s(x)=s'(x)$. This means that $x$ is a fixed point of $s'^{-1}\circ s$. Again, this is possible if and only if $s'^{-1}\circ s\in \left\langle g_i\right\rangle$. 
Then, by taking the orbit of $\mathcal{S}_{g_i,g_j}$ under the right multiplication action of $\left\langle g_i\right\rangle$, we obtain the lemma.
\end{proof}

\begin{nota}\label{Sg}
When $\left\langle g_i\right\rangle=\left\langle g_j\right\rangle=\left\langle g\right\rangle$, we set $\mathcal{S}_{g}:=\mathcal{S}_{g_i,g_j}$.
\end{nota}
We summarize what we found in this section in the next remark.
\begin{rmk}\label{summarySigma}
We have:
\begin{itemize}
\item
If $g$ has order 4 or 6, then $\left(\SpFix_{S\times S} g\right) \cap \Sigma$ is isomorphic to $\left(\SpFix_S g\right)\times \mathcal{S}_{g}/\left\langle g\right\rangle$.
\item
If $g$ has order 2 or 3, then $\SpFix_{S\times S} g \cap \Sigma$ is isomorphic to 
$$\left(\SpFix_S g\right)\times \mathcal{S}_{g}/\left\langle g\right\rangle
\bigcup_{i,j,i\neq j} \left(\SpFix_S g_i\right)\times \mathcal{S}_{g_i,g_j}/\left\langle g_i\right\rangle,$$
where the $(g_i,g_j)$ are the couples of automorphisms with same order, $\left\langle g_i\right\rangle\neq \left\langle g_j\right\rangle$ and $\left\langle g_i\right\rangle\cap \left\langle g_j\right\rangle=\left\langle g\right\rangle$.

\end{itemize}
\end{rmk}
This achieves our first goal of counting the points in $\left(\SpFix_{S\times S} g\right) \cap \Sigma$ for a given $g\in G\smallsetminus \left\{\id\right\}$.

To help the understanding, we propose the following example.
\begin{ex}\label{exSgigj}
We keep the same notation as in Example \ref{Rapporteurbete2}.
\begin{itemize}
\item
We keep the choices $g_i=a$, $g_j=ab$ and $g=a^2$.
We have:
\[
\mathcal{S}_{g_i,g_j}:=\bigl\{\,s\in G\ \big|\ s^{-1}\,\theta(g_j)\,s=g_i\ \text{or}\ s^{-1}\,\theta(g_j)\,s=g_i^{-1}\,\bigr\}.
\]

Since $G$ is abelian, conjugation is trivial: $s^{-1}\theta(g_j)s=\theta(g_j)$ for all $s\in G$. 
With $\theta=\mathrm{id}$ we have $\theta(g_j)=g_j=ab$, hence the defining condition reduces to
\[
ab=a\quad\text{or}\quad ab=a^{-1},
\]
which is impossible. Therefore
\[
\mathcal{S}_{g_i,g_j}=\varnothing.
\]
It means that there is not point in $\Sigma$ of type $(g_i,\theta(g_j))$.
Moreover note that $k_6(g)=3$ because $g$ is contained in $\left\langle a\right\rangle$, $\left\langle ab\right\rangle$ and $\left\langle a^2b\right\rangle$. So according to Lemma \ref{speci46}, $\SpFix_S g=\varnothing$. Hence we obtain from Remark \ref{summarySigma} that $\SpFix_{S\times S} g \cap \Sigma=\varnothing$.

\medskip
\item
Now, we can consider $g=a$.
\[
\mathcal{S}_g=\{\,s\in G\mid s^{-1}g s=g \text{ or } g^{-1}\,\}=G,
\]
and $S_g/\langle g\rangle$ has $\frac{|G|}{|\langle g\rangle|}=\frac{12}{6}=2$ elements. 
Apply Lemma \ref{speci3}, we know that $\#\left(\SpFix_S g\right)=2$.
Therefore according to Remark \ref{summarySigma}, we obtain that: 
$$\#\left(\SpFix_{S\times S} g \cap \Sigma\right)=2\times 2=4.$$
\end{itemize}
\end{ex}
\subsubsection{Specific fixed points which are not in $\Sigma$}\label{notinsigma}
\begin{lemme}\label{fixedG}
Let $g\in G\smallsetminus\left\{\id\right\}$. Let $t(g):=\#(\mathcal{S}_g/\left\langle g\right\rangle)$.
If $g_i$ and $g_j$ are in $G\smallsetminus\left\{\id\right\}$, we also set $t(g_i,g_j):=\#(\mathcal{S}_{g_i,g_j}/\left\langle g_i\right\rangle)$.
Let $N(g):=\#\left(\SpFix_{S\times S}g\smallsetminus \Sigma\right)$ (see Notation \ref{fixednota} (iii) and (vii)).
We have:
\begin{itemize}
\item 
If $\mathcal{O}(g)=6$, $N(g)=2(2-t(g))$.
\item
If $\mathcal{O}(g)=4$, $N(g)=4(4-t(g))$.
\item
If $\mathcal{O}(g)=3$, let $(g_i)_{i\in \left\{1,...,k_6(g)\right\}}$ be the automorphisms (modulo inverse) such that $g_i^2=g$. Then:
$$N(g)=(6-2k_6(g))(6+2k_6(g)-t(g))+\sum_{i,j, i\neq j}2(2-t(g_i,g_j)).$$
\item
If $\mathcal{O}(g)=2$, let $(g_i)_{i\in \left\{1,...,k_6(g)\right\}}$ be the automorphisms (modulo inverse) such that $g_i^3=g$ and $(h_i)_{i\in \left\{1,...,k_4(g)\right\}}$ be the automorphisms (modulo inverse) such that $h_i^2=g$. Then:
\begin{align*}
N(g)&=(8-2k_6(g)-4k_4(g))(8+2k_6(g)+4k_4(g)-t(g))\\
&+\sum_{i,j, i\neq j}2(2-t(g_i,g_j))+\sum_{i,j, i\neq j}4(4-t(h_i,h_j))+16k_6(g)k_4(g).
\end{align*}
\end{itemize}
\end{lemme}
\begin{proof}
Let $(x,y)\in\SpFix_{S\times S}g\smallsetminus \Sigma$.
If $g$ is of order 4 or 6, according to Lemma \ref{differenttypeofpoint} and Hypothesis \ref{hypo}, a specific fixed point of $g$ for its action on $S\times S$ is necessarily of type 
$(g,\theta(g))$. There are $n(g)$ possible choices for $x$, since $x$ must be a specific fixed point of $g$ (see Notation \ref{fixednota} (ii)). 
For a given $x$, there are $n(g)$ possible fixed points $y$, but if $y$ coincides with $s(x)$ for some $s\in G$
we fall on the diagonal $\Sigma$ and such pairs must be excluded. By Lemma \ref{SgigjLemme} (see also Notation \ref{Sg}), there is a bijection between $\left\{\left.s(x)\right|\ s\in G\right\}\cap\SpFix_S \theta(g)$ and $\mathcal{S}_g/\left\langle g\right\rangle$.
Hence there remain $n(g)-t(g)$ admissible choices for $y$. 
This explains why the total number of such pairs is $n(g)\bigl(n(g)-t(g)\bigr)$.
This gives the result for $g$ of order 4 and 6.

Now, we assume that $\mathcal{O}(g)=3$. According to Lemma \ref{differenttypeofpoint}, there are several cases for the specific fixed points of $g$ for the action on $S\times S$:
\begin{itemize}
\item[(a)]
points of type $(g,\theta(g))$;
\item[(b)]
points of type $(g,\theta(g_i))$ or $(g_i,\theta(g))$;
\item[(c)]
points of type $(g_i,\theta(g_j))$, with $i\neq j$.
\end{itemize}
We count these different cases.
\begin{itemize}
\item[(a)]
According to Lemma \ref{speci3}, there are $6-2k_6(g)$ choices for $x$ and by Lemma \ref{SgigjLemme}, there are $6-2k_6(g)-t(g)$ choices for $y$. So the type (a) provides $(6-2k_6(g))(6-2k_6(g)-t(g))$ fixed points. 
\item[(b)]
We assume that $(x,y)$ is of type $(g,\theta(g_i))$. Since $g$ and $g_i$ do not have the same order, $(x,y)\notin\Sigma$ by Lemma \ref{order=}. 
Each $\theta(g_i)$ or $g_i$ has 2 fixed points. Hence there are $2\times (6-k_6(g))$ specific fixed points of $g$ of type $(g,\theta(g_i))$ and the same for the specific fixed points of $g$ of type $(g_i,\theta(g))$. There are $k_6(g)$ automorphisms $g_i$. So the type (b) provides 
$k_6(g)\times2\times2\times (6-k_6(g))$ fixed points.
\item[(b)]
For the type (c), we can choose $x\in \SpFix_S g_i$, then we need to choose $y\in \SpFix_S \theta(g_j)$ such that $y$ cannot be written $s(x)$ for some $s\in G$. Hence, we have 2 choices for $x$ and then $(2-t(g_i,g_j))$ choices for $y$ according to Lemma \ref{SgigjLemme}. Hence there are $\sum_{i,j, i\neq j}2(2-t(g_i,g_j))$ fixed points of type (c).
\end{itemize}
Then, the result follows for $\mathcal{O}(g)=3$.

Similarly, when $\mathcal{O}(g)=2$, there are several cases:
\begin{itemize}
\item[(1)]
points of type $(g,\theta(g))$;
\item[(2)]
points of type $(g,\theta(g_i))$ or $(g_i,\theta(g)))$;
\item[(3)]
points of type $(g,\theta(h_i))$ or $(h_i,\theta(g)))$;
\item[(4)]
points of type $(g_i,\theta(g_j))$, with $i\neq j$;
\item[(5)]
points of type $(h_i,\theta(h_j))$, with $i\neq j$;
\item[(6)]
points of type $(h_i,\theta(g_j))$ or $(g_i,\theta(h_j))$.
\end{itemize}
The number of specific fixed points for each type are given by:
\begin{itemize}
\item[(1)]
It is similar to the case (a) above for $g$ of order 3.
According to Lemma \ref{speci3}, there are $8-2k_6(g)-4k_4(g)$ choices for $x$ and again by Lemma \ref{SgigjLemme}, there are $8-2k_6(g)-4k_4(g)-t(g)$ choices for $y$. So the type (1) provides $(8-2k_6(g)-4k_4(g))(8-2k_6(g)-4k_4(g)-t(g))$ fixed points.
\item[(2)]
It is similar to the case (b) for $g$ of order 3. We obtain
$k_6(g)\times2\times2\times (8-2k_6(g)-4k_4(g))=4k_6(g)(8-2k_6(g)-4k_4(g))$  fixed points of type (2). 
\item[(3)]
This case is similar to the previous one, the only difference is that $h_i$ has order 4 so has 4 specific fixed points. 
Therefore, we have
$8k_4(g)(8-2k_6(g)-4k_4(g))$ fixed points of type (3).
\item[(4)]
This case is similar to the case (c) for $g$ of order 3. We obtain 
$\sum_{i,j, i\neq j}2(2-t(g_i,g_j))$ fixed points of type (4).
\item[(5)]
This case is similar to the previous one, the only difference is that $h_i$ has order 4 so has 4 specific fixed points.
We obtain $\sum_{i,j, i\neq j}4(4-t(h_i,h_j))$ fixed points of type (5).
\item[(6)]
We assume that $(x,y)$ is of type $(h_i,\theta(g_j))$. Since $h_i$ and $g_j$ do not have the same order, $(x,y)\notin\Sigma$ by Lemma \ref{order=}. 
Moreover $h_i$ has 4 specific fixed points and $g_j$ has 2 specific fixed points. So there are 4 choices for $x$ dans 2 choices for $y$. Moreover there are $k_4(g)$ choices for $h_i$ and $k_6(g)$ choices for $g_j$. So we obtain $4k_4(g)\times 2k_6(g)$ points. Identically, the points of type $(g_i,\theta(h_j))$ contributes for $2k_6(g)\times 4k_4(g)$ points. So in total, there are
$4k_4(g)\times 2k_6(g)+2k_6(g)\times 4k_4(g)=16k_6(g)k_4(g)$ points of type (6).
\end{itemize}
\end{proof}
\begin{ex}
We keep the same notation as in Example \ref{Rapporteurbete2} and we choose $g=a^2$.
We have seen in Example \ref{exSgigj} that $k_6(g)=3$. Moreover, we have seen in Example \ref{exSgigj} that $t(g_i,g_j)=0$.
Hence Lemma \ref{fixedG} gives:
$$N(g)=0+3\times2\times (2-0)=12.$$
\end{ex}
\subsubsection{Stabilizers of fixed points in $\Sigma$}\label{pointsinsigma}
It remains to understand the stabilizers of the fixed points in $\Sigma$. 
Let $g\in G\smallsetminus \left\{\id\right\}$. According to Lemmas \ref{differenttypeofpoint} and \ref{order=}, there are two possible types for the points in $\Sigma \cap \SpFix_{S\times S}g $:
\begin{itemize}
\item
points of type $(g,\theta(g))$, that is $x\in \SpFix g$ and $s(x)\in \SpFix \theta(g)$;
\item
points of type $(g_i,\theta(g_j))$, that is $x\in \SpFix g_i$ and $s(x)\in \SpFix \theta(g_j)$, with $g_i$ and $g_j$ two different automorphisms of the same order such that $\left\langle g\right\rangle=\left\langle g_i\right\rangle\cap\left\langle g_j\right\rangle$. In this case $g$ has order 2 or 3 and $g_i$, $g_j$ have order 4 or 6.
\end{itemize}
The purpose of the next lemmas is to determine the stabilizer of a fixed point in $\Sigma$. The formulation of the lemmas is made in order to be applied to the two previous possible types of fixed points.
\begin{lemme}\label{stabili2}
Let $g\in G\smallsetminus\left\{\id\right\}$. Let $g_i$ and $g_j$ be two automorphisms of the same order such that $\left\langle g_i\right\rangle\cap\left\langle g_j\right\rangle=\left\langle g\right\rangle$ (eventually with $\left\langle g_i\right\rangle=\left\langle g_j\right\rangle=\left\langle g\right\rangle$).
Let $s\in \mathcal{S}_{g_i,g_j}$ and $x\in\SpFix g_i$.
We set:
 $$H_{g_i,s}:=\left\{\left.s_0\circ h\right|\ s^{-1}\circ h\in\left\langle g_i\right\rangle\ \text{and}\ \theta(h)\circ h\in\left\langle g_i\right\rangle \right\}.$$	
Then the stabilizer of $(x,s(x))$ is $\left\langle H_{g_i,s},g\right\rangle$.
\end{lemme}
\begin{proof}
Let $\mathcal{G}_{(x,s(x))}$ be the stabilizer of $(x,s(x))$ for the action of $\mathcal{G}$.
Let $s_0G:=\left\{\left.s_0\circ h\right|\ h\in G\right\}$. 
Since each element of $\mathcal{G}$ can be written uniquely either as an element of $G$ or as $s_0\circ h$ with $h\in G$, we have the disjoint union
$\mathcal{G}=G\bigsqcup s_0G$.
So:
$$\mathcal{G}_{(x,s(x))}=\left(\mathcal{G}_{(x,s(x))}\cap G\right) \bigsqcup \left(\mathcal{G}_{(x,s(x))}\cap s_0G\right).$$
We can compute $\mathcal{G}_{(x,s(x))}\cap G$ and $\mathcal{G}_{(x,s(x))}\cap s_0G$ separately. 

First we have $\mathcal{G}_{(x,s(x))}\cap G=G_{(x,s(x))}$. By Lemma \ref{SgigjLemme}, $(x,s(x))$ is of type $(g_i,\theta(g_j))$; therefore by Lemma \ref{differenttypeofpoint}: 
$$G_{(x,s(x))}=\left\langle g\right\rangle.$$
Now, we compute $\mathcal{G}_{(x,s(x))}\cap s_0G$. We proceed by equivalence.
\begin{align*}
&s_0\circ h\in \mathcal{G}_{(x,s(x))}\cap s_0G\\
\Leftrightarrow& s_0\circ h\ \text{fixes}\ (x,s(x))\\
\Leftrightarrow&(\theta(h)\circ s(x),h(x))=(x,s(x))\\
\Leftrightarrow&\theta(h)\circ s(x)=x\ \text{and}\ s^{-1}\circ h(x)=x\\
\Leftrightarrow&(\theta(h)\circ s,s^{-1}\circ h)\in \left\langle g_i\right\rangle^2.
\end{align*}
Moreover, we also have:
\begin{equation}
(s_0\circ h)^2=\theta(h)\circ h\in G_{(x,s(x))}=\left\langle g\right\rangle.
\label{hthetah}
\end{equation}
Hence:
\begin{align*}
&s_0\circ h\in \mathcal{G}_{(x,s(x))}\cap s_0G\\
\Leftrightarrow&(\theta(h)\circ s,s^{-1}\circ h)\in \left\langle g_i\right\rangle\ \text{and}\ \theta(h)\circ h\in \left\langle g\right\rangle\\
\Leftrightarrow& s^{-1}\circ h\in \left\langle g_i\right\rangle\ \text{and}\ \theta(h)\circ h\in \left\langle g\right\rangle\ \text{because}\ \theta(h)\circ s=\theta(h)\circ h\circ (s^{-1}\circ h)^{-1}.
\end{align*}
It finishes the proof.
\end{proof}
\begin{rmk}
If $\left\langle g_i\right\rangle\neq \left\langle g_j\right\rangle$, we have described the stabilizer of a point of type $(g_i,\theta(g_j))$. If $\left\langle g_i\right\rangle=\left\langle g_j\right\rangle$, we have described the stabilizer of a point of type $(g,\theta(g))$.
\end{rmk}
\begin{lemme}\label{s0h}
Let $g\in G\smallsetminus\left\{\id\right\}$. Let $g_i$, $g_j$, $s$ and $H_{g_i,s}$ as in Lemma \ref{stabili2}. There are only two possibilities for the group $\left\langle H_{g_i,s},g\right\rangle$.
\begin{itemize}
\item[(i)]
If there exists $h=s\circ g_i^k$ with $k\in \N$ such that $\theta(h)\circ h\in\left\langle g\right\rangle$, then 
$\left\langle H_{g_i,s},g\right\rangle=\left\langle g, s_0\circ h\right\rangle$.
\item[(ii)]
If there is not an element $h=s\circ g_i^k$ with $k\in \N$ such that $\theta(h)\circ h\in\left\langle g\right\rangle$, then 
$\left\langle H_{g_i,s},g\right\rangle=\left\langle g\right\rangle$.
\end{itemize}
\end{lemme}
\begin{proof}
In case (ii), we have $H_{g_i,s}=\varnothing$ according to Lemma \ref{stabili2}; hence $\left\langle H_{g_i,s},g\right\rangle=\left\langle g\right\rangle$.

We assume that we are in case (i) and we consider $h\in G$ which verifies the desired condition. 
Then by definition of $H_{g_i,s}$, we have $\left\langle H_{g_i,s},g\right\rangle\supset\left\langle g, s_0\circ h\right\rangle$.
It remains to show that $\left\langle H_{g_i,s},g\right\rangle\subset\left\langle g, s_0\circ h\right\rangle$.
Let $s_0\circ h'\in H_{g_i,s}$ be another element. We have $s_0\circ h\circ s_0\circ h'=\theta(h)\circ h'$ which fixes $(x,s(x))$.
Hence by Lemma \ref{differenttypeofpoint}, $\theta(h)\circ h'\in \left\langle g\right\rangle$.
Moreover by assumption $\theta(h)\circ h\in \left\langle g\right\rangle$.
Since both $\theta(h)\circ h'$ and $\theta(h)\circ h$ belong to $\left\langle g\right\rangle$, their quotient $h^{-1}\circ h'$ lies in $\left\langle g\right\rangle$. It provides that $s_0\circ h'=\left(s_0\circ h\right)\circ\left(h^{-1}\circ h'\right)\in \left\langle g, s_0\circ h\right\rangle$.
\end{proof}
\begin{ex}
We keep the same group and notation as in Example \ref{Rapporteurbete2} and we choose $g=a$. 

Let $x\in \SpFix_S g$. According to Example \ref{exSgigj}, we have seen that $\mathcal{S}_g=G$; more precisely, we can write $\mathcal{S}_g/\left\langle g\right\rangle=\{\overline{\id},\overline{b}\}$. Therefore, we can compute the stabilizers of $(x,x)$ and $(x,b(x))$.
In this two cases, we can choose $h=s$ (for each $s\in \mathcal{S}_g/\left\langle g\right\rangle$) and $\theta(h)\circ h\in \left\langle g\right\rangle$.
So by Lemma \ref{s0h}, the stabilizer is  $\left\langle g,s_0\circ s\right\rangle\simeq C_6\times C_2$ for all $s\in \mathcal{S}_g/\left\langle g\right\rangle$.
\end{ex}
There is one kind of fixed point which have not been investigated yet; it is the object of the next lemma.
\begin{lemme}\label{complete}
Let $(x,s(x))\in\Sigma$ be a fixed point which is not fixed by an element of $G\smallsetminus \left\{\id\right\}$. Then the stabilizer of $(x,s(x))$ is 
$\left\langle s_0\circ s\right\rangle$ with $s\in F$ (see Notation \ref{fixednota} (iv)).
\end{lemme}
\begin{proof}
Let $(x,y)\in\Sigma$ which is not fixed by an element of $G\smallsetminus \left\{\id\right\}$. Then it is fixed by an element $s_0\circ s\in s_0G$ with $y=s(x)$ (see Lemma \ref{surfaces0}). We have $(s_0\circ s)^2=\theta(s)\circ s\in G$ which fixes $(x,s(x))$. However we have assumed that $(x,s(x))$ is not fixed by a non-trivial element in $G$. Hence $\theta(s)\circ s=\id$; that is $s\in F$. Moreover if $(x,s(x))$ is fixed by another element $s_0\circ s'\in s_0G$, then $(s_0\circ s)\circ(s_0\circ s')=\theta(s)\circ s'=s^{-1}\circ s'\in G$ fixes $(x,s(x))$. So by assumption $s^{-1}\circ s'=\id$. We obtain $s=s'$. It shows that the stabilizer of $(x,s(x))$ is $\left\langle s_0\circ s\right\rangle$.
\end{proof}
With the next remark, we verify that we have investigated all fixed points by an element of $\mathcal{G}\smallsetminus\left\{\id\right\}$.
\begin{rmk}\label{completermk}
We recall that $\mathcal{G}$ is the union of two disjoint sub-sets: $G$ and $s_0G:=\left\{\left.s_0\circ g\right|\ g\in G\right\}$.
Let $(x,y)\in S^2$ fixed by an element in $\mathcal{G}\smallsetminus\left\{\id\right\}$.
The point $(x,y)$  can be in $\Sigma$ or not. 
\begin{itemize}
\item
If $(x,y)$ is not in $\Sigma$ then according to Lemma \ref{surfaces0}, the point $(x,y)$ can only be fixed by an element in $G$. These points have been studied in Lemma \ref{fixedG}.
\item
If $(x,y)$ is in $\Sigma$, there are also two possibilities. The point $(x,y)$ is fixed by an element of $G\smallsetminus \left\{\id\right\}$ or not.
\begin{itemize}
\item
Let $(x,y)\in\Sigma$ fixed by an element $g\in G\smallsetminus \left\{\id\right\}$; the stabilizer of such point have been identified in Lemma \ref{s0h}.
\item
Let $(x,y)\in\Sigma$ which is not fixed by an element of $G\smallsetminus \left\{\id\right\}$. According to Lemma \ref{complete} the stabilizer of such point is $\left\langle s_0\circ s\right\rangle$ with $s\in F$. Hence $(x,s(x))$ is a generic point of a fixed surface according to Lemma \ref{surfaces}; such a point does not contribute to the singularities of $S(G)_{\theta}^{[2]}$.
\end{itemize}
\end{itemize}
\end{rmk}
\subsubsection{The possible stabilizers}
In this section we investigate which groups the stabilizers $\left\langle s_0\circ h,g\right\rangle$ of Lemma \ref{s0h} can be.
We set 
\begin{equation}
T_{m}:=\begin{pmatrix}
																													0 & 0 & 1 & 0\\
																													0 & 0 & 0 & 1\\
																													\xi_m & 0 & 0 & 0\\
																													0 & \xi_m^{-1} & 0 & 0
																													\end{pmatrix},
																													\label{Tm}
																													\end{equation}
																													where $\xi_m=e^{\frac{2i\pi}{m}}$ with $m\in \N^*$.

\begin{lemme}\label{commute}
Let $g\in G\smallsetminus\left\{\id\right\}$ and $s\in G$ such that $\mathcal{G}_{(x,s(x))}:=\left\langle s_0\circ h,g\right\rangle$ is the stabilizer of a point $(x,s(x))\in S^2$ as in Lemma \ref{s0h}. Then:
$$g\circ (s_0\circ h)=(s_0\circ h)\circ g\ \text{or}\ g\circ (s_0\circ h)=(s_0\circ h)\circ g^{-1}.$$
\end{lemme}
\begin{proof}
According to Lemma \ref{complete}, there exists $g_i$ and $g_j$ in $G$ such that $s\in \mathcal{S}_{g_i,g_j}$ and $\left\langle g_i\right\rangle\cap \left\langle g_j\right\rangle=\left\langle g\right\rangle$ (eventually with $\left\langle g_i\right\rangle=\left\langle g_j\right\rangle=\left\langle g\right\rangle$); in particular, there exists $t\in \N$ such that $g=g_i^t$. Moreover by Lemma \ref{s0h}, there exists $k\in\N$ such that $h=s\circ g_i^k$.
It follows that:
$$(s_0\circ h)\circ g=s_0\circ s\circ g_i^k\circ g_i^t=s_0\circ s\circ g_i^t \circ g_i^k.$$
By (\ref{Sgigj}):
$$s_0\circ s\circ g_i^t \circ g_i^k=s_0\circ \theta(g_j)^{\pm t}\circ s \circ g_i^k= g_j^{\pm t}\circ s_0\circ s \circ g_i^k=g_j^{\pm t}\circ (s_0\circ h).$$
However $g_i$ and $g_j$ has the same order; hence $g_j^{\pm t}=g^{\pm 1}$. This concludes the proof.
\end{proof}																													%and $d\in\left\{1,...,t\right\}$.
\begin{lemme}\label{stabilyprely}
Let $g\in G\smallsetminus\left\{\id\right\}$ and $s\in G$ such that $\mathcal{G}_{(x,s(x))}:=\left\langle s_0\circ h,g\right\rangle$ is the stabilizer of a point $(x,s(x))\in S^2$ as in Lemma \ref{s0h}.
\begin{itemize}
\item[(i)]
If $\mathcal{G}_{x,s(x)}$ is abelian, then the local action of $\mathcal{G}_{x,s(x)}$ around $(x,s(x))$ corresponds to the action of the group
$$\left\langle \diag(\xi_k,\xi_k^{-1},\xi_k,\xi_k^{-1}),\ T_1\right\rangle\ \text{or}\ \left\langle T_k\right\rangle,$$
on $\C^4$, around $0$, with $k\in \left\{2,3,4,6\right\}$.
\item[(ii)]
If $\mathcal{G}_{x,s(x)}$ is non-abelian, then the local action of $G_{x,s(x)}$ around $(x,s(x))$ corresponds to the action of the group
$$\left\langle \diag(\xi_k,\xi_k^{-1},\xi_k^{-1},\xi_k),\ T_1\right\rangle\ \text{or}\ \left\langle\diag(\xi_k,\xi_k^{-1},\xi_k^{-1},\xi_k), T_2\right\rangle,$$
on $\C^4$, around $0$, with respectively $k\in \left\{3,4,6\right\}$ or $k\in \left\{4,6\right\}$.
\end{itemize}
\end{lemme}
\begin{proof}
The automorphism $g$ only has isolated fixed points and is of order 2, 3, 4 or 6. Hence, in a good base of $\C^4$, its local action around $(x,s(x))$ can be written:
$$Q_k:=\diag(\xi_k,\xi_k^{-1},\xi_k,\xi_k^{-1}),$$
for $k\in \left\{2,3,4,6\right\}$.

The automorphism $s_0\circ h$ when it acts on $S^2$ exchanges the K3 surfaces.
Therefore, the local action of $s_0\circ h$ around $(x,s(x))$ has to be of the form:
$$H:=\begin{pmatrix}
0 & A\\
B & 0
\end{pmatrix},$$
with $A$ and $B$ in $\SU(2)$.
Moreover, by (\ref{hthetah}), we recall that: 
\begin{equation}
(s_0\circ h)^2\in \left\langle g\right\rangle. 
\label{truc}
\end{equation}
So $H^2=\diag(\xi_m,\xi_m^{-1},\xi_m,\xi_m^{-1})$ with $m\in\left\{1,2,3,4,6\right\}$ and $m$ divides $k$.
We obtain that:
\begin{equation}
AB=BA=\diag(\xi_m,\xi_m^{-1}).
\label{AB}
\end{equation}
We denote by $I_2$ the identity matrix in dimension 2. If $k=2$, $Q_k$ is identical for all basis of $\C^4$; therefore, we can take $A=\id$ in a convenient basis of $\C^4$. Hence $H=T_1$ or $H=T_2$. We have found the two cases of (i) when $k=2$.

Now, we assume that $k\neq2$ for all the sequel of the proof.
We can write $$A=\begin{pmatrix}
a_1 & a_2\\
a_3 & a_4
\end{pmatrix}.$$
So, we have:
$$A^{-1}=\begin{pmatrix}
a_4 & -a_2\\
-a_3 & a_1
\end{pmatrix}.$$
By (\ref{AB}), we have $B=\diag(\xi_m,\xi_m^{-1})A^{-1}$.
Therefore $\diag(\xi_m,\xi_m^{-1})A=A\diag(\xi_m,\xi_m^{-1})$ provides that:
\begin{equation}
a_2(\xi_m^{-1}-\xi_m)=0\ \text{and}\ a_3(\xi_m^{-1}-\xi_m)=0.
\label{matrix}
\end{equation}
These equations have two different behaviors if $m\in \left\{1,2\right\}$ or if $m\in \left\{3,4,6\right\}$.
\begin{itemize}
\item\emph{First case $m\in \left\{1,2\right\}$:}
~\\
In this case, we obtain from (\ref{AB}):
\begin{equation}
A=\pm B^{-1}.
\label{matrixx}
\end{equation}
Let $\mathfrak{G}:=\diag(\xi_k,\xi_k^{-1})$.
By Lemma \ref{commute}, we know that:
$$Q_k H=H Q_k^{\pm 1}.$$
It follows:
$$A\mathfrak{G}=\mathfrak{G}^{\pm1}A.$$
\underline{If $A\mathfrak{G}=\mathfrak{G}A$}, then $\xi_k a_2=\xi_k^{-1}a_2$ and $\xi_k a_3=\xi_k^{-1}a_3$; that is $a_2=a_3=0$. Therefore in a convenient basis of $\C^4$, we can have $A=\id$ keeping $Q_k=\diag(\xi_k,\xi_k^{-1},\xi_k,\xi_k^{-1}).$ We obtain from (\ref{matrixx}) that $H=T_1$ or $H=T_2$.
We have found the two possible groups: 
\begin{itemize}
\item[(a)]
$\left\langle T_1,\diag(\xi_k,\xi_k^{-1},\xi_k,\xi_k^{-1}) \right\rangle$ which is the left case of (i)
\item[(b)]
or $\left\langle T_2,\diag(\xi_k,\xi_k^{-1},\xi_k,\xi_k^{-1}) \right\rangle$ with $k\in \left\{4,6\right\}$.
\end{itemize}
We prove that case (b) also corresponds to one of the groups of (i).
If $k=4$, then we can replace $T_2$ by:
$$T_2\times\diag(i,-i,i,-i)=\begin{pmatrix}
																													0 & 0 & i & 0\\
																													0 & 0 & 0 & -i\\
																													-i & 0 & 0 & 0\\
																													0 & i & 0 & 0
																													\end{pmatrix},$$ which is $T_1$ after changing the base of $\C^4$ by $(i e_1, -i e_2, e_3, e_4)$.
If $k=6$, then $T_2$ can be replaced by: $$T_2\times \diag(\xi_3,\xi_3^{-1},\xi_3,\xi_3^{-1})=\begin{pmatrix}
																													0 & 0 & \xi_3 & 0\\
																													0 & 0 & 0 & \xi_3^{-1}\\
																													-\xi_3 & 0 & 0 & 0\\
																													0 & -\xi_3^{-1} & 0 & 0
																													\end{pmatrix},$$ which is $T_6$ after changing the base of $\C^4$ by $(\xi_3e_1,\xi_3^{-1}e_2,e_3,e_4)$.

\underline{If $A\mathfrak{G}=\mathfrak{G}^{-1}A$}, then $a_1=a_4=0$. Denoting $(e_1,e_2,e_3,e_4)$ the canonical basis of $\C^4$. We can exchange $e_3$ and $e_4$ and obtain a new basis $(e_1',e_2',e_3',e_4')$. In this new basis: 
$Q_k=\diag(\xi_k,\xi_k^{-1},\xi_k^{-1},\xi_k)$ and $A=\diag(a_2,a_3)$. Then, after replacing $e_3'$ by $e_3'/a_2$ and $e_4'$ by $e_4'/a_3$, we obtain $H=T_1$ or $H=T_2$. This corresponds to case (ii).
\item\emph{Second case $m\in \left\{3,4,6\right\}$:}
~\\
In this case (\ref{matrix}) provides that $a_2=a_3=0$. Then taking $(e_1,e_2,e_3/a_1,e_4/a_2)$ for the basis of $\C^4$, we can set that $A=I_2$ and we obtain by (\ref{AB}) that $H=T_m$. To fit with case (i), it remains to show that we can always take $m=1$ or $m=k$.

If $k=3$, there is nothing to prove since $m$ divides $k$. 

If $k=4$, there are 3 possibilities for $m$ which are 1, 2 or 4. If $m=1$ or $m=4$, there is nothing to prove. If $m=2$, as before $T_2$ can be replaced by $T_2\times\diag(i,-i,i,-i)$
which is $T_1$ after a change of basis which keeps $Q_k$ unchanged. 

If $k=6$, the are 4 possibilities for $m$ which are 1, 2, 3 or 6. The cases $m=1$ and $m=6$ are clear. 
As before $T_3$ can be replaced by $T_3\times\diag(\xi_6,\xi_6^{-1},\xi_6,\xi_6^{-1})$ which is $T_1$ in a good basis and $T_2$ can be replaced by $T_2\times \diag(\xi_3,\xi_3^{-1},\xi_3,\xi_3^{-1})$ which is $T_6$ in a good basis.
\end{itemize}
\end{proof}
\subsection{Singularities of $S(G)_{\theta}^{[2]}$}\label{singFinal}
\begin{nota}\label{defsing}
Let $X$ be a given orbifold. We denote as follows the singularities of $X$.
\begin{itemize}
\item
We denote by $a_k(X)$ the number of singularities of analytic type
$$\C^4/\diag(\xi_k,\xi_k^{-1},\xi_k,\xi_k^{-1}),$$
where $k\in \left\{2,3,4,6\right\}$. We also say that these singularities are of \emph{type} $a_k$.
\item
For $k\in\left\{4,6\right\}$, we denote by $a_{2k}(X)$ the number of singularities of analytic type
$$\C^4/T_k,$$
where $T_k$ is defined in (\ref{Tm}). We also say that these singularities are of \emph{type} $a_{2k}$.
\item
For $k\in\left\{4,6\right\}$, we denote by $\mathfrak{b}_{k}(X)$ the number of singularities of analytic type
$$\C^4/\left\langle\diag(\xi_k,\xi_k^{-1},\xi_k^{-1},\xi_k), T_2\right\rangle,$$
with $T_2$ which is defined in (\ref{Tm}). We also say that these singularities are of \emph{type} $\mathfrak{b}_{k}$.
\end{itemize}
When the is no ambiguity on $X$, we simply denote $a_k$ and $\mathfrak{b}_k$ instead of $a_k(X)$ and $\mathfrak{b}_k(X)$.
\end{nota}
We consider the elements of order $k$ of a group $G$.
\begin{nota}
On $G$, we consider the bijection $\inv(g)=g^{-1}$. 
The group $\left\langle \inv\right\rangle$ acts on $G$.
For $k\in \left\{2,3,4,6\right\}$,
we denote:
$$\mathcal{O}(G)_k:=\left\{\left.g\in G\right|\ \mathcal{O}(g)=k\right\}/\left\langle \inv\right\rangle,$$
where $\mathcal{O}(g)$ is the order of $g$. In other words, in $\mathcal{O}(G)_k$ we have identified $g$ and $g^{-1}$.
\end{nota}
\begin{rmk}
An element $g\in G$ and its inverse $g^{-1}$ have the same specific fixed points. So these fixed points have to be counted only once for $g$ or for $g^{-1}$. It is why the sets $\mathcal{O}(G)_k$ have been introduced.
\end{rmk}
For a given $x\in \SpFix g$, we have seen with Lemma \ref{SgigjLemme}, that $\mathcal{S}_g/\left\langle g\right\rangle$ corresponds to fixed point of $g$ of the form $(x,s(x))$. In the previous section, we have seen that the stabilizers of $(x,s(x))$ can be different groups. The aim of the next notation is to make a partition of $\mathcal{S}_g$ with each subsets corresponding to one of the possible stabilizers. All the notation are explained in Remark \ref{explainSg}.
\begin{nota}
Let $g\in G\smallsetminus\left\{\id\right\}$.
\begin{itemize}
\item
Let $\mathcal{S}_g$ be the set defined in Notation \ref{Sg}. 
According to Lemma \ref{s0h}, the set $\mathcal{S}_g$ can be divided in two subsets:
$$\mathcal{S}_g^+:=\left\{\left.s\in \mathcal{S}_g\right|\ \exists \alpha\in \left\langle g\right\rangle,\ \theta(s\circ\alpha)\circ s\circ\alpha\in\left\langle g\right\rangle\right\}\ \text{and}\ \mathcal{S}_g^-:=\mathcal{S}_g\smallsetminus \mathcal{S}_g^+.$$
The set $\mathcal{S}_g^+$ can also be divided in two subsets:
$$\mathcal{S}_g^{+,c}:=\left\{\left.s\in \mathcal{S}_g^+\right|\ g\ \text{and}\ s_0\circ s\ \text{commute}\right\}\ \text{and}\ \mathcal{S}_g^{+,nc}:=\left\{\left.s\in \mathcal{S}_g^+\right|\ g\ \text{and}\ s_0\circ s\ \text{do not commute}\right\}.$$
We set $\overline{\mathcal{S}}_g^{\bullet}:=\mathcal{S}_g^{\bullet}/\left\langle g\right\rangle$ the orbits under the right multiplication action, with 
"$\bullet$" which can be any of the previous decorations.
\item
We consider the same notation for $\mathcal{S}_{g_i,g_j}$, where $g_i$, $g_j$ are two different automorphisms in $G$ such that $\left\langle g_i \right\rangle\cap\left\langle g_j\right\rangle=\left\langle g\right\rangle$.
$$\mathcal{S}_{g_i,g_j}^+:=\left\{\left.s\in \mathcal{S}_{g_i,g_j}\right|\ \exists \alpha\in \left\langle g_i\right\rangle,\ \theta(s\circ\alpha)\circ s\circ\alpha\in\left\langle g\right\rangle\right\}\ \text{and}\ \mathcal{S}_{g_i,g_j}^-:=\mathcal{S}_{g_i,g_j}\smallsetminus \mathcal{S}_{g_i,g_j}^+.$$
Also:
$$\mathcal{S}_{g_i,g_j}^{+,c}:=\left\{\left.s\in \mathcal{S}_{g_i,g_j}^+\right|\ g\ \text{and}\ s_0\circ s\ \text{commute}\right\}\ \text{and}\ \mathcal{S}_{g_i,g_j}^{+,nc}:=\left\{\left.s\in \mathcal{S}_{g_i,g_j}^+\right|\ g\ \text{and}\ s_0\circ s\ \text{do not commute}\right\}.$$
As before, we denote 
$\overline{\mathcal{S}}_{g_i,g_j}^{\bullet}:=\mathcal{S}_{g_i,g_j}^{\bullet}/\left\langle g_i\right\rangle$ the orbits under the right multiplication action, with 
"$\bullet$" which can be any of the previous decorations.
\item
We set $$\mathcal{F}_g:=\left\{\left.s\circ \alpha\right|\ s\in F,\ \alpha\in\left\langle g\right\rangle\right\}$$
and $\overline{\mathcal{F}}_g:=\mathcal{F}_g/\left\langle g\right\rangle$ the orbits under the right multiplication action.
We recall that $F$ is defined in Notation \ref{fixednota} (iv).
\end{itemize}
\end{nota}
\begin{rmk}\label{explainSg}
Let $(x,s(x))\in\Sigma$. 
\begin{itemize}
\item
By Lemma \ref{SgigjLemme}, $(x,s(x))$ is fixed by $g\in G$ and of type $(g,\theta(g))$ (resp. $(g_i,\theta(g_j))$) if and only if $s\in \mathcal{S}_g$ (resp. $s\in \mathcal{S}_{g_i,g_j}$). 
\item
Moreover the stabilizer of $(x,s(x))$ is as in case (i) of Lemma \ref{s0h} if and only if $s\in \mathcal{S}_g^+$ (resp. $s\in\mathcal{S}_{g_i,g_j}^+$). Then, the stabilizer of $(x,s(x))$ is as in case (ii) of Lemma \ref{s0h} if and only if $s\in \mathcal{S}_g^-$ (resp. $s\in\mathcal{S}_{g_i,g_j}^-$).
\item
When $s\in \mathcal{S}_g^+$ (resp. $s\in\mathcal{S}_{g_i,g_j}^+$), according to Lemma \ref{stabilyprely}, there are four possibilities for the stabilizer of $(x,s(x))$. These four possibilities correspond to the following four subsets of $\mathcal{S}_g^+$ (resp. $s\in\mathcal{S}_{g_i,g_j}^+$):
\begin{equation}
\mathcal{S}_g^{+,c}\cap \mathcal{F}_g,\ \ \ \mathcal{S}_g^{+,nc}\cap \mathcal{F}_g,\ \ \ \mathcal{S}_g^{+,c}\smallsetminus \mathcal{F}_g,\ \ \ \text{and}\ \ \ \mathcal{S}_g^{+,nc}\smallsetminus \mathcal{F}_g
\label{allsets}
\end{equation}
$$\text{(resp.}\ \ \ \mathcal{S}_{g_i,g_j}^{+,c}\cap \mathcal{F}_{g_i},\ \ \ \mathcal{S}_{g_i,g_j}^{+,nc}\cap \mathcal{F}_{g_i},\ \ \ \mathcal{S}_{g_i,g_j}^{+,c}\smallsetminus \mathcal{F}_{g_i},\ \ \ \text{and}\ \ \ \mathcal{S}_{g_i,g_j}^{+,nc}\smallsetminus \mathcal{F}_{g_i}\text{).}$$
\end{itemize}
\end{rmk}
\begin{prop}\label{possible}
Let $S$ be a K3 surface and 
$G$ an finite admissible symplectic automorphism group of $S$.
Let $\theta$ be an involution on $G$ (not necessarily valid).
Then $S(G)^{[2]}_{\theta}$ has only singularities of type $a_k$ and $\mathfrak{b}_m$ for $k\in \left\{2,3,4,6,8,12\right\}$ and $m\in \left\{4,6\right\}$.
\end{prop}
\begin{proof}
Let $(x,y)\in S^2$. If $(x,y)\notin\Sigma$, then by Lemma \ref{surfaces0}, the stabilizer of $(x,y)$ is given by $\left\langle g\right\rangle$ with $g$ of order 2, 3, 4 or 6. If $(x,y)\in\Sigma$, by Lemmas \ref{stabili2} and \ref{s0h}, the stabilizer of $(x,y)$ can only be 
$\left\langle g\right\rangle$ or $\left\langle g,s_0\circ s\right\rangle$, with $g$ of order 2, 3, 4 or 6 and $s\in G$. 
Therefore we obtain that the singularities of $Z=S^2/\mathcal{G}$ can be of type $a_k$ for $k\in \left\{2,3,4,6\right\}$ or of type
$\C^4/\mathcal{G}_{(x,y)}$ with $\mathcal{G}_{(x,y)}$ described in Lemma \ref{stabilyprely}.

If $\mathcal{G}_{(x,y)}=\left\langle \diag(\xi_k,\xi_k^{-1},\xi_k,\xi_k^{-1}),T_1\right\rangle$ or $\mathcal{G}_{(x,y)}=\left\langle \diag(\xi_k,\xi_k^{-1},\xi_k^{-1},\xi_k),T_1\right\rangle$ with $k\in \left\{2,3,4,6\right\}$, the quotient $\C^4/\mathcal{G}_{(x,y)}$ is a singularity of type $M_k$ or $N_k$ respectively. 
According to Proposition \ref{mini}, these singularities induce on $S(G)^{[2]}_{\theta}$ only singularities of type $a_2$ or $a_3$.

The case when $\mathcal{G}_{(x,y)}=\left\langle \diag(\xi_k,\xi_k^{-1},\xi_k^{-1},\xi_k),T_2\right\rangle$ for $k\in\left\{4,6\right\}$ provides the singularities of type $\mathfrak{b}_k$.

The case  $\mathcal{G}_{(x,y)}=\left\langle T_{k}\right\rangle$ for $k\in \left\{2,3,4,6\right\}$ provides the singularities of type $a_{2k}$.
When $k\in \left\{2,3\right\}$ the singularities $\C^4/\diag(\xi_{2k},\xi_{2k}^{-1},\xi_{2k}^{-1},\xi_{2k})$ and $\C^4/T_k$ are equivalent; however this is not the case when $k\in \left\{4,6\right\}$.

\end{proof}
The singularities of types $a_8$, $a_{12}$, $\mathfrak{b}_4$ and $\mathfrak{b}_6$ are appearing only in one configuration; they are the rarest and then the simplest to determine.
\begin{prop}\label{a8}
Let $G$ be a finite admissible symplectic automorphism group action on a K3 surface $S$.
Let $a_k$ and $\mathfrak{b}_k$ be the number of singularities of $S(G)^{[2]}_{\theta}$ as defined in Notation \ref{defsing}.
We have:
\begin{equation}
a_8=\frac{16}{\left|G\right|}\sum_{g\in \mathcal{O}(G)_4}\#(\overline{\mathcal{S}}_g^{+,c}\smallsetminus \overline{\mathcal{F}}_g);
\label{a88}
\end{equation}
\begin{equation}
a_{12}=\frac{12}{\left|G\right|}\sum_{g\in \mathcal{O}(G)_6}\#(\overline{\mathcal{S}}_g^{+,c}\smallsetminus \overline{\mathcal{F}}_g);
\label{a12}
\end{equation}
\begin{equation}
\mathfrak{b}_4=\frac{16}{\left|G\right|}\sum_{g\in \mathcal{O}(G)_4}\#(\overline{\mathcal{S}}_g^{+,nc}\smallsetminus \overline{\mathcal{F}}_g);
\label{b4}
\end{equation}
\begin{equation}
\mathfrak{b}_{6}=\frac{12}{\left|G\right|}\sum_{g\in \mathcal{O}(G)_6}\#(\overline{\mathcal{S}}_g^{+,nc}\smallsetminus \overline{\mathcal{F}}_g).
\label{b6}
\end{equation}
\end{prop}
\begin{proof}
A singularity of type $a_8$ comes from a fixed point with a stabilizer which is cyclic of order 8. According to Section \ref{onSxS} and Remark \ref{completermk}, this can only occur in the right side of case (i) of Lemma \ref{stabilyprely} with $\mathcal{O}(g)=4$. 
Therefore, a singularity of type $a_8$ comes from a fixed point $(x,y)$ with a stabilizer $\left\langle g,s_0\circ s\right\rangle$ which verifies the following conditions:
\begin{itemize}
\item[(i)]
$g$ has order 4;
\item[(ii)]
the fixed point is of the form $(x,s(x))$ with $x\in \SpFix g$ and $s\in \mathcal{S}_g^+$;
\item[(iii)]
$g$ and $s_0\circ s$ commutes;
\item[(iv)]
$\overline{s}\notin \overline{\mathcal{F}}_g$.
\end{itemize}
According to Lemma \ref{s0h}, the condition (ii) is needed to have a stabilizer of the form $\left\langle g,s_0\circ s\right\rangle$ rather than $\left\langle g\right\rangle$.
The condition (iii) is needed to have an abelian stabilizer. The condition (iv) is needed to choose the right side of Lemma \ref{stabilyprely} (i) rather than the left side. 

Now, we count how many points with such a stabilizer there are in $S^2$.
These points are of the form $(x,s(x))$ with $x\in \SpFix g$ and $s\in \mathcal{S}_g^{+,c}\smallsetminus \mathcal{F}_g$.
We recall that if $s'=s\circ g^t$ then $s'(x)=s(x)$.
Therefore by Lemma \ref{speci46}, there are $4\times \#(\overline{\mathcal{S}}_g^{+,c}\smallsetminus \overline{\mathcal{F}}_g)$ such fixed points.
Let $\Fix_8$ be the set of points in $S^2$ with a stabilizer which is a cyclic group of order 8.
Then:
$$\#\Fix_8=4\sum_{g\in\mathcal{O}(G)_4}\#(\overline{\mathcal{S}}_g^{+,c}\smallsetminus \overline{\mathcal{F}}_g).$$
The group $\mathcal{G}$ acts on $\Fix_8$. By definition of $\Fix_8$, the orbits of this action have cardinality $\frac{\left|\mathcal{G}\right|}{8}=\frac{\left|G\right|}{4}$; each of these orbits provide a singularity of type $a_8$. We obtain:
$$a_8=\frac{4\times 4}{\left|G\right|}\sum_{g\in\mathcal{O}(G)_4}\#(\overline{\mathcal{S}}_g^{+,c}\smallsetminus \overline{\mathcal{F}}_g).$$
The proofs of (\ref{a12}), (\ref{b4}) and (\ref{b6}) are identical with other conditions on the stabilizer. 
\end{proof}
\begin{rmk}
Actually, considering the list of all possible $G$ and $\theta$ on a K3 surface, the singularity $a_{12}$ will never be encountered, but this could not have been predicted only from Hypothesis \ref{hypo} (see Section \ref{examplescalculs}).
\end{rmk}
The next singularities, slightly more frequent, are the ones of type $a_4$ and $a_6$. We recall that $t(g)$ is defined in Lemma \ref{fixedG}.

\begin{prop}\label{a6}
Let $G$ be a finite admissible symplectic automorphism group acting on a K3 surface $S$.
Let $a_k$ be the number of singularities of $S(G)^{[2]}_{\theta}$ as defined in Notation \ref{defsing}.
We have:
\begin{align*}
a_6&=\frac{3}{\left|G\right|}\sum_{g\in \mathcal{O}(G)_3}(6-2k_6(g))\#(\overline{\mathcal{S}}_g^+\smallsetminus \overline{\mathcal{F}}_g)+\frac{6}{\left|G\right|}\sum_{g\in \mathcal{O}(G)_3}\sum_{1\leq i\neq j \leq k_6(g)}\#(\overline{\mathcal{S}}_{g_i,g_j}^+\smallsetminus \overline{\mathcal{F}}_{g_i})\\
&+\frac{6}{\left|G\right|}\sum_{g\in \mathcal{O}(G)_6}\#(\overline{\mathcal{S}}_g^-)+\frac{6}{\left|G\right|}\sum_{g\in \mathcal{O}(G)_6}2-t(g),
\end{align*}
where for each $g\in \mathcal{O}(G)_3$, the $(g_i)_{i\in \left\{1,...,k_6(g)\right\}}$ are generators of the different cyclic groups of order 6 containing $g$. Similarly:
\begin{align*}
a_4&=\frac{2}{\left|G\right|}\sum_{g\in \mathcal{O}(G)_2}(8-2k_6(g)-4k_4(g))\#(\overline{\mathcal{S}}_g^+\smallsetminus \overline{\mathcal{F}}_g)+\frac{4}{\left|G\right|}\sum_{g\in \mathcal{O}(G)_2}\sum_{1\leq i\neq j \leq k_6(g)}\#(\overline{\mathcal{S}}_{g_i,g_j}^+\smallsetminus \overline{\mathcal{F}}_{g_i})\\
&+\frac{8}{\left|G\right|}\sum_{g\in \mathcal{O}(G)_2}\sum_{1\leq i\neq j \leq k_4(g)}\#(\overline{\mathcal{S}}_{h_i,h_j}^+\smallsetminus \overline{\mathcal{F}}_{g_i})+\frac{8}{\left|G\right|}\sum_{g\in \mathcal{O}(G)_4}\#(\overline{\mathcal{S}}_g^-)+\frac{8}{\left|G\right|}\sum_{g\in \mathcal{O}(G)_4}4-t(g),
\end{align*}
where for each $g\in \mathcal{O}(G)_2$, the $(g_i)_{i\in \left\{1,...,k_6(g)\right\}}$ are generators of the different cyclic groups of order 6 containing $g$ and the $(h_i)_{i\in \left\{1,...,k_4(g)\right\}}$ are generators of the different cyclic groups of order 4 containing $g$.
\end{prop}
\begin{proof}
We prove the proposition for $a_4$, the proof is identical for $a_6$.
The singularities of type $a_4$ can appear from three different configurations of fixed points:
\begin{itemize}
\item[(a)]
a fixed point in $\Sigma$ with a cyclic stabilizer $\left\langle g,s_0\circ h\right\rangle$ with $g\in G$ of order 2 and $(s_0\circ h)^2\in \left\langle g\right\rangle$. This is the case (i) of Lemma \ref{s0h} and the right side of Lemma \ref{stabilyprely} (i).
\item[(b)]
a fixed point in $\Sigma$ with a stabilizer $\left\langle g\right\rangle$ with $g\in G$ of order 4. This is the case (ii) of Lemma \ref{s0h}.
\item[(c)]
a fixed point which is not in $\Sigma$ fixed by an automorphism $g\in G$ of order 4. This is described in Lemma \ref{fixedG}.
\end{itemize}
We are going to count all the fixed points from these three configurations.
\begin{itemize}
\item[(c)]
The simplest configuration are the one given by Lemma \ref{fixedG}. %It corresponds to fixed points of type $(g,\theta(g))$ with $g$ of order 4.
Let $g$ be an automorphism of order 4, then $g$ has $4(4-t(g))$ specific fixed points which are not in $\Sigma$. Moreover the stabilizer of such a fixed point has order 4; so the action of $\mathcal{G}$ on this set of points has orbits of cardinality $\frac{\left|\mathcal{G}\right|}{4}=\frac{\left|G\right|}{2}$. 
Therefore, the contribution to the singularities of these kinds of fixed points is:
$$\frac{2}{\left|G\right|}\sum_{g\in \mathcal{O}(G)_4}4(4-t(g))=\frac{8}{\left|G\right|}\sum_{g\in \mathcal{O}(G)_4}(4-t(g)).$$
\item[(b)]
These fixed points are points that can be written $(x,s(x))$ with $s\in \mathcal{S}_g^-$ and $x$ fixed by an automorphism of order 4. Therefore there are 4 choices for $x$ by Lemma \ref{speci46} and $\#(\overline{\mathcal{S}}_g^-)$ choices for $s(x)$ according to Lemma \ref{s0h} (ii). As before the orbits of the action of $\mathcal{G}$ on this set of points have cardinality $\frac{\left|G\right|}{2}$.
Hence the contribution of these fixed points to the singularities is:
$$\frac{2}{\left|G\right|}\sum_{g\in \mathcal{O}(G)_4}4\times\#(\overline{\mathcal{S}}_g^-)=\frac{8}{\left|G\right|}\sum_{g\in \mathcal{O}(G)_4}\#(\overline{\mathcal{S}}_g^-).$$
\item[(a)]
These fixed points have a stabilizer $\left\langle g,s_0\circ h\right\rangle$ as described in Lemma \ref{s0h} (i) and the right side of Lemma \ref{stabilyprely} (i), with $g$ of order 2.
These fixed points can be of three different types:
\begin{itemize}
\item[(I)]
the type $(g,\theta(g))$;
\item[(II)]
the type $(g_i,\theta(g_j))$ with $g_i$ and $g_j$ of order 6 such that $\left\langle g_i\right\rangle\cap\left\langle g_j\right\rangle=\left\langle g\right\rangle$;
\item[(III)]
the type $(h_i,\theta(h_j))$ with $h_i$ and $h_j$ of order 4 such that $\left\langle h_i\right\rangle\cap\left\langle h_j\right\rangle=\left\langle g\right\rangle$.
\end{itemize}
We count the number of fixed points for each types.
\begin{itemize}
\item[(I)]
Let $(x,s(x))$ be a fixed point of type $(g,\theta(g))$ with stabilizer $\left\langle g,s_0\circ h\right\rangle$. There are $8-k_6(g)-k_4(g)$ choices for $x$ according to Lemma \ref{speci2} and there are $\#(\overline{\mathcal{S}}_g^+\smallsetminus \overline{\mathcal{F}}_g)$ choices for $s(x)$. Therefore, these fixed points contribute to $$\frac{2}{\left|G\right|}\sum_{g\in \mathcal{O}(G)_2}(8-2k_6(g)-4k_4(g))\#(\overline{\mathcal{S}}_g^+\smallsetminus \overline{\mathcal{F}}_g)$$
singularities of type $a_4$.
\item[(II)]
Let $(x,s(x))$ be a fixed point of type $(g_i,\theta(g_j))$ with stabilizer $\left\langle g,s_0\circ h\right\rangle$ and $g_i$, $g_j$ of order 6. There are 2 choices for $x$ according to Lemma \ref{speci46} and $\#(\overline{\mathcal{S}}_{g_i,g_j}^+\smallsetminus \overline{\mathcal{F}}_{g_i})$ choices for $s(x)$.
Therefore, the contribution to the singularities of these points are given by:
$$\frac{2}{\left|G\right|}\sum_{g\in \mathcal{O}(G)_2}\sum_{1\leq i\neq j \leq k_6(g)}2\times\#(\overline{\mathcal{S}}_{g_i,g_j}^+\smallsetminus \overline{\mathcal{F}}_{g_i})=\frac{4}{\left|G\right|}\sum_{g\in \mathcal{O}(G)_2}\sum_{1\leq i\neq j \leq k_6(g)}\#(\overline{\mathcal{S}}_{g_i,g_j}^+\smallsetminus \overline{\mathcal{F}}_{g_i}).$$
\item[(III)]
In this case, the computation is identical to the one in case (II).
\end{itemize}
\end{itemize}
\end{proof}
Finally, it remains to count the singularities of type $a_2$ and $a_3$ which are the most frequent.
\begin{prop}\label{a3}
Let $G$ be a finite admissible symplectic automorphism group action on a K3 surface $S$. 
Let $a_k$ be the number of singularities of $S(G)^{[2]}_{\theta}$ as defined in Notation \ref{defsing}. For each $g\in G$, let $N(g)$ be the numbers provided in Lemma \ref{fixedG}.
We have:
\begin{align*}
a_3&=\frac{3}{2\left|G\right|}\sum_{g\in \mathcal{O}(G)_3}N(g)+\frac{3}{2\left|G\right|}\sum_{g\in \mathcal{O}(G)_3}(6-2k_6(g))\#(\overline{\mathcal{S}}_g^-)+\frac{3}{\left|G\right|}\sum_{g\in \mathcal{O}(G)_3}\sum_{1\leq i\neq j \leq k_6(g)}\#(\overline{\mathcal{S}}_{g_i,g_j}^-)\\
&+\frac{6}{\left|G\right|}\sum_{g\in \mathcal{O}(G)_3}(6-2k_6(g))\#(\overline{\mathcal{S}}_g^{+,c}\cap \overline{\mathcal{F}}_g)+\frac{12}{\left|G\right|}\sum_{g\in \mathcal{O}(G)_3}\sum_{1\leq i\neq j \leq k_6(g)}\#(\overline{\mathcal{S}}_{g_i,g_j}^{+,c}\cap \overline{\mathcal{F}}_{g_i})\\
&+\frac{48}{\left|G\right|}\sum_{g\in \mathcal{O}(G)_6}\#(\overline{\mathcal{S}}_g^{+,c}\cap \overline{\mathcal{F}}_g).
\end{align*}
Moreover, we have:
\begin{align*}
a_2&=\frac{1}{\left|G\right|}\sum_{g\in \mathcal{O}(G)_2}N(g)+\frac{1}{\left|G\right|}\sum_{g\in \mathcal{O}(G)_2}(8-2k_6(g)-4k_4(g))\#(\overline{\mathcal{S}}_g^-)+\frac{2}{\left|G\right|}\sum_{g\in \mathcal{O}(G)_2}\sum_{1\leq i\neq j \leq k_6(g)}\#(\overline{\mathcal{S}}_{g_i,g_j}^-)\\
&+\frac{4}{\left|G\right|}\sum_{g\in \mathcal{O}(G)_2}\sum_{1\leq i\neq j \leq k_4(g)}\#(\overline{\mathcal{S}}_{h_i,h_j}^-)+\frac{12}{\left|G\right|}\sum_{g\in \mathcal{O}(G)_6}\#(\overline{\mathcal{S}}_g^{+,nc}\cap \overline{\mathcal{F}}_g)+\frac{64}{\left|G\right|}\sum_{g\in \mathcal{O}(G)_4}\#(\overline{\mathcal{S}}_g^{+,c}\cap \overline{\mathcal{F}}_g)
\end{align*}
\end{prop}
\begin{proof}
The computation is slightly different for $a_3$ and $a_2$, so we are going to prove the two equations.
\subsubsection*{Computation of $a_3$:}
The singularities of type $a_3$ can appear from three different configurations of fixed points.
\begin{itemize}
\item[(a)]
The fixed points which are not in $\Sigma$ and are fixed by an automorphism of order 3. There are described in Lemma \ref{fixedG}.
\item[(b)]
The fixed points which are in $\Sigma$ with a stabilizer of order 3. This is the case of Lemma \ref{s0h} (ii) when $g$ has order 3.
\item[(c)]
It can also be singularities remaining after a partial resolution as described in Proposition \ref{mini}.
\end{itemize}
We count the singularities for each of these configurations.
\begin{itemize}
\item[(a)]
According to Lemma \ref{fixedG}, the number of fixed points in this configuration is $N(g)$ for each $g\in \mathcal{O}(G)_3$. Hence, it provides the following contribution to the singularities:
$$\frac{3}{2\left|G\right|}\sum_{g\in \mathcal{O}(G)_3}N(g).$$
\item[(b)]
The fixed points of this configuration are in $\Sigma$, so can be written $(x,s(x))$. They can be of two different types:
\begin{itemize}
\item[(I)]
type $(g,\theta(g))$;
\item[(II)]
type $(g_i,\theta(g_j))$, with $g_i$ and $g_j$ of order 6. 
\end{itemize}
Since the stabilizer of these fixed points is of order 3, we are in the case (ii) of Lemma \ref{s0h}. Therefore, in case (I), there are $6-2k_6(g)$ choices for $x$ according to Lemma \ref{speci3} and $\#\left(\overline{\mathcal{S}}_g^{-}\right)$ choices for $s(x)$. We obtain the following contribution to the singularities: 
$$\frac{3}{2\left|G\right|}\sum_{g\in \mathcal{O}(G)_3}(6-2k_6(g))\#(\overline{\mathcal{S}}_g^-).$$
Similarly, in case (II), we have 2 choices for $x$ according to Lemma \ref{speci46} and $\#\left(\overline{\mathcal{S}}_{g_i,g_j}^{-}\right)$ choices for $s(x)$. We obtain the following contribution to the singularities: 
$$\frac{3}{\left|G\right|}\sum_{g\in \mathcal{O}(G)_3}\sum_{1\leq i\neq j \leq k_6(g)}\#(\overline{\mathcal{S}}_{g_i,g_j}^-).$$
\item[(c)]
The singularities in this case are obtained as remaining singularities of terminalizations (see Proposition \ref{mini}). Therefore, they come from fixed points contained in a fixed surface. These fixed points can be written $(x,s(x))$ with $s\in F$ (see Lemma \ref{surfaces}). If the stabilizer of $(x,s(x))$ is given by $\left\langle s_0\circ s\right\rangle$ then the singularities induced by $(x,s(x))$ is resolved by a terminalization. Hence by Lemma \ref{complete}, a fixed point $(x,s(x))$ with $s\in F$ can induce singularities on $S(G)_{\theta}^{[2]}$ if and only if there exists $g$, $g_i$, $g_j$ in $G\smallsetminus \left\{\id\right\}$ such that $s\in \mathcal{S}_{g_i,g_j}$ (eventually with $g=g_i=g_j$). According to Proposition \ref{mini}, there are two configurations that can lead to singularities of type $a_3$ on the terminalization:
\begin{itemize}
\item[(I)]
when $g$ has order 3 and the stabilizer of $(x,s(x))$ is abelian;
\item[(II)]
when $g$ has order 6 and the stabilizer of $(x,s(x))$ is abelian.
\end{itemize}
We count the singularities induced by these two previous configurations. We stat by considering the case (II) which is slightly simplest.
\begin{itemize}
\item[(II)]
Since $g$ has order 6, by Lemma \ref{speci46}, there are 2 choices for $x$. Since $s\in \mathcal{F}_g$ and the stabilizer of $(x,s(x))$ is abelian, by Lemma \ref{SgigjLemme}, there are 
$\#\overline{\mathcal{S}}_{g}^{+,c}\cap \overline{\mathcal{F}}_g$ choices for $s(x)$. Moreover according to Proposition \ref{mini}, the terminalization provides 4 singularities of type $a_3$. Note also that the stabilizer of $ (x,s(x))$ has order $2\times 6=12$. We obtain that the contribution to the singularities in this case is:
$$\frac{12}{2\left|G\right|}\sum_{g\in \mathcal{O}(G)_6}4\times 2\times\#(\overline{\mathcal{S}}_g^{+,c}\cap \overline{\mathcal{F}}_g)=\frac{48}{\left|G\right|}\sum_{g\in \mathcal{O}(G)_6}\#(\overline{\mathcal{S}}_g^{+,c}\cap \overline{\mathcal{F}}_g).$$
\item[(I)]
In this case $g$ has order 3. Therefore, the stabilizer of $(x,s(x))$ has order $2\times 3=6$. Moreover, $(x,s(x))$ can be of type $(g,\theta(g))$ or $(g_i,\theta(g_j))$ with $\left\langle g_i\right\rangle\cap\left\langle g_j\right\rangle=\left\langle g\right\rangle$ and $g_i$ and $g_j$ of order 6. 

When $(x,s(x))$ is of type $(g,\theta(g))$, applying the same results as before (Lemmas \ref{speci3}, \ref{SgigjLemme} and Proposition \ref{mini}), there are $6-2k_6(g)$ choices for $x$ and $\# \overline{\mathcal{S}}_g^{+,c}\cap \overline{\mathcal{F}}_g$ choices for $s(x)$. Moreover, the terminalization in this case provide 2 singularities of type $a_3$. We obtain the following contribution to the singularities:
$$\frac{6}{2\left|G\right|}\sum_{g\in \mathcal{O}(G)_3}2\times(6-2k_6(g))\#(\overline{\mathcal{S}}_g^{+,c}\cap \overline{\mathcal{F}}_g)=\frac{6}{\left|G\right|}\sum_{g\in \mathcal{O}(G)_3}(6-2k_6(g))\#(\overline{\mathcal{S}}_g^{+,c}\cap \overline{\mathcal{F}}_g).$$

When $(x,s(x))$ is of type $(g_i,\theta(g_j))$, there are 2 choices for $x$ and $\# \overline{\mathcal{S}}_{g_i,g_j}^{+,c}\cap \overline{\mathcal{F}}_{g_i}$ choices for $s(x)$. Hence, we have the following contribution to the singularities:
$$\frac{6}{2\left|G\right|}\sum_{g\in \mathcal{O}(G)_3}\sum_{1\leq i\neq j \leq k_6(g)}2\times 2\times\#(\overline{\mathcal{S}}_{g_i,g_j}^{+,c}\cap \overline{\mathcal{F}}_{g_i})=\frac{12}{\left|G\right|}\sum_{g\in \mathcal{O}(G)_3}\sum_{1\leq i\neq j \leq k_6(g)}\#(\overline{\mathcal{S}}_{g_i,g_j}^{+,c}\cap \overline{\mathcal{F}}_{g_i}).$$
\end{itemize}
 \end{itemize}
\subsubsection*{Computation of $a_2$:}
As before, the singularities of type $a_2$ can appear from three different configurations of fixed points.
\begin{itemize}
\item[(a)]
The fixed points which are not in $\Sigma$ and are fixed by an automorphism of order 2. There are described in Lemma \ref{fixedG}.
\item[(b)]
The fixed points which are in $\Sigma$ with a stabilizer of order 2. This is the case of Lemma \ref{s0h} (ii) when $g$ has order 2.
\item[(c)]
It can also be singularities remaining after a partial resolution as described in Proposition \ref{mini}.
\end{itemize}
The computation in case (a) follows as before from Lemma \ref{fixedG}. The case (b) is also similar apart that the fixed points $(x,s(x))$ can be of three different types: $(g,\theta(g))$, with $g$ of order 2, $(g_i,\theta(g_j))$ with $g_i$, $g_j$ of order 6 and $(h_i,\theta(h_j))$ with $h_i$, $h_j$ of order 4. Then the proof is identical to the proof for $a_3$ replacing Lemma \ref{speci3} by Lemma \ref{speci2}.

However, the case (c) is slightly different for $a_3$ and for $a_2$. Case (c) comes from fixed points $(x,s(x))\in\Sigma$ with a stabilizer $\left\langle g,s_0\circ s\right\rangle$ and $s\in F$ (see Lemma \ref{surfaces}). According to Proposition \ref{mini}, the terminalizations which leave singularities of type $a_2$ are the following:
\begin{itemize}
\item[(I)]
When $\left\langle g,s_0\circ s\right\rangle$ is non-abelian and $g$ has order 6.
\item[(II)]
When $\left\langle g,s_0\circ s\right\rangle$ is abelian and $g$ has order 4.
\end{itemize}
We count the singularities provided by these two configurations.
\begin{itemize}
\item[(I)]
In this case $g$ has order 6. Hence there are 2 choices for $x$ by Lemma \ref{speci46}. Since the stabilizer is non-abelian, there are 
$\#\overline{\mathcal{S}}_{g}^{+,nc}\cap \overline{\mathcal{F}}_g$ choices for $s(x)$ (see Lemma \ref{SgigjLemme}). Moreover according to Proposition \ref{mini}, the terminalization provides one singularity of type $a_2$. Note also that the stabilizer of $(x,s(x))$ has order $2\times 6=12$. We obtain that the contribution to the singularities in this case is:
$$\frac{12}{2\left|G\right|}\sum_{g\in \mathcal{O}(G)_6}2\times\#(\overline{\mathcal{S}}_g^{+,nc}\cap \overline{\mathcal{F}}_g)=\frac{12}{\left|G\right|}\sum_{g\in \mathcal{O}(G)_6}\#(\overline{\mathcal{S}}_g^{+,nc}\cap \overline{\mathcal{F}}_g).$$
\item[(II)]
In this case $g$ has order 4. Hence, there are 4 choices for $x$ according to Lemma \ref{speci46}. Since the stabilizer is abelian, there are 
$\#\overline{\mathcal{S}}_{g}^{+,c}\cap \overline{\mathcal{F}}_g$ choices for $s(x)$. By Proposition \ref{mini}, the terminalization provides 4 singularities of type $a_2$. In this case the stabilizer of $(x,s(x))$ has order $2\times 4=8$. We obtain that the contribution to the singularities in this case is:
$$\frac{8}{2\left|G\right|}\sum_{g\in \mathcal{O}(G)_4}4\times 4\times\#(\overline{\mathcal{S}}_g^{+,c}\cap \overline{\mathcal{F}}_g)=\frac{64}{\left|G\right|}\sum_{g\in \mathcal{O}(G)_4}\#(\overline{\mathcal{S}}_g^{+,c}\cap \overline{\mathcal{F}}_g).$$
\end{itemize}
\end{proof}
\subsection{A practical method to verify the computations of the singularities}\label{verif}
This section comes from a discussion with Song Jieao (see \cite{Song}). 
\begin{prop}\label{rational}
Let $X$ be a primitively symplectic orbifold. Let $C_X$ be its Fujiki constant (see Section \ref{Fujikirelationsection}).
Then, the number
$$\sqrt{\frac{(7c_2(X)^2-4c_4(X))C_X}{15}}$$
is a rational number.
\end{prop}
\begin{proof}
Let $X$ be a primitively symplectic orbifold.
By \cite[Lemma 4.6]{Lol}, there exists a rational number $C(c_2)$ such that:
$$c_2(X)\cdot \alpha^2=C(c_2)q_X(\alpha),$$
for all $\alpha\in H^2(X,\C)$.
By \cite[Corollary 2.6 and Section 3]{Song}, we have:
\begin{equation}
7c_2(X)^2-4c_4(X)=\frac{15C(c_2)^2}{C_X}.
\label{songeq}
\end{equation}
Then, we have:
\begin{equation}
C(c_2)=\sqrt{\frac{(7c_2(X)^2-4c_4(X))C_X}{15}}
\label{Cc2}
\end{equation}
which is a rational number.
\end{proof}
\begin{cor}\label{verification}
Let $S$ be a K3 surface and $G$ a symplectic automorphism group of $S$.
Let $X=S(G)_{\theta}^{[2]}$. Then:
$$\sqrt{\frac{\left|G\right|(7c_2(X)^2-4c_4(X))}{5}}$$
is a rational number.
\end{cor}
\begin{proof}
By Proposition \ref{Fujiki}, we have:
\begin{equation}
C_Xq_X(r^*\epsilon(\alpha))^2=3\left|G\right|(\left|G\right|\alpha^2)^2,
\label{n=2}
\end{equation}
where $r$ and $\epsilon$ are defined in Section \ref{Fujikirelationsection} and $\alpha\in H^2(S,\Z)$.
Combining (\ref{n=2}) with Proposition \ref{rational}, we obtain our result.
\end{proof}
\begin{rmk}
In practice, we use the results of Section \ref{topo} to compute $\sqrt{\frac{\left|G\right|(7c_2(X)^2-4c_4(X))}{5}}$ (see Section \ref{examplescalculs}); then it can be expressed in term of the singularities and the Betti numbers $b_2$, $b_3$. Therefore Proposition \ref{rational} can be used to verify if the computation of the singularities is correct.
\end{rmk}
For this purpose, we set:
\begin{nota}\label{verifnota}
$\overline{C(c_2)}:=\sqrt{\frac{\left|G\right|(7c_2(X)^2-4c_4(X))}{5}}.$
\end{nota}
We provide an example of use of Proposition \ref{rational}. The following error have been noticed and corrected in \cite{Song}.
\begin{erra}
Let $K_2(T)$ be a generalized Kummer fourfold endowed with $g_4$ a natural symplectic automorphism of order 4 (coming from an automorphism of the torus). 
Let $K_4'$ be a terminalization of the quotient $K_2(T)/g_4$. 

The computation of the singularities of $K_4'$ in \cite[Sections 5.4]{Fu} is incorrect.
The correct singularities are 8 singularities of analytic type $\C^4/h_4$ and 30 singularities of analytic type $\C^4/-\id$,
with $h_4=\diag(i,-i,-i,i)$. 
\end{erra} 
\begin{proof}
Let $a_2$ (resp. $a_4$) be the number of singularities of analytic type $\C^4/-\id$ (resp. $\C^4/h_4$). 
In \cite[Sections 5.4]{Fu}, it is given:
$$b_2(K_4')=6,\ b_3(K_4'),\ a_2=45,\ \text{and}\ a_4=2.$$
We show that this result leads to $\sqrt{\frac{(7c_2(X)^2-4c_4(X))C_X}{15}}$ irrational.
By Lemma \ref{s}, we have:
$$s(K_4')=-45-3\times 2=-51.$$
By Proposition \ref{RR}, we obtain:
$$\chi(K_4')=48+12\times 6-51=69.$$
By Proposition \ref{c4}:
$$c_4(K_4')=69-\frac{45}{2}-\frac{2\times 3}{4}=45.$$
By Lemma \ref{S0}:
$$S_0(K_4')=\frac{45}{2^5}+\frac{18}{2^6}=\frac{27}{16}.$$
Therefore by Proposition \ref{c2}:
$$c_2(K_4')^2=720-\frac{240\times 27}{16}+15=330.$$
Since $C_{K_2(T)}=3^2$, by Proposition \ref{Fujiki2}:
$$C_{K_4'}=y^2,$$
with $y$ a rational number.
Therefore:
$$\frac{(7c_2(K_4')^2-4c_4(K_4'))C_{K_4'}}{15}=\frac{(7\times 330-4\times 45)y^2}{15}=142y^2.$$ 
By Proposition \ref{rational}, $\sqrt{142}$ is a rational number; this leads to a contradiction. It proves that the computation of the singularities of $K_4'$ in \cite[Sections 5.4]{Fu} is incorrect. The correct computation is provided in \cite[Appendix A]{Song}.
\end{proof}
\begin{rmk}\label{mistakes}
Let $\mathcal{D}_6$ be the dihedral group of order 6.
With the same method, we find that the computations of the singularities of $S(\mathcal{D}_6)^{[2]}_{\id}$ in \cite[Section 5.7]{Fu} and of $S(C_4^2)^{[2]}$, $S(C_6)^{[2]}$, $S(C_2\times C_4)^{[2]}$, $S(C_2\times C_6)^{[2]}$ in \cite[Section 13]{Fujiki} are incorrect.
The correct computations are provided in Section \ref{examplescalculs}.
\end{rmk}
\section{Examples}\label{examplescalculs}
\subsection{The valid involutions}
We recall that the equivalent relation between involutions on a finite symplectic group $G$ is stated in Definition \ref{equiinv}.
In order to prove Theorem \ref{main4}, we need a description of all the possible valid involution modulo equivalence.  
For this purpose, every groups $G$ are embedded in a permutation group $\mathfrak{S}_n$; this embedding is described by giving a family of generators of $G$ in $\mathfrak{S}_n$. Then the involutions on $G$ are described via an element of order 2 in $\mathfrak{S}_n$; the involution is obtained by the action by conjugation of this element on $G$. For instance $\mathfrak{A}_4$ is embedded in $\mathfrak{S}_4$ via the family of generators $\left\{(0,1,2),(1,2,3)\right\}$. Then the transposition $(1,2)$ induces an involution on $\mathfrak{A}_4$: $\theta(g)=(1,2)\circ g \circ (1,2)$.

The next lemma has been obtained with a computer, its proof is given in the annexe sections \ref{firstmethod} and \ref{secondmethod} and it is based on Proposition \ref{invodeform}.
In the annexe we explain the computer programs but do not provide them entirely; they can be found in \cite{github}.
We follow the classification of symplectic automorphism groups of a K3 surface given in \cite{xiao}. The group's names are specified in Section \ref{notanota}.
\begin{lemme}\label{mainlemma2}
Let $G$ be a finite admissible non-abelian symplectic automorphism group of a K3 surface.
The list of all equivalent classes of valid involutions on $G$ is given by the following representatives:
\[
\begin{tabular}{|c|c|c|c|}
\hline
$G$ & n & Embedding in $\mathfrak{S}_n$ & classes of valid involutions\\
\hline
$\mathfrak{S}_3$ & 3 & $(0,1)$; $(0,1,2)$ &$\id$\\
\hline
$\mathcal{D}_4$ & 4 &$(0,1,2,3)$; $(0,3)(1,2)$ & $\id$\\
\hline
$\mathfrak{A}_4$ & 4 &$(0,1,2)$; $(1,2,3)$ & $(1,2)$\\
\hline
$\mathcal{D}_6$ & 6 & $(0,1,2,3,4,5)$; $(0,5)(1,4)(2,3)$ & $\id$\\
\hline
$C_2\times\mathcal{D}_4$ & 6 &$(0,1,2,3)$; $(0,3)(1,2)$; $(4,5)$ & $\id$\\
\hline
$C_2^2\rtimes C_4$ & 7 &$(0,1,2,7)(3,4,5,6)$; $(0,4)(2,6)$& $(0, 6)(2, 4)$\\
\hline
$\mathfrak{A}_{3,3}$ & 9& $(0,1,8)(2,3,4)(5,6,7)$; $(0,1)(2,5)(3,7)(4,6)$; & $\id$ and $(0, 5)(1, 2)(3, 7)$\\
& & $(0,3,6)(1,4,7)(2,5,8)$& \\
\hline
$C_3\times \mathfrak{S}_3$ & 6 & $(0,1,2)$; $(0,1)$; $(3,4,5)$ & $(3,4)$\\
\hline
$\mathfrak{S}_4$ & 4 & $(0,1)$; $(0,1,2,3)$ & $\id$\\
\hline
$C_2\times\mathfrak{A}_4$ & 6 & $(0,1,2)$; $(1,2,3)$; $(4,5)$ & $(1,2)$\\
\hline
$C_2^2\wr C_2$ & 8&$(0,1)$; $(2,3)$; $(4,5)$; $(6,7)$; $(0,4)(1,5)(2,6)(3,7)$ & $\id$ and $(2, 6)(3, 7)$\\
\hline
$C_2^3\rtimes C_4$ &8 &$(0,1,2,7)(3,4,5,6)$; $(1,5)(2,6)$ & $(1, 3)(5, 7)$\\
\hline
$C_3^2\rtimes C_4$ & 6 & $(0,3,4,1)(2,5)$; $(1,3,5)$; $(0,4)(1,3)$ & $(0, 2)$\\
\hline
$C_3\times \mathfrak{A}_4$ & 7 & $(0,1,2)$; $(1,2,3)$; $(4,5,6)$ & $(1,2)(4,6)$\\
\hline
$\mathfrak{S}_3\times \mathfrak{S}_3$ & 6 & $(0,1)$; $(0,1,2)$; $(3,4)$; $(3,4,5)$ & $\id$\\
\hline
$C_2^2\rtimes \mathfrak{A}_4$ & 12 & $(0,9)(2,11)(3,6)(5,8)$; $(0,6)(2,8)(3,9)(5,11)$ & $(0, 4)(1, 9)(3, 7)(6, 10)$\\
 & & $(0,4,8)(1,5,9)(2,6,10)(3,7,11)$ & \\
\hline
$C_4^2\rtimes C_3$ & 12 & $(0,9,6,3)(2,5,8,11)$; $(0,4,8)(1,5,9)(2,6,10)(3,7,11)$ & $(1, 8)(2, 7)(4, 11)(5, 10)$\\
\hline
$C_2\times \mathfrak{S}_4$ & 6 & $(0,1)$; $(2,3)$; $(2,3,4,5)$ & $\id$\\
\hline
$C_2^2\times\mathfrak{A}_4$ & 8 & $(0,1,2)$; $(1,2,3)$; $(4,5)$; $(6,7)$ & $(1,2)$\\
\hline
$\mathfrak{S}_3\wr C_2$ & 9 & $(0,1)(2,4)(5,6)$; $(0,1,8)(2,3,4)(5,6,7)$;& $\id$\\ 
 & &$(0,7)(1,3)(4,6)$; $(0,3,6)(1,4,7)(2,5,8)$ &\\
\hline
$C_2^4\rtimes \mathfrak{S}_3$ & 8&$(0,7)(1,2)$; $(0,1,2)(4,5,6)$; $(0,4)(1,6)(2,5)(3,7)$ & $\id$ and $(0, 1)(3, 5)$\\
\hline
$C_2^4\rtimes C_6$ & 8 &  $(0,2)(1,7)(3,5)(4,6)$; $(3,5)(4,6)$; $(0,7)(1,2)(3,4)(5,6)$; & $(0, 1)(4, 5)$\\
 & &$(0,1,2)(3,5,4)$; $(0,4)(1,5)(2,6)(3,7)$ & \\
\hline
$\mathfrak{A}_4^2$ & 8 & $(0,1,2)$; $(1,2,3)$; $(4,5,6)$; $(5,6,7)$ & $(1,2)(5,6)$\\
\hline
\end{tabular}
\]
\end{lemme}

\begin{rmk}
Note that when $G$ is abelian, there is only one valid involution which is given by the inverse: $g\mapsto g^{-1}$.
\end{rmk}
\begin{rmk}\label{invremark}
We remark in Lemma \ref{mainlemma2} that for most of the groups, there is only one class of valid involutions. There are two classes only for the groups $\mathfrak{A}_{3,3}$, $C_2^2\wr C_2$ and $C_2^4\rtimes \mathfrak{S}_3$. Moreover, when there are two classes of valid involutions, one is always the class of $\id$.
\end{rmk}
\begin{nota}
As a consequence of Remark \ref{invremark} and to simplify the notation, we adopt the following conversion. When there is only one class $\theta$ of valid involutions on the group $G$, the orbifold $S(G)^{[2]}_{\theta}$ is denoted by $S(G)^{[2]}$. When there are two classes of valid involutions on $G$ and one of them is the class of $\id$, the two associated orbifolds are denoted by $S(G)^{[2]}_{\id}$ and $S(G)^{[2]}_{\not\sim\id}$.
\end{nota}
\subsection{Examples in dimension 4}
We deduce from Lemma \ref{mainlemma2} all the possible deformation classes of Fujiki orbifolds obtained from an admissible group $G$.

For each orbifold, we provide the second Betti number $b_2$, the singularities, the fourth Betti number $b_4$, the Euler characteristic $\chi$, the Chern numbers $c_4$, $c_2^2$ and $\overline{C(c_2)}$ defined in Notation \ref{verifnota}. We recall that the third Betti number is always trivial according to Proposition \ref{b3}. The different type of singularities are defined in Notation \ref{defsing}.
\begin{thm}\label{main3}
All the 4-dimensional Fujiki orbifolds $S(G)^{[2]}_{\theta}$ (modulo deformation), obtained from a finite admissible symplectic group $G$ on a K3 surface $S$,
are listed below.
\[
\begin{tabular}{|c|c|c|c|c|c|c|c|}
\hline
$G$, $\theta$ & $b_2$ & singularities&$b_4$ &$\chi$&$c_4$&$c_2^2$&$\overline{C(c_2)}$\\
\hline
$C_2$ & 16 & $a_2=28$ &178 & 212 & 198 & 576 &36\\
\hline
$C_3$ & 11 & $a_3=15$ & 126 &150&140&500&42\\
\hline
$C_2^2$ & 14 & $a_2=36$&150&180&162&504&48\\
\hline
$C_4$ & 10 & $a_2=10$, $a_4=6$&118&140&130.5&486&48\\
\hline
$\mathfrak{S}_3$ & 10 & $a_2=28$, $a_3=12$&94&116&94&328&48\\
\hline
$C_6$ & 8 & $a_2=9$, $a_3=6$, $a_6=1$&100&118&$\frac{326}{3}$&$\frac{1472}{3}$&60\\
\hline
$C_2^3$ & 16 & $a_2=28$ & 178 & 212 & 198 & 576 & 72\\
\hline
$\mathcal{D}_4$ & 11 & $a_2=36$, $a_4=3$&111&135&114.75&387&60\\
\hline
$C_2\times C_4$ &10 & $a_2=12$, $a_4=4$ &122 & 144 &135&540&72\\
\hline
$C_3^2$ & 7 &  $a_3=12$&92&108&100&540&78\\
\hline
$\mathfrak{A}_4$ & 7 & $a_2=12$, $a_3=15$, $a_4=4$&62&78&59&248&60\\
\hline
$\mathcal{D}_6$ & 10 & $a_2=28$, $a_3=10$&98&120&$\frac{298}{3}$&$\frac{1096}{3}$&72\\
\hline
$C_2\times C_6$ & 8 & $a_2=12$, $a_3=3$&108&126&118&616&96\\
\hline
$C_2\times \mathcal{D}_4$ & 14 & $a_2=36$&150&180&162&504&96\\
\hline
$C_2^2\rtimes C_4$ & 10 & $a_2=10$, $a_4=6$&118&140&130.5&486&96\\
\hline
$C_4^2$ & 8 & $a_2=6$&120&138&135&720&120\\
\hline
$\mathfrak{A}_{3,3}$,\ $\id$ & 8 & $a_2=28$, $a_3=12$&74&92&70&320&84\\
\hline
$\mathfrak{A}_{3,3}$,\ $\not\sim\id$& 8 & $a_2=28$, $a_3=20$& 58&76&$\frac{146}{3}$&$\frac{512}{3}$&60\\
\hline
$C_3\times \mathcal{D}_3$ & 6 & $a_2=9$, $a_3=10$, $a_6=1$&72&86&74&408&96\\
\hline
$\mathfrak{S}_4$ & 8 & $a_2=24$, $a_3=12$, $a_4=3$&69&87&64.75&247&84\\
\hline
$C_2\times \mathfrak{A}_4$ & 6 & $a_2=13$, $a_3=6$, $a_4=4$, $a_6=1$&64&78&$\frac{191}{3}$&$\frac{932}{3}$&96\\
\hline
$C_2^2\wr C_2$,\ $\id$ & 16 & $a_2=28$&178&212&198&576&144\\ 
\hline
$C_2^2\wr C_2$,\ $\not\sim\id$ & 11 & $a_2=36$, $a_4=3$&111&135&114.75&387&120\\ 
\hline
$C_2^3\rtimes C_4$ & 8& $a_2=17$, $a_4=6$, $b_4=1$&87&105&$\frac{729}{8}$&373.5&120\\
\hline
$C_3^2\rtimes C_4$ & 6& $a_2=10$, $a_3=14$, $a_4=6$&50&64&$\frac{271}{6}$&$\frac{626}{3}$&96\\
\hline
$C_3\times \mathfrak{A}_4$ &5 & $a_2=3$, $a_3=9$, $a_4=3$, $b_6=1$&61&73&$\frac{187}{3}$&$\frac{1144}{3}$&132\\
\hline
$\mathfrak{S}_3^2$ &8 & $a_2=20$, $a_3=15$&76&94&74&328&120\\
\hline
$C_2^2\rtimes \mathfrak{A}_4$ &6 & $a_2=14$, $a_3=15$, $a_4=6$&44&58&36.5&158&96\\
\hline
$C_4^2\rtimes C_3$ & 5& $a_2=10$, $a_3=15$, $a_4=1$, $a_8=2$, $b_4=1$&35&47&$\frac{229}{8}$&153.5&96\\
\hline
$C_2\times \mathfrak{S}_4$ & 10& $a_2=28$, $a_3=10$&98&120&$\frac{298}{3}$&$\frac{1096}{3}$&144\\
\hline
$C_2^2\times \mathfrak{A}_4$ &7 & $a_2=12$, $a_3=3$, $a_4=4$&86&102&91&472&168\\
\hline
$\mathfrak{S}_3\wr C_2$ & 7 & $a_2=20$, $a_3=12$, $a_4=3$&63&79&$\frac{235}{4}$&275&156\\
\hline
$C_2^4\rtimes \mathfrak{S}_3$,\ $\id$&10 & $a_2=28$, $a_3=12$&94&116&94&328&192\\
\hline
$C_2^4\rtimes \mathfrak{S}_3$,\ $\not\sim\id$&6 & $a_2=19$, $a_3=12$, $a_4=6$, $b_4=1$&41&55&$\frac{257}{8}$&125.5&120\\
\hline
$C_2^4\rtimes C_6$ & 5& $a_2=16$, $a_3=6$, $a_4=4$, $a_6=1$, $b_4=1$&47&59&$\frac{1015}{24}$&$\frac{1405}{6}$&168\\
\hline
$\mathfrak{A}_4^2$ & 4 & $a_2=4$, $a_3=6$, $a_4=4$, $b_6=2$&48&58&$\frac{283}{6}$&$\frac{938}{3}$&240\\
\hline
\end{tabular} 
 \]
\end{thm}
\begin{proof}
Applying the results of the previous sections, the data of the theorem are obtained with the help of a computer (see \cite{github} for our computer programs and the annexe section for explanations) .
\subsubsection*{Second Betti number}
The second Betti number is computed directly from Proposition \ref{b2} using the program \cite[Second Betti number]{github}. Program \cite[Second Betti number]{github} provides the factor $\#(F/G)$ of Proposition \ref{b2}. The factors $\rk H^2(S,\Z)^G$ are found in \cite{xiao} for each $G$.
\subsubsection*{The singularities}
The singularities are obtained from Propositions \ref{a8}, \ref{a6}, \ref{a3} using the program \cite[Singularities]{github}. According to Proposition \ref{possible}, the singularities of $S(G)_{\theta}^{[2]}$ are expressed in term of $a_2$, $a_3$, $a_4$, $a_6$, $a_8$, $\mathfrak{b}_4$ and $\mathfrak{b}_6$ (in practice, we realize that the singularities of type $a_{12}$ never appear; so for simplicity in the computation, we omit it).
\subsubsection*{The fourth Betti number and the Euler characteristic}
We set $X:=S(G)_{\theta}^{[2]}$. By Proposition \ref{RR} and Proposition \ref{b3}, we have:
$$b_4(X)=46+10b_2(X)+s(X)\ \text{and}\ \chi(X)=48+12b_2(X)+s(X).$$
Then by Definition of $s(X)$ (see Notation \ref{sS0}) and Lemmas \ref{s}, \ref{s2}:
\begin{equation}
b_4(X)=46+10b_2(X)-a_2-2a_3-3a_4-5a_6-7a_8-4\mathfrak{b}_4-5\mathfrak{b}_6
\label{bb4}
\end{equation}
and:
\begin{equation}
\chi(X)=48+12b_2(X)-a_2-2a_3-3a_4-5a_6-7a_8-4\mathfrak{b}_4-5\mathfrak{b}_6.
\label{chi}
\end{equation}
So $b_4(X)$ and $\chi(X)$ are easily computed from $b_2(X)$ and the singularities.
\subsubsection*{The Chern numbers $c_4$ and $c_2^2$}
By Proposition \ref{c4}:
\begin{equation}
c_4(X)=\chi(X)-\frac{a_2}{2}-\frac{2a_3}{3}-\frac{3a_4}{4}-\frac{5a_6}{6}-\frac{7a_8}{8}-\frac{7\mathfrak{b}_4}{8}-\frac{11\mathfrak{b}_6}{12}.
\label{cc4}
\end{equation}
Moreover by Proposition \ref{c2}: 
\begin{equation}
c_2(X)^2=720+\frac{c_4(X)}{3}-240\left(\frac{a_2}{2^5}+\frac{2a_3}{27}+\frac{9a_4}{2^6}+\frac{329a_6}{864}+\frac{41a_8}{2^7}+\frac{25\mathfrak{b}_4}{2^7}+\frac{545\mathfrak{b}_6}{1728}\right).
\label{cc2}
\end{equation}
\end{proof}
\begin{rmk}
The orbifold $S(C_2^4)^{[2]}$ has been omitted since it is deformation equivalent to $S^{[2]}$ as explained in \cite[Proposition 14.5]{Fujiki} (see Proposition \ref{defequiprop} below).
\end{rmk}
\begin{rmk}\label{mistakeFujiki}
Note that Theorem \ref{main3} corrects some mistakes related to the singularities in \cite[Theorem 13.1]{Fujiki}. 
Indeed, the singularities of $S(C_6)^{[2]}$, $S(C_2\times C_4)^{[2]}$, $S(C_2\times C_6)^{[2]}$ and $S(C_4^2)^{[2]}$ provided in \cite[Theorem 13.1]{Fujiki} are incorrect. If we compute $\overline{C(c_2)}$ with the data of \cite[Theorem 13.1]{Fujiki}, we find respectively:
$$\sqrt{3720},\ 16\sqrt{19},\ 36\sqrt{7},\ \text{and}\ 8\sqrt{210},$$
which are irrational numbers (we recall from Corollary \ref{verification} that these numbers have to be rational).
\end{rmk}
\begin{rmk}
As a consequence of Proposition \ref{locallytriv}, we know that two orbifolds $S(G_1)^{[2]}_{\theta_1}$ and $S(G_2)^{[2]}_{\theta_2}$ are not deformation equivalent when their second Betti number or their singularities are different. It follows that Theorem \ref{main3} provides at least 29 deformation classes of irreducible symplectic orbifolds. 
\end{rmk}
We can actually prove that some of the orbifolds given in Theorem \ref{main3} are deformation equivalent.
The next proposition is a generalization of an idea of Fujiki \cite[Proposition 14.5]{Fujiki}.
\begin{prop}\label{defequiprop}
The following couples of orbifolds are deformation equivalent:
\begin{itemize}
\item[(i)]
$S(C_2^4)^{[2]}\sim S^{[2]}$;
\item[(ii)]
$S(C_2^2\wr C_2)_{\id}^{[2]}\sim S(C_2)^{[2]}$;
\item[(iii)]
$S(C_2^4\rtimes \mathfrak{S}_3)^{[2]}\sim S(\mathfrak{S}_3)^{[2]}$;
\item[(iv)]
$S(C_2^3)^{[2]}\sim S(C_2)^{[2]}$;
\item[(v)]
$S(C_2\times \mathcal{D}_4)^{[2]}\sim S(C_2^2)^{[2]}$.
\end{itemize}
\end{prop}
\begin{proof}
The general idea of the proof is to find some bimeromorphisms between the orbifolds listed in the proposition and then to apply Proposition \ref{bimero}.
The bimeromorphisms is obtained as follows.
Let $S\rightarrow T/-\id$ be a Kummer surface with $T$ a torus. 
Assume that we want to show that two orbifolds $S(G_1)_{\theta_1}^{[2]}$ and $S(G_2)_{\theta_2}^{[2]}$ are deformation equivalent. 
Using Proposition \ref{bimero}, it is enough to find a bimeromorphism between $S^2/\left\langle j_{\theta_1}(G_1),\mathfrak{S}_2\right\rangle $ and $S^2/\left\langle j_{\theta_2}(G_2),\mathfrak{S}_2\right\rangle$ (see Notation \ref{jtheta} for the definition of $j_{\theta}$). Such a bimeromorphism is obtained via an isogeny of $T^2$.
We denote $\mathcal{G}_i:=\left\langle j_{\theta_i}(G_i),\mathfrak{S}_2\right\rangle$ for all $i\in\left\{1,2\right\}$.
The goal of the proof is to find two groups $\widetilde{\mathcal{G}_1}$ and $\widetilde{\mathcal{G}_2}$ acting on $T^2$ such that there exists a bimeromorphism:
$$T^2/\widetilde{\mathcal{G}_1}\dashrightarrow T^2/\widetilde{\mathcal{G}_2},$$
with the groups $\widetilde{\mathcal{G}_i}$ which are given by extensions of $\left\langle (-\id,\id),(\id,-\id)\right\rangle$ by $\mathcal{G}_i$;
then the groups $\mathcal{G}_i$ induce an automorphism group on $S^2\rightarrow T^2/\left\langle (-\id,\id),(\id,-\id)\right\rangle$ and we obtain the desired bimeromorphism between $S^2/\mathcal{G}_1$ and $S^2/\mathcal{G}_2$.

Let $g$ and $h$ be two automorphisms of $T$, we denote by $(g,h)$ the induced diagonal action on $T^2$. To simplify the expression, we set $I:=\left\langle (-\id,\id),(\id,-\id)\right\rangle$.

The isogeny used to obtain the different bimeromorphisms is always the same:
$$\xymatrix@R0pt{f:\ \ T^2\ar[r]&T^2\\
(x,y)\ar[r]&(x+y,x-y).}$$
Note that: 
\begin{equation}
\Ker f=\left\{\left.(a,a)\right|\ a\in T[2]\right\}\simeq C_2^4.
\label{kerf}
\end{equation}
To finish the proof, we only have to provide the groups $\widetilde{\mathcal{G}_1}$ and $\widetilde{\mathcal{G}_2}$ in the different cases.
We denote by $s_0$ the automorphism of $T^2$ which exchanges the two factors ($s_0(x,y)=(y,x)$).
\begin{itemize}
\item[(i)]
In this case, $\mathcal{G}_2=\left\langle s_0\right\rangle$ and $\mathcal{G}_1=\left\langle s_0 \right\rangle\times \left\{\left.t_{(a,a)}\right|\ a\in T[2]\right\}$, with $t_{(a,a)}$ the translation by $(a,a)$. 
Let $H$ be a sub-group of $\Aut(T^2)$, we denote $f_*(H):=\left\{\left.f\circ h\right|\ h\in H\right\}$.
Since (\ref{kerf}), note that: $$f_*(I\rtimes\mathcal{G}_1)=I\rtimes\mathcal{G}_2.$$
Therefore, $f$ induces an isomorphism:
$$\frac{T^2}{I\rtimes \mathcal{G}_1}\rightarrow\frac{T^2}{I\rtimes \mathcal{G}_2}.$$
Moreover, the group $\widetilde{\mathcal{G}_2}=I\rtimes \mathcal{G}_2$ induces the group $\mathfrak{S}_2$ on $S^2$ and the group $\widetilde{\mathcal{G}_1}=I\rtimes \mathcal{G}_1$ induces the group $\left\langle j_{\id}(C_2^4),\mathfrak{S}_2\right\rangle$.
\item[(ii)]
Since the method is identical as (i), we only provide the groups $\widetilde{\mathcal{G}_1}$ and $\widetilde{\mathcal{G}_2}$.
In this case, we assume that $T=E\times E$, with $E$ an elliptic curve. We consider $s_2\in \Aut(T)$ given by the matrix:
$$\begin{pmatrix}0&1\\
-1&0
\end{pmatrix}.$$ 
Note that $s_2^2=-\id$.
Then, we take $\widetilde{\mathcal{G}_2}=\left\langle I, (s_2,s_2), s_0\right\rangle$ and 
$$\widetilde{\mathcal{G}_1}= \left(\left\{\left.t_{(a,a)}\right|\ a\in T[2]\right\}\rtimes\left\langle I, (s_2,s_2)\right\rangle \right)\times\left\langle s_0 \right\rangle.$$
The group $\widetilde{\mathcal{G}_2}$ is an extension of $I$ and $\left\langle (s_2,s_2), s_0\right\rangle/(-\id,-\id)$ and the group $\widetilde{\mathcal{G}_2}$ is an extension of $I$ and $\left(\left\{\left.t_{(a,a)}\right|\ a\in T[2]\right\}\rtimes\left\langle (s_2,s_2)\right\rangle /(-\id,-\id)\right)\times\left\langle s_0 \right\rangle$.
Therefore, we have $\mathcal{G}_2\simeq C_2\times \left\langle s_0\right\rangle$ and $\mathcal{G}_1\simeq (C_2^2\wr C_2)\times \left\langle s_0\right\rangle$.
\item[(iii)]
In this case, we assume that $T=E_{\xi_3}\times E_{\xi_3}$ with $E_{\xi_3}:=\C/\left\langle 1,\xi_3\right\rangle$, where $\xi_3$ is a third root of the unity. We denote by $\rho$ the automorphism of $T$ with the diagonal matrix $\diag(\xi_3,\xi_3^{-1})$.
Then, we take 
$\widetilde{\mathcal{G}_2}=\left\langle (s_2,s_2),(\rho,\rho), I,s_0\right\rangle$ 
and 
$\widetilde{\mathcal{G}_1}= \left(\left\{\left.t_{(a,a)}\right|\ a\in T[2]\right\}\rtimes\left\langle I,(s_2,s_2),(\rho,\rho)\right\rangle\right)\times\left\langle s_0 \right\rangle.$
\item[(iv)]
Let $a\in T[2]$. We take $\widetilde{\mathcal{G}_2}=\left\langle I, t_{(a,a)}, s_0\right\rangle$ and $\widetilde{\mathcal{G}_1}=\left(\left\{\left.t_{(b,b)}\right|\ b\in T[2]\right\}\right)\times\left\langle I,s_0,t_{(a,0)} \right\rangle$.
Note that $f\circ t_{(a,0)}=t_{(a,a)}$.
Moreover the quotient $T^2/\widetilde{\mathcal{G}_1}$ can be written differently as follows:
$$T^2/\widetilde{\mathcal{G}_1}=T'^2/\widetilde{\mathcal{G}_1}',$$
with $T':=T/\left\langle t_{(a,0)}\right\rangle$ and $\widetilde{\mathcal{G}_1}'=\widetilde{\mathcal{G}_1}/\left\langle t_{(a,0)},t_{(0,a)}\right\rangle$. We have $\widetilde{\mathcal{G}_1}'\simeq I\times C_2^3\times \left\langle s_0'\right\rangle$, with $s_0'$ which exchanges the two factors of $T'^2$; then $f$ provides a bimeromorphism:
$$S'^2/\left\langle j_{\id}(C_2^3),\mathfrak{S}_2\right\rangle\dashrightarrow S^2/\left\langle j_{\id}(C_2),\mathfrak{S}_2\right\rangle,$$
where $S'$ is the Kummer surface obtained from $T'$ and $S$ the Kummer surface obtained from $T$. We conclude as usual with Proposition \ref{bimero}.
\item[(v)]
This case is very similar to (iv).
The next bimeromorphism is obtained as before via $f$ and a quotient torus $T'$ with an isomorphism: 
$$T'^2/\widetilde{\mathcal{G}_1}' \rightarrow T^2/\widetilde{\mathcal{G}_2}.$$
So we directly provide $T'$, $\widetilde{\mathcal{G}_1}'$ and $\widetilde{\mathcal{G}_2}$.
We assume that $T=E\times E$ with $E=\C/\left\langle 1,x\right\rangle$ an elliptic curve.
Then we consider $\widetilde{\mathcal{G}_2}=\left\langle (s_2,s_2),I,s_0,t_{(\frac{x}{2},\frac{x}{2},\frac{x}{2},\frac{x}{2})}\right\rangle$ and
$$\widetilde{\mathcal{G}_1}=\left(\left\{\left.t_{(a,a)}\right|\ a\in T[2]\right\}\rtimes\left\langle (s_2,s_2),I,t_{(\frac{x}{2},\frac{x}{2},0,0)},t_{(0,0,\frac{x}{2},\frac{x}{2})}\right\rangle \right)\times\left\langle s_0 \right\rangle.$$
Note that $f\circ t_{(\frac{x}{2},\frac{x}{2},0,0)}=f\circ t_{(0,0,\frac{x}{2},\frac{x}{2})}=t_{(\frac{x}{2},\frac{x}{2},\frac{x}{2},\frac{x}{2})}$.
Then we consider $T'=T/\left\langle t_{(\frac{x}{2},\frac{x}{2})}\right\rangle$ and $\widetilde{\mathcal{G}_1}'=\widetilde{\mathcal{G}_1}/\left\langle t_{(\frac{x}{2},\frac{x}{2},0,0)},t_{(0,0,\frac{x}{2},\frac{x}{2})}\right\rangle$.
It remains to verify that $\widetilde{\mathcal{G}_1}'$ is the relevant group. 
Let $g\in \widetilde{\mathcal{G}_1}$, we denote its image in $\widetilde{\mathcal{G}_1}'$ by $\overline{g}$. 
The group $\widetilde{\mathcal{G}_1}'$ is generated by $\overline{I}$, $\overline{s_0}$, $\overline{(s_2,s_2)}\circ \overline{t_{(\frac{1}{2},0,\frac{1}{2},0)}}$, $\overline{t_{(\frac{1}{2},0,\frac{1}{2},0)}}$ and $\overline{t_{(\frac{x}{2},0,\frac{x}{2},0)}}$.
Note that $\left(\overline{(s_2,s_2)}\circ \overline{t_{(\frac{1}{2},0,\frac{1}{2},0)}}\right)^2=\overline{(-\id,-\id)}\circ \overline{t_{(\frac{1}{2},\frac{1}{2},\frac{1}{2},\frac{1}{2})}}$.
The group $\left\langle s_2\circ t_{(\frac{1}{2},0)}, t_{(\frac{1}{2},0)}\right\rangle/-\id\simeq \mathcal{D}_4$ because
$$(s_2\circ t_{(\frac{1}{2},0)})^{-1}=-t_{(\frac{1}{2},0)}\circ s_2.$$
Moreover $\overline{t_{(\frac{x}{2},0,\frac{x}{2},0)}}$ commutes with $\overline{(s_2,s_2)}$ since we have quotiented $T$ by $t_{(\frac{x}{2},\frac{x}{2})}$. 
Therefore, we have:
$$\mathcal{G}_1'=\widetilde{\mathcal{G}_1}'/\overline{I}\simeq \mathcal{D}_4\times C_2\times \left\langle s_0\right\rangle.$$
\end{itemize}
\end{proof}
\begin{rmk}
Note that we still have three couples of orbifolds with the same Betti numbers and singularities. These couples are:
\begin{itemize}
\item
$S(C_2^2\wr C_2)_{\not\sim\id}^{[2]}$ and $S(\mathcal{D}_4)^{[2]}$;
\item
$S(C_2^2\rtimes C_4)^{[2]}$ and $S(C_4)^{[2]}$;
\item
$S(C_2\times \mathfrak{S}_4)^{[2]}$ and $S(\mathcal{D}_6)^{[2]}$.
\end{itemize}
These couples could be pairs of deformation equivalent orbifolds;
note that the previous method do not apply for these couples.
\end{rmk}
\begin{rmk}
Let $S$ be a K3 surface endowed with a finite admissible symplectic automorphism group $G$. 
All the results of Section \ref{sing}, all the methods described in this section and our computer programs \cite{github} still applied to $S(G)_{\theta}^{[2]}$ when the involution $\theta$ is not valid. Therefore, we could also provide all the examples of Fujiki primitively symplectic orbifolds obtained from a finite admissible symplectic automorphism group $G$ endowed with a non-valid involution.
\end{rmk}

\subsection{Some examples in dimension 6}\label{dim6section}
In  dimension higher than 4, the terminalizations are unknown. However, in few cases, it is still possible to conclude.
\begin{prop}\label{dim6bis}
The varieties $S(C_2^k)^{[3]}$ for $1\leq k \leq 4$ are irreducible symplectic orbifolds.
\end{prop}
\begin{proof}
We recall that $S(C_2^k)^{[3]}$ is obtained as a terminalization of the quotient 
$S^3/\mathcal{G}$, with
$$\mathcal{G}=\left\langle \left\{\left.(\iota,\iota,\id)\right|\ \iota\in C_2^k\right\}\cup\mathfrak{S}_3\right\rangle.$$
By Theorem \ref{quotientsvalide}, we know that $S(C_2^k)^{[3]}$ is an irreducible symplectic variety with a simply connected smooth locus.
Moreover by Propositions \ref{bimero} and \ref{locallytriv}, we know that if one terminalization of $S^3/\mathcal{G}$ is orbifold, then all the terminalizations are orbifold.
Therefore, it only remains to find one terminalization that makes $S(C_2^k)^{[3]}$ an orbifold. 

We denote $\pi:S^3\rightarrow S^3/\mathcal{G}$ the quotient map.
The group $\mathcal{G}$ has cardinality $6\times 2^{2k}$ according to Lemma \ref{card}. We denote by $O$ the subset of $S^3$ which contains the elements with an orbit under the action of $\mathcal{G}$ of a cardinality strictly smaller than $6\times 2^{2k}$. 
According to Proposition \ref{terminal}, to find the terminalization of $S^3/\mathcal{G}$, we have first to find the irreducible components of $O$ of dimension 4. We denote by $O_4$ the union of all these components.
According to Remark \ref{fixed}, these components are of the form:
$$\left\{\left.(x,\iota(x))\right|\ x\in S\right\}\times S\ \ \text{or}\ \ \left\{\left.(x,x)\right|\ x\in S\right\}\times S,$$
with $\iota\in C_2^k\smallsetminus{\id}$. The points of these components are fixed by a subgroup of $\mathcal{G}$. Since a point in $S$ can only be fixed by one unique symplectic involution, the different possibilities are the following (up to permutation of the factors).
\begin{itemize}
\item[(i)]
generic point of $O_4$; they are fixed by an automorphism of order 2: a transposition or the composition of a transposition with an involution.
\item[(ii)]
The points of the form $(x,x,y)$ with $x\in \Fix \iota$, $\iota\in C_2^k\smallsetminus{\id}$ and $y$ a generic point in $S$.
These points are fixed by $\left\langle (1,2),(\iota,\iota,\id)\right\rangle$.
\item[(iii)]
The points of the form $(x,\iota(x),\iota'(x))$ with $\iota$ and $\iota'$ two involutions in $C_2^k$ (possibly $\id$) and $x$ a generic point in $S$.
These points are fixed by $\left\langle (1,2)\circ(\iota,\iota,\id),(1,3)\circ(\iota',\id,\iota')\right\rangle$.
\item[(iv)]
The points of the form $(x,\iota'(x),y)$ with $x\in \Fix \iota$, $\iota\in C_2^k\smallsetminus{\id}$, $\iota'$ any involution (possibly $\id$) and $y\in \Fix \iota$ but $y$ is not in the orbit of $x$ under the action of $C_2^k$.
These points are fixed by $\left\langle (1,2)\circ(\iota',\iota',\id),(\iota,\iota,\id),(\iota,\id,\iota),(\id,\iota,\iota)\right\rangle$.
\item[(v)]
The points of the form $(x,\iota'(x),\iota''(x))$ with $x\in \Fix \iota$ and $\iota$, $\iota'$ and $\iota''$ three involutions in $C_2^k$ such that $\iota\neq\id$ (while $\iota'$ and $\iota''$ can possibly be any involution including $\id$).
These points are fixed by $\left\langle (1,2)\circ(\iota',\iota',\id),(\iota,\iota,\id),(\iota,\id,\iota),(\id,\iota,\iota), (1,3)\circ(\iota'',\id,\iota'')\right\rangle$.
\end{itemize} 
The image by $\pi$ of the previous points are of the following analytic types:
\begin{itemize}
\item[(i)]
the type $\left(\C^2\right)^3/(-\id,\id,\id)$; a blow-up resolves crepantly these singularities.
\item[(ii)]
The type $\left(\C^2\right)^3/\left\langle (-\id,-\id,\id),(1,2)\right\rangle$. This quotient corresponds to $\left(\C^2\right)^2/\left\langle (-\id,-\id),(1,2)\right\rangle\times \C^2$. Then, a crepant resolution is provided by Proposition \ref{mini}.
\item[(iii)]
The type $\left(\C^2\right)^3/\mathfrak{S}_3$. A crepant resolution is obtained by the Hilbert scheme of 3 points on $\C^2$.
\item[(iv)]
 Let $$B=\left\langle (-\id,-\id,\id),(1,2),(-\id,\id,-\id),(\id,-\id,-\id)\right\rangle.$$  In this case, the singular point is of analytic type $\left(\C^2\right)^3/B$. The group $A=\left\langle (-\id,-\id,\id),(1,2)\right\rangle$ is normal in $B$.
Then, we can first quotient by $A$.
By Proposition \ref{mini}, we have a crepant resolution $\widetilde{(\C^2)^3/A}\rightarrow \left(\C^2\right)^3/A$. Then the involutions $(-\id,\id,-\id)$ and $(\id,-\id,-\id)$ induce involutions on $\widetilde{(\C^2)^3/A}$ with a fixed locus in codimension 4. We keep the same notation for the induced involutions on $\widetilde{(\C^2)^3/A}$. Then $\left(\widetilde{(\C^2)^3/A}\right)/\left\langle (-\id,\id,-\id),(\id,-\id,-\id)\right\rangle$ has singularities in codimension 4. By Proposition \ref{terminal}, we have found a terminalization.
\item[(v)]
The case (v) is of analytic type $\left(\C^2\right)^3/\left\langle (-\id,-\id,\id),\mathfrak{S}_3\right\rangle$.
The sequel of the proof is dedicated to this case.
\end{itemize}
There is a natural embedding of $\left(\C^2\right)^3$ in $\left(\C^2\right)^4$ given by $(x,y,z)\mapsto (x,y,z,-x-y-z)$.
Therefore there is a natural action of $\mathfrak{S}_4$ on $\left(\C^2\right)^3$ via the previous embedding. For instance the permutation $(1,2)(3,4)$ acts on $\left(\C^2\right)^3$ by sending $(x,y,z)$ to $(y,x,-x-y-z)$. 
We are going to show that cases (v) and (vi) are analytically equivalent to $\left(\C^2\right)^3/\mathfrak{S}_4$. This will end the proof, since this quotient has a crepant resolution given by the Kummer resolution.
We consider the following change of variables:
$$\xymatrix@R0pt{f:\ \left(\C^2\right)^3\ar[r]&\left(\C^2\right)^3\\
(x,y,z)\ar[r]&(\frac{y+z-x}{2},\frac{x+z-y}{2},\frac{x+y-z}{2}).}$$
Via this change of variables the action of $\mathfrak{S}_3$ is unchanged. 
Moreover, we have: 
\begin{align*}
f\circ (-\id,-\id,\id)(x,y,z)&=(\frac{-y+z+x}{2},\frac{-x+z+y}{2},\frac{-x-y-z}{2})\\
&=(\frac{-y+z+x}{2},\frac{-x+z+y}{2},-\frac{-x+y+z}{2}-\frac{x-y+z}{2}-\frac{x+y-z}{2}).
\end{align*}
Therefore, the action of $(-\id,-\id,\id)$ is now given by the following matrix:
$\begin{pmatrix}
0 & 1 & -1\\
1 & 0 & -1\\
0 & 0 & -1
\end{pmatrix}.$
This corresponds to the action of $(1,2)(3,4)$. It follows that the action of the group $\left\langle (-\id,-\id,\id),\mathfrak{S}_3\right\rangle$ after the change of variables corresponds to the action of the group $\left\langle (1,2)(3,4),\mathfrak{S}_3\right\rangle=\mathfrak{S}_4$.
\end{proof}
\begin{rmk}
According to Corollary \ref{examples}, the previous proposition provides 6-dimensional irreducible symplectic orbifolds with second Betti number 15, 11, 9 and 8 respectively.
\end{rmk}
\newpage
\appendix
\section{Introduction}
The programs to compute the second Betti number and the singularities are straightforward application of Propositions \ref{b2}, \ref{a8}, \ref{a6}, \ref{a3} (see Section \ref{singpropor} and \ref{singProgram}). However, the methods to find the valid involutions of Lemma \ref{mainlemma2} have not yet been explained. The main goal of this section is to explain these methods in relation with our Python programs. In particular, we provide the Python commands in order to use the programs in \cite{github}. To help the understanding, we illustrate the explanations by studying some precise examples of groups.

All our programs are based on permutation groups. Therefore a first step is to find a permutation representation for each groups. Each group is given via a list of generators in a permutation group. To do so we use the two data bases \cite{Dokchitser} and \cite{LMFDB}.
\section{Finding the valid involutions: first method}\label{firstmethod}
\subsection{General description of the method}
This method is efficient when the group $G$ is with a trivial center and with $\Aut(G)$ small enough. The goal of this method is to see involution as conjugation by an element of order 2. 
We have the following exact sequence:
$$\xymatrix{0\ar[r]&\ar[r]Z(G)& G\ar[r]^{\Phi\ \ \ \ \ }& \Aut(G)\ar[r]&\Out(G)\ar[r]&0,}$$
with $\Phi(g)(h)=ghg^{-1}$. 
Assume that $Z(G)=\left\{\id\right\}$ and we can find a permutation representation such that:
$G\hookrightarrow \Aut(G)\hookrightarrow \mathfrak{S}_n$. Then all involutions on $G$ can be expressed by the conjugation by an element of $\mathfrak{S}_n$. Indeed, let $f\in \Aut(G)$, then $f\circ \Phi(g) \circ f^{-1}=\Phi(f(g))$. Based on these remarks, the program \cite[Finding\_involution1]{github} is constructed as follows.
\begin{itemize}
\item
First, we consider all the elements of order 2 in $\mathfrak{S}_n$ and we select the ones which correspond to valid involutions on $G$. 
In \cite[Finding\_involution1]{github}, this is given by the function:
\begin{lstlisting}
valid_involution(L,n)
\end{lstlisting}
\begin{itemize}
\item \textbf{input}: 
\begin{itemize}
\item
$L$: list of generators of $G$
\item
$n$: integer to specify that we are in $\mathfrak{S}_n$.
\end{itemize}
\item \textbf{output}: the list of all order 2-elements in $\mathfrak{S}_n$ that induces a valid involution on $G$.
\end{itemize}
\item
Second, we find the valid involutions which are equivalent using Proposition \ref{invodeform}.
This is given by the function:
\begin{lstlisting}
classes_valid_involution(L,N,P,n)
\end{lstlisting}
\begin{itemize}
\item \textbf{input}:
\begin{itemize}
\item
$L$: list of generators of $G$
\item
$N$: list of generators of $\widetilde{G}$ (see proposition \ref{invodeform})
\item
$P$: a list of valid involutions
\item
$n$: integer to specify that we are in $\mathfrak{S}_n$.
\end{itemize}
\item\textbf{output}:
a list obtained from $P$ by keeping only one representative of the equivalent classes of valid involutions. 
\end{itemize}
\end{itemize}
This first method can also be used for some groups with non-trivial center using the following trick.
\begin{prop}\label{AxG}
We assume that $G=A\times G'$ with $A$ a finite abelian group and $G'$ a finite group. We denote by $\inv_A$ the unique valid involution on $A$. Then all valid involutions on $G$ can be written $\inv_A\times \theta'$, where $\theta'$ is a valid involution on $G'$.
\end{prop}
\begin{proof}
To simplify the notation, we identify $A$ with $A\times \left\{\id\right\}$ and $G'$ with $\left\{\id\right\}\times G'$. 
Let $\theta$ be a valid involution on $G$. Therefore, there exists a family of generators $(f_1,...,f_k)$ such that $\theta(f_j)=f_j^{-1}$ for all $j\in\left\{1,...,k\right\}$. All this generators can be written $f_j=a_jf'_j$ with $a_j\in A$ and $f'_j\in G'$. Let $a\in A$; we can write $a$ as a product of elements of $(f_1,...,f_k)$:
$$a=\prod f_j=\prod a_j \prod f'_j=\prod a_j,$$
because, necessarily $\prod f'_j=\id$.
Moreover:
$$\theta(a)=\prod f_j^{-1}=\prod a_j^{-1} \prod f_j'^{-1}=\prod a_j^{-1}=a^{-1}.$$
Hence $\theta_{|A}=\inv_A$.
Necessarily, $(f_1',...,f_k')$ is a family of generators of $G'$.
Furthermore, we have:
$$a_j^{-1}f_j'^{-1}=f_j^{-1}=\theta(f_j)=\theta(a_j)\theta(f_j')=a_j^{-1}\theta(f_j').$$
That is $\theta(f_j')=f_j'^{-1}$. Hence $\theta_{|G'}$ is a valid involution.
\end{proof}
\subsection{Description on an example}\label{C2p4S3}
We illustrate the method on one example. 
We consider the group $\SmallGroup(96,227)$;
it is a semi-direct product of $C_2^4$ with $\mathfrak{S}_{3}$. 

According to the data base \cite{Dokchitser} and \cite{LMFDB}, we can embed $C_2^4\rtimes\mathfrak{S}_{3}$ in $\mathfrak{S}_{8}$ as follows. We set \textbf{C2p4S3} a list of generators.
\begin{lstlisting}
C2p4S3=[Permutation(0,7)(1,2), Permutation(7)(0,1,2)(4,5,6),
        Permutation(0,4)(1,6)(2,5)(3,7)]
\end{lstlisting} 
Note that $C_2^4\rtimes\mathfrak{S}_{3}$ has a trivial center.
Then with Python (for instance by using the function "The\_Automorphisms" of \cite[Finding\_involution1]{github}), we can verify that we have an embedding:
$$C_2^4\rtimes\mathfrak{S}_{3}\hookrightarrow\Aut(C_2^4\rtimes \mathfrak{S}_{3})\hookrightarrow \mathfrak{S}_{8}.$$
So, all the valid involutions on $C_2^4\rtimes\mathfrak{S}_{3}$ can be obtained with the method explained before.
First, we can run the program with $\widetilde{G}=G=C_2^4\rtimes\mathfrak{S}_{3}$.
\begin{lstlisting}
>>> P=valid_involution(C2p4S3,8)
>>> classes_valid_involution(C2p4S3,C2p4S3,P,8)
[Permutation(7), Permutation(0, 7)(5, 6), Permutation(0, 7)(4, 6), Permutation(7)(0, 1)(3, 5)]
\end{lstlisting}
We find 4 possible classes of valid involutions. To sharpen this result, we need to consider a bigger $\widetilde{G}$.
However, it can be hard to find the appropriated $\widetilde{G}$. For this reason, we have developed a function for this purpose.

\subsection{Method to find the group $\widetilde{G}$ of Proposition \ref{invodeform}}\label{Gtilde}
We start with $H\supset G$ an over-group of $G$ , $\theta_1$ and $\theta_2$ two valid involutions on $G$. The algorithm provides $\widetilde{G}$ such that:
\begin{itemize}
\item
$G\subset \widetilde{G} \subset H$;
\item
we can find $h_1$ and $h_2$ in $\widetilde{G}$ that verified the assumptions of Proposition \ref{invodeform} given $\theta_1$ and $\theta_2$ equivalent.
\end{itemize}
This is given by the function:
\begin{lstlisting}
finding_G_tilde(L,N,A,B,r)
\end{lstlisting}
\begin{itemize}
\item \textbf{input}:
\begin{itemize}
\item
$L$: list of generators of $G$
\item
$N$: list of generators of $H$
\item
$A$ and $B$: two valid involutions that we want to show equivalent
\item
$r$: the desired cardinality for $\widetilde{G}$.
\end{itemize}
\item\textbf{output}:
an element $g$ such that $\widetilde{G}=\left\langle g,G\right\rangle$. Moreover $A$ and $B$ are found equivalent by applying Proposition \ref{invodeform} with $\widetilde{G}$.
\end{itemize}

We can use this function to our previous example.
\begin{lstlisting}
>>> N=[Permutation(0,1),Permutation(0,1,2,3,4,5,6,7)]
>>> A=Permutation(0, 7)(5, 6)
>>> B=Permutation(0, 7)(4, 6)
>>> finding_G_tilde(C2p4S3,N,A,B,192)
Permutation(0, 4)(1, 6)(2, 3)(5, 7)
\end{lstlisting}
I.e. taking $\widetilde{G}=\left\langle G, (0, 4)(1, 6)(2, 3)(5, 7)\right\rangle$ allows to show that $(0, 7)(5, 6)$ and $(0, 7)(4, 6)$ are equivalent using Proposition \ref{invodeform}. It remains to verify that $\widetilde{G}$ is an automorphism group of a K3 surface. We can verify with Python that $\widetilde{G}$ has trivial center, then $\widetilde{G}$ can only be isomorphic to $H_{192}=C_2^4\rtimes \mathcal{D}_6$ according to \cite{Dokchitser}; moreover this group is a symplectic group of a K3 surface as explained in \cite[Section 2, n°8]{Mukai}. We apply the same method with the permutations $(0, 7)(4, 6)$ and $(0, 1)(3, 5)$ to show that they are equivalent. 
We obtain that there are at most two equivalent classes of valid involution on $C_2^4\rtimes\mathfrak{S}_{3}$ given by $\id$ and:
\begin{lstlisting}
Permutation(7)(0, 1)(3, 5)
\end{lstlisting} 
In Section \ref{singpropor}, we will see that these two involutions provide orbifolds with different second Betti numbers; so they cannot be equivalent. So, we have two classes of valid involutions as stated in Lemma \ref{mainlemma2}.
\subsection{Application of the method on our list of admissible groups}
\subsubsection*{Groups of the form $A\times G'$}
According to Proposition \ref{AxG}, we can use this first method to find the classes of valid involutions on the groups $G=\mathfrak{S}_3$, $C_3\times \mathfrak{S}_3$, $C_2\times \mathfrak{S}_4$, $\mathfrak{A}_4$, $C_2\times\mathfrak{A}_4$, $C_3\times \mathfrak{A}_4$, $C_2^2\times\mathfrak{A}_4$ or $\mathfrak{S}_4$. For all these groups, applying the function \emph{classes\_valid\_involution} with $\widetilde{G}=G$ is enough to conclude.
\subsubsection*{The groups $\mathfrak{S}_3^2$, $\mathfrak{S}_3\wr C_2$, $\mathfrak{A}_4^2$ and $C_2^4\rtimes C_6$}
These groups are with a trivial center and we can find an embedding 
$G\hookrightarrow \Aut(G)\hookrightarrow \mathfrak{S}_n$. We provide these embeddings in the next table by giving a family of generators for $\Aut(G)$ (the generators of $G$ are given in Lemma \ref{mainlemma2}).
These embeddings can be found by hand or using the data bases \cite{Dokchitser} and \cite{LMFDB}.
Moreover, for these groups applying our method with $\widetilde{G}=G$ is enough to conclude.
\[
\begin{tabular}{|c|c|c|c|}
\hline
$G$ & n & $\Aut(G)$ & embbeding of $\Aut(G)$ in $\mathfrak{S}_n$\\
\hline
$\mathfrak{S}_3^2$ & 6 & $\mathfrak{S}_3\wr C_2$ & $(0,1)$; $(0,1,2)$; $(3,4)$; $(3,4,5)$; $(0,3)(1,4)(2,5)$\\
\hline
$\mathfrak{S}_3\wr C_2$ & 9 &$\Agam_1(\mathbb{F}_9)$ &$(0,5,3,4,1,2,7,6)$; $(0,1)(2,4)(5,6)$; \\
& & &$(0,1,8)(2,3,4)(5,6,7)$; $(0,3,6)(1,4,7)(2,5,8)$\\ 
\hline
$\mathfrak{A}_4^2$ & 8 &$\mathfrak{S}_4\wr C_2$ & $(0,1)$; $(0,1,2,3)$; $(4,5)$; $(4,5,6,7)$; $(0,4)(1,5)(2,6)(3,7)$\\
\hline
\end{tabular}
\]
For $C_2^4\rtimes C_6$, there is a small complication.
In order to have $C_2^4\rtimes C_6\hookrightarrow \Aut(C_2^4\rtimes C_6)\hookrightarrow \mathfrak{S}_n$,
we need to consider a different embedding from the one provided in Lemma \ref{mainlemma2}.
We need to take $n=12$.
\begin{itemize}
\item
\textbf{Embedding of $C_2^4\rtimes C_6$:} $(0,2,4,6,8,10)(1,3,5,7,9,11)$, $(0,5,11,6)(1,7,2,8)(3,9)(4,10)$.
\item
\textbf{Embedding of $\Aut(C_2^4\rtimes C_6)$:} 

$(0,11)(1,2)$, $(0,2,4,6,8,10)(1,3,5,7,9,11)$, $(0,10)(1,7)(2,8)(3,5)(4,6)(9,11)$.
\end{itemize}
Once we know that there is only one class of valid involutions, we can give a representative using the simplest embedding given in Lemma \ref{mainlemma2}.
\subsubsection*{The group $\mathfrak{A}_{3,3}$}
The group $\Aut(\mathfrak{A}_{3,3})$ is given by $\AGL_2(\mathbb{F}_3)$. According to  \cite{Dokchitser} and \cite{LMFDB}, we can find an embedding 
$$\mathfrak{A}_{3,3}\hookrightarrow \Aut(\mathfrak{A}_{3,3}) \hookrightarrow \mathfrak{S}_{9},$$
obtained with the following generators (the generators of $\mathfrak{A}_{3,3}$ are given in Lemma \ref{mainlemma2}):
				$$\Aut(\mathfrak{A}_{3,3})=\left\langle(0,5,3,4,1,2,7,6); (0,1,8)(2,3,4)(5,6,7); (2,3,4)(5,7,6); (0,3,6)(1,4,7)(2,5,8)\right\rangle.$$

If we apply Proposition \ref{invodeform} with $\widetilde{G}=G=\mathfrak{A}_{3,3}$, we find 7 possible classes of valid involutions. Hence, we have to use the function \emph{finding\_G\_tilde} of \cite[Finding\_involution1]{github} as explained in Section \ref{Gtilde}.

Let $(\theta_j)_{j\in\left\{1,...,7\right\}}$ be the 7 valid involutions that we find before. With the function \emph{finding\_G\_tilde} we can find 5 embeddings for $2\leq j\leq 6$:
$$\mathfrak{A}_{3,3}\hookrightarrow H_j \hookrightarrow \Aut(\mathfrak{A}_{3,3}),$$
such that $\left|H_j\right|=36$ and applying Proposition \ref{invodeform} with $\widetilde{G}=H_j$ provides $\theta_j\sim \theta_{j+1}$. 
Necessarily $H_j$ can only be $\mathfrak{S}_3^2$ or $C_3^2\rtimes C_4$ which are both automorphism groups of a K3 surface (the only groups of order 36 containing $\mathfrak{A}_{3,3}$ according to \cite{Dokchitser}).

After this process, it remains potential two classes of valid involutions. We can see that they are not equivalent by computing the singularities of the associated orbifolds (see Theorem \ref{main3}). 
\section{Finding the valid involutions: second method}\label{secondmethod}
When $G$ has a non-trivial center, or when $\Aut(G)$ is big, or if we do not know an embedded of $\Aut(G)$ in a permutation group, the previous method is not effective. In these cases, we need another strategy.
\subsection{General description of the method}
A valid involution can also be seen via a set of generators. Indeed if $\theta$ is a valid involution, there exists a set $\left\{g_1,...,g_k\right\}$ of generators such that $\theta(g_i)=g_i^{-1}$ for all $i\in\left\{1,...,k\right\}$. Therefore the involution $\theta$ is fully defined by the set $\left\{g_1,...,g_k\right\}$. The program \cite[Finding\_involution2]{github} is based on this remark. In particular, it uses a direct consequence of the Tarski irredundant Basis Theorem.
\begin{defi}
Let $G$ be a finite group. Let $\mathcal{B}\subset G$ be a subset. We say that $\mathcal{B}$ is a basis of $G$ if:
\begin{itemize}
\item[(i)]
$\left\langle\mathcal{B} \right\rangle=G$ and;
\item[(ii)]
 none of the proper subsets of $\mathcal{B}$ generate $G$.
\end{itemize}
To shorten, a basis is a minimal set of generators.
\end{defi}
\begin{thm}[Tarski irredundant Basis Theorem]\label{Tarski}
~\\
Let $G$ be a finite group. Let $d(G)$ and $m(G)$ be respectively the the smallest and the largest cardinality for a basis of $G$.
Then for all $d(G)\leq k \leq m(G)$, there exists a basis of $G$ of cardinality $k$.
\end{thm}
The program \cite[Finding\_involution2]{github} follows several steps:
\begin{itemize}
\item
We find all the family of generators of the group $G$ using Theorem \ref{Tarski}. This is the function 
\begin{lstlisting}
generators(L)
\end{lstlisting}
\begin{itemize}
\item \textbf{input}: $L$ : a list of generators of $G$.
\item \textbf{output}: the list of all families of generators of $G$.
\end{itemize}
\item
Among these families of generators we select the ones that provide a well defined involution on $G$. We also eliminate the redundancy: if two family of generators provide the same involution, we consider only one of them.
There are the functions 
\begin{lstlisting}
bijection(L)
involutions(L)
\end{lstlisting}
\begin{itemize}
\item \textbf{input}: $L$ : a list of generators of $G$.
\item \textbf{output of \emph{bijection}}: the list of all bijections (given by families of generators).
\item \textbf{output of \emph{involutions}}: the list of all involutions (given by families of generators).
\end{itemize}
\item
Finally, we find the valid involutions which are equivalent with exactly the same idea as before using Proposition \ref{invodeform}.
\begin{lstlisting}
classes_valid_involution2(P,L,N)
\end{lstlisting}
\begin{itemize}
\item \textbf{input}: 
\begin{itemize}
\item $P$: list of involutions (given by families of generators of $G$);
\item $L$: a list of generators of $G$;
\item $N$: a list of generators of $\widetilde{G}$.
\end{itemize}
\item \textbf{output}: a list obtained from $P$ by keeping only one representative of the equivalent classes of valid involutions.
\end{itemize}
\end{itemize}
\subsection{Description on an example}
We consider the group $C_2^2\rtimes C_4$. Using the data bases \cite{Dokchitser} and \cite{LMFDB}, we set \textbf{C2p2C4} a family of generators: 
\begin{lstlisting}
C2p2C4=[Permutation(0,1,2,7)(3,4,5,6),Permutation(7)(0,4)(2,6)]
\end{lstlisting}
Then:
\begin{lstlisting}
>>> P=involutions(C2p2C4)
>>> P
[[Permutation(0, 1, 2, 7)(3, 4, 5, 6), Permutation(7)(0, 4)(2, 6)], [Permutation(0, 3, 6, 1)(2, 5, 4, 7), Permutation(7)(0, 4)(2, 6)]]
\end{lstlisting}
We find two different valid involutions $\theta_1$ and $\theta_2$ each given by a family of generators. Then, we apply the function \emph{classes\_valid\_involution2} with the group $\widetilde{G}=G$ (see Proposition \ref{invodeform}). 
\begin{lstlisting}
>>> classes_valid_involution2(P,C2p2C4,C2p2C4)
[[Permutation(0, 1, 2, 7)(3, 4, 5, 6), Permutation(7)(0, 4)(2, 6)]]
\end{lstlisting}
We obtain only one class of valid involution. 
Since, we know that there is only one class of valid involutions, we can still apply the first method (Section \ref{firstmethod}) to find an expression of a valid involution as provided in Lemma \ref{mainlemma2}.
\subsection{Application of the method on our list of admissible groups}
\subsubsection*{The groups $\mathcal{D}_4$, $\mathcal{D}_6$, $C_2\times \mathcal{D}_4$, $C_2^3\rtimes C_4$ or $C_3^2\rtimes C_4$}
For all these groups, we use the function \emph{classes\_valid\_involution2} with $\widetilde{G}=G$ and we find only one class of valid involutions.
\subsubsection*{The group $C_2^2\wr C_2$}
For this group applying Proposition \ref{invodeform} with $\widetilde{G}=G$ is not enough to conclude. We need to consider $\widetilde{G}=C_2^4\rtimes C_6$ which is one of our admissible automorphism groups (see Lemma \ref{mainlemma2}). To apply \emph{classes\_valid\_involution2}, we need to know an embedding $C_2^2\wr C_2\hookrightarrow C_2^4\rtimes C_6$ (see Lemma \ref{mainlemma2} for the generators of $C_2^4\rtimes C_6$):
$$C_2^2\wr C_2=\left\langle (0,2)(1,7)(3,5)(4,6); (3,4)(5,6); (3,5)(4,6); (0,7)(1,2)(3,4)(5,6); (0,4)(1,5)(2,6)(3,7)\right\rangle.$$
\subsubsection*{The group $C_4^2\rtimes C_3$}
In this case, if we try to apply Proposition \ref{invodeform} with $\widetilde{G}=G$, we find several valid involutions that could still be equivalent.
Therefore, we need to consider a bigger $\widetilde{G}$. 
We are going to take for $\widetilde{G}$
the group $\widetilde{F}$ defined in \cite[Section 2, n°5]{Mukai} acting on the Fermat quartic and containing $C_4^2\rtimes C_3$. This group is generated by the automorphisms of order 4: $(x,y,z,t)\mapsto (i^ax,i^by,i^cz,i^dt)$, with $a,b,c,d \in C_4$ and by the permutations of the coordinates. This group can be embedded in $\mathfrak{S}_{64}$ via its natural action on the set:$$\left\{\left.(i^a:i^b:i^c:1)\right|\ (a,b,c)\in C_4^3\right\}.$$
This embedding is realized in Python by the function ``F'' in \cite[Finding\_involution2]{github}. We proceed as follows.
We order the set $E$ by numbering in base 4 with 1 which corresponds to 0, i to 1, -1 to 2 and -i to 3. For instance $(1:i:-i:1)$ is the number 310 in base 4. Then, the Python function provides the following generators of $\widetilde{F}$ as permutations in $\mathfrak{S}_{64}$:
$$u:(x,y,z,t)\mapsto (ix,y,z,t);\ \ v:(x,y,z,t)\mapsto (x,iy,z,t);\ \ w:(x,y,z,t)\mapsto (x,y,iz,t);$$
$$ t1:(x,y,z,t)\mapsto (y,x,z,t)\ \ t2:(x,y,z,t)\mapsto (x,z,y,t)\ \ t3:(x,y,z,t)\mapsto (x,y,t,z).$$
For instance, the morphism $u$ corresponds to the permutation: $$\prod_{0}^{15}(i,i+1,i+2,i+3).$$
Taking $\widetilde{G}=\widetilde{F}$, we find only one class of valid involution.
\subsubsection*{The group $C_2^2\rtimes \mathfrak{A}_4$}
If we apply Proposition \ref{invodeform} taking $\widetilde{G}=G$, we obtain 10 valid involutions $(\theta_{i})$.
It is hard to find a effective $\widetilde{G}$ by hand. So we use the function \emph{finding\_G\_tilde} described in Section \ref{Gtilde} (the function \emph{finding\_G\_tilde} can be implemented with the two methods and can also be found in \cite[Finding\_involution2]{github}).

We choose $H=(C_2^2\rtimes \mathfrak{A}_{4})\rtimes\mathfrak{A}_{5}$ which can be embedded in $\mathfrak{S}_{16}$ (we choose this group because it is easy to construct and big enough (of index 2 in $\Aut(C_2^2\rtimes \mathfrak{A}_{4})$)). The embedding $C_2^2\rtimes \mathfrak{A}_4\hookrightarrow (C_2^2\rtimes \mathfrak{A}_{4})\rtimes\mathfrak{A}_{5} \hookrightarrow \mathfrak{S}_{16}$ can be provided by the following generators:
\begin{align*}
&C_2^2\rtimes \mathfrak{A}_{4}=\left\langle (0,8)(1,9)(2,10)(3,11)(4,14)(5,15)(6,12)(7,13);\right.\\ 
&\left.(0,11,14)(1,9,13)(2,8,15)(3,10,12)(4,6,5); (0,10)(1,11)(2,8)(3,9)(4,12)(5,13)(6,14)(7,15)\right\rangle;
\end{align*}
\begin{align*}
&(C_2^2\rtimes \mathfrak{A}_{4})\rtimes\mathfrak{A}_{5}=\left\langle C_2^2\rtimes \mathfrak{A}_{4}; (0,1)(11,9)(14,13)(2,3)(8,10)(15,12);\right.\\
&\left.(0,2)(11,8)(14,15)(1,3)(9,10)(13,12); (0,4)(11,6)(14,5)(2,13)(9,15)(1,8)\right\rangle.
\end{align*}
Then, we apply the function \emph{finding\_G\_tilde} of \cite[Finding\_involution2]{github} to find 9 embeddings:
$$ C_2^2\rtimes \mathfrak{A}_{4}\hookrightarrow \widetilde{G}_i \hookrightarrow (C_2^2\rtimes \mathfrak{A}_{4})\rtimes\mathfrak{A}_{5},$$
with $\widetilde{G}_i\simeq C_2^4\rtimes C_6$.
Each of these embeddings is chosen such that applying Proposition \ref{invodeform} with $\widetilde{G}_i$ shows that $\theta_{1}$ is equivalent to $\theta_{i}$. So finally, we obtain only one class of valid involution. 
\section{Second Betti numbers}\label{singpropor}
To find the second Betti number, we use the function \emph{div} of the program \cite[Second Betti number]{github} based on Proposition \ref{b2}. 
\begin{lstlisting}
div(L,T)
\end{lstlisting}
\begin{itemize}
\item \textbf{input}: 
\begin{itemize}
\item $L$: a list of generators of $G$;
\item $T$: an involution.
\end{itemize}
\item \textbf{output}: the number of exceptional divisors of $S(G)_\theta^{[2]}\rightarrow S\times S/\left\langle j_{\theta}(G),\mathfrak{S}_2\right\rangle$ which is the factor $\#\left(F/G\right)$ in Proposition \ref{b2} (see also Section \ref{Fujikivar} for the notation).
\end{itemize}
To finish the computation of $b_2$, we add $\rk H^2(S,\Z)^G$ which can be found in \cite{xiao}.
For instance, if we use the function with the group $C_2^4\rtimes\mathfrak{S}_{3}$ studied in Section \ref{C2p4S3}:
\begin{lstlisting}
>>> T=Permutation(7)(0, 1)(3, 5)
>>> div(C2p4S3,T)
2
\end{lstlisting}
In \cite{xiao}, we find that $\rk H^2(S,\Z)^{C_2^4\rtimes\mathfrak{S}_{3}}=4$. Hence, we obtain:
$$b_2\left(S(C_2^4\rtimes\mathfrak{S}_{3})_T^{[2]}\right)=2+4=6.$$
If we use the other valid involution that we found.
\begin{lstlisting}
>>> Id=Permutation(7)
>>> div(C2p4S3,Id)
6
\end{lstlisting}
So, we obtain:
$$b_2\left(S(C_2^4\rtimes\mathfrak{S}_{3})_{\id}^{[2]}\right)=6+4=10.$$
\section{Singularities}\label{singProgram}
To find the singularities we use the function \emph{singularities} of the program \cite[Singularities]{github} based on Propositions \ref{b2}, \ref{a8}, \ref{a6}, \ref{a3}. 
\begin{lstlisting}
singularities(L,T)
\end{lstlisting}
\begin{itemize}
\item \textbf{input}: 
\begin{itemize}
\item $L$: a list of generators of $G$;
\item $T$: an involution.
\end{itemize}
\item \textbf{output}: the table $\left[a_2,a_3,a_4,a_6,a_8,\mathfrak{b}_4,\mathfrak{b}_6\right]$. These numbers are defined in Notation \ref{defsing}.
\end{itemize}
For instance with the group $C_2^4\rtimes\mathfrak{S}_{3}$ studied in Section \ref{C2p4S3}:
\begin{lstlisting}
>>> T=Permutation(7)(0, 1)(3, 5)
>>> singularities(C2p4S3,T)
[19,12,6,0,0,1,0]
\end{lstlisting}

\bibliographystyle{amssort}

\noindent
Gr\'egoire \textsc{Menet}

\noindent
Académie militaire de Saint-Cyr Coëtquidan

\noindent
 CReC Saint-Cyr (Centre de recherche de l'académie militaire de Saint-Cyr) 

\noindent 
 56380 Guer, France.

\noindent
{\tt gregoire.menet@ac-amiens.fr}

\end{document}